\documentclass[a4paper,11pt,oneside]{amsart}
\usepackage{graphicx}
\usepackage{geometry}

\usepackage{amsmath}
\usepackage{amssymb}
\usepackage{amsthm}
\usepackage{hyperref}
\numberwithin{equation}{section}
\usepackage{graphicx}
\usepackage{float}
\usepackage[section]{placeins}
\newtheorem{theorem}{Theorem}[section]
\newtheorem{lemma}{Lemma}[section]

\newtheorem{definition}{Definition}[section]
\newtheorem{remark}{Remark}[section]
\newtheorem{assume}{Assumption}[section]

\hypersetup{hidelinks} 
\usepackage{dsfont}
\usepackage{enumerate}
\usepackage{color}
\usepackage{booktabs}

\title{ Optimal Dividends under Markov-Modulated \\Bankruptcy Level}
\date{\today}
\author{Giorgio Ferrari}
\address{G. Ferrari: Center for Mathematical Economics (IMW), Bielefeld University, Universit{\"a}tsstrasse 25, 33615, Bielefeld, Germany}
\email{giorgio.ferrari@uni-bielefeld.de}

\author{Patrick Schuhmann}
\address{P. Schuhmann: Center for Mathematical Economics (IMW), Bielefeld University, Universit{\"a}tsstrasse 25, 33615, Bielefeld, Germany}
\email{patrick.schuhmann@uni-bielefeld.de}

\author{Shihao Zhu}
\address{S. Zhu: Center for Mathematical Economics (IMW), Bielefeld University, Universit{\"a}tsstrasse 25, 33615, Bielefeld, Germany}
\email{shihao.zhu@uni-bielefeld.de}
\begin{document}
\newpage
\maketitle

\begin{abstract}
This paper proposes and studies an optimal dividend problem in which a two-state regime-switching environment affects the dynamics of the company's cash surplus and, as a novel feature, also the bankruptcy level. The aim is to maximize the total expected profits from dividends until bankruptcy. The company's optimal dividend payout is therefore influenced by four factors simultaneously: Brownian fluctuations in the cash surplus, as well as regime changes in drift, volatility and bankruptcy levels. In particular, the average profitability can assume different signs in the two regimes. We find a rich structure of the optimal strategy, which, depending on the interaction of the model's parameters, can be either of \emph{barrier-type} or of \emph{liquidation-barrier type}. Furthermore, we  provide explicit expressions of the optimal policies and value functions. Finally, we complement our theoretical results by a detailed numerical study, where also a thorough analysis of the sensitivities of the optimal dividend policy with respect to the problem's parameters is performed.

\end{abstract}

\vspace{2mm}


\vspace{2mm}

\vspace{2mm}

\vspace{2mm}

\section{Introduction}
The classical definition of \textit{ruin time} assumes that the bankruptcy level of a company is constant (usually set to $0$) and a single downcrossing of this level by the surplus process leads to bankruptcy. Since such a modeling is far too unrealistic in many practical situations, several modifications of ruin time have been explored in the literature, being the main idea to distinguish between ruin (negative surplus) and bankruptcy (going out of business). For example, absolute ruin arises when the premia received do not suffice to make the interest payments on debt (see, e.g., \cite{gerber2007absolute} and \cite{luo2011absolute}); Parisian ruin has been considered by introducing a grace period before liquidation (see \cite{dassios2008parisian} and \cite{loeffen2013parisian}, among many others); a so-called Omega model has been proposed by assuming that there is a bankruptcy rate function $\omega(x)$ describing the probability of bankruptcy when the company's  surplus $x$ is negative (see, e.g., \cite{albrecher2011optimal} and \cite{gerber2012omega}); \cite{vierkotter2017optimal} introduces penalty payments when the surplus becomes negative. Very recently, \cite{li2020liquidation} and \cite{wang2021optimality}  proposed a model based on the U.S. bankruptcy code (liquidation process in Chapter 7 and reorganization process in Chapter 11) and set liquidation, rehabilitation, and solvency barriers.  Furthermore, there are research works considering stochastic barriers in risk theory, although the focus lies mainly on analyzing the first passage time density of a Brownian motion through a stochastic boundary (see, e.g., \cite{che2013stochastic}, \cite{guillaume2021closed} and \cite{park1983stochastic}).

The empirical study by \cite{bernstein2019asset} shows that while there are uniform criteria by which a reorganization process may be converted into a liquidation procedure, there is significant variation in the interpretation of these criteria. Furthermore, \cite{antill2019optimal} also points out that while debtors may in principle choose to convert to liquidation process, and creditors may petition for such a conversion, the ultimate decision lies with the judge. These findings suggest that in reality several external factors can actually interfere with the liquidation procedure of a company.

In this paper we propose and study an optimal dividend problem in which the bankruptcy level as well the surplus process of a company is modulated by a two-state continuous-time Markov chain. In the spirit of the seminal \cite{hamilton1989new}, the Markov chain models macroeconomic conditions that inevitably affect a company's solvency status and therefore its bankruptcy level. For example, in a crisis (like the recent Covid-epidemics) companies can be supported by governments through the postponement of taxes' payments, which then effectively delay the arise of bankruptcy.

 Because the ruin probability approach neglects the time value of money and supposes that the surplus tends to infinity, maximizing expected discounted dividends until bankruptcy has been a popular and well-investigated topic in actuarial science, whose roots date back to the work of \cite{de1957impostazione}.  We refer to \cite{albrecher2009optimality}, \cite{avanzi2009strategies}, \cite{schmidli2007stochastic} and references therein for comprehensive surveys on developments in optimal dividends and related methodology. Optimal dividend problems with regime-switching have attracted substantial interest in the literature. We refer the reader to, e.g., \cite{zhu2013dividend} for a model where the cash surplus evolves as a general diffusion;  \cite{jiang2015optimal}, \cite{jiang2019optimal} and  \cite{zhu2014singular} for a  jump-diffusive setting; \cite{wei2010classical} and \cite{wei2010optimal} for models involving reinsurance contacts;  \cite{zhu2016optimal} for the allowance of capital injection; \cite{akyildirim2014optimal} and \cite{brinker2021dividend} for regime-dependent preference rate; to \cite{bandini2022optimal} for a model with a stochastic diffusive discount rate. Particularly relevant for our study are the works \cite{sotomayor2011classical} and \cite{jiang2012optimal}, where the optimal dividend strategy was studied within a Brownian model.
 The authors of \cite{sotomayor2011classical} addressed two different cases: bounded dividend rates and unbounded dividend rates. Meanwhile, via a ``guess and verify" approach they obtained the analytical expressions of the optimal dividend policy, which depends on the regime of the economy. On the other hand, \cite{jiang2012optimal} tackled a similar model using a different approach. They constructed the candidate value function by employing the dynamic programming equation and presented a contraction algorithm to compute the optimal threshold levels. It is worth pointing out that besides studying the classical case where the drift is positive in both regimes, in \cite{jiang2012optimal} the surplus' drift is also allowed to be negative in one regime. The latter modeling assumption significantly impacts the company's optimal dividend strategy, which is in fact shown in  \cite{jiang2012optimal} to exhibit a band structure, that otherwise frequently appears in jump models and was first introduced by \cite{gerber1969entscheidungskriterien}.

As in \cite{jiang2012optimal}, also in this paper, we allow the drift to be negative in one regime. Furthermore, as a novel feature, the bankruptcy level depends on the underlying Markov chain. Therefore, the company's optimal dividend strategy will be influenced by several factors together: Brownian fluctuations of the cash surplus, and regime changes yielding different signs of drifts in the two regimes, different volatilities and, moreover, different bankruptcy levels. To the best of our knowledge, this is the first paper that catches these features simultaneously. 

 We solve the problem by following a classical ``guess and verify" approach and provide expressions for both the value function and the optimal control in four different cases. The Hamilton-Jacobi-Bellman (HJB) equation associated with the optimal dividend problem takes the form of a system of two coupled variational inequalities. The coupling is through the transition rates of the Markov chain $\epsilon$, and there are regime-dependent boundary conditions due to regime-dependent bankruptcy levels. These aspects make the problem of finding an explicit solution quite involved, and the structure of the optimal solution quite rich. We guess that, when we fix the negative drift in regime one and vary $\mu_2 \in \mathbb{R}$ in regime two, the structure of the optimal dividend strategies strongly depends on the size of $\mu_{2}$. Specifically, when $\mu_2$ is sufficiently low, the optimal dividend strategy should be of \emph{barrier type}, while it changes to a \emph{liquidation-barrier type} (band structure) when $\mu_2$ is higher. A similar behavior was also observed in \cite{reppen2020optimal}, where an optimal dividend model with random diffusive profitability has been studied. 
 
Such an ``Ansatz'' is confirmed by our numerical analysis where, by varying $\mu_2$, we identify four different cases. The numerical exercise developed in Section 5 below hinges on the detailed analytical study of Section 4.  Therein we provide a set of sufficient conditions guaranteeing that appropriate solution to the HJB equation identifies with the value function. It is worth noticing that those sufficient conditions can be thought of as ``minimal'' ones, in the sense that they involve only the solution to a system of highly nonlinear equations, the model's parameters, and the explicitly constructed candidate value function. Furthermore, it is an easy numerical task to check their validity, as we in fact do in Section 5. The main result of Section 4 consists in the nontrivial verification of the optimality of the constructed candidate value functions, which is accomplished by combining probabilistic and analytic techniques (see, e.g., proofs in Sections A.3 and A.4). For example, a suitable application of the weak maximum principle yields the concavity and convexity of the (candidate) value function, which is a key step in subsequent verification (see again the proofs in Sections A.3 and A.4).

The rich structure of the optimal policy seem to be numerically robust with respect to changes of the model's parameters and leads to interesting economic conclusions. As an example, we find that although the surplus process has negative drift in both regimes, the different bankruptcy levels induce the company to continue the business in regime two in order to strategically exploit a change to regime one, which is associated to a lower bankruptcy level (cf. Case (B) in Section 4 and Section 5.2).

In conclusion, we believe that our main contributions are the following. Firstly, we introduce and study a new model, in which the bankruptcy level is random and depends on the underlying business conditions, thus generalizing \cite{jiang2012optimal} and \cite{sotomayor2011classical}. Moreover, through a detailed numerical analysis in a case study we provide some interesting novel economic implications of the optimal policy. Finally, from a mathematical point of view, we believe that our study nicely complements the existing literature on singular control with regime-switching (see, e.g., \cite{ferrari2018optimal} and references therein). Indeed, we identify minimal sufficient conditions for the optimality of a qualitatively rich candidate optimal dividend policy and we confirm the veridicity of our educated guess through a nontrivial and technical verification step.

The rest of the paper is organized as follows. In Section 2, we set up the model. In Section 3, we introduce the associated HJB equation, and we prove a verification theorem. In Section 4, we obtain the optimal dividend policy in four cases. Section 5 presents a detailed numerical study, where we examine the impact of different bankruptcy levels and provide the sensitivity analysis of the optimal solution with respect to the model's parameters. Section 6 concludes. Appendix A collects the proofs of some results of Section 3 and Section 4, whereas Appendix B contains the auxiliary results needed in the paper.

\section{Model formulation}
Consider a complete probability space $(\Omega,\mathcal{F},\mathbb{P})$ on which it is defined a standard Brownian motion $W:=\{W_t,t\geq0\}$ and an observable continuous-time Markov chain $\epsilon:=\{\epsilon_t, t\geq0\}$ with finite state $\mathcal{S}=\{1,2\}$. Here,  $\epsilon_t$  represents the regime of the economy at time $t$. We assume that $W$ and $\epsilon$ are independent, and that the Markov chain $\epsilon$ has a transition intensity matrix $\mathbb{Q} := (q_{ij})_{i,j \in \mathcal{S}}$, with transition rates such that $\lambda_{i}:=-q_{ii} >0$ and $\sum_{j \in \mathcal{S}}q_{ij}=0$ for every $i \in \mathcal{S}$. We denote by $\mathbb{F}:= \{\mathcal{F}_t, t\geq 0\}$ the filtration jointly generated by $W$ and $\epsilon$, as usual augmented by  $\mathbb{P}$-null sets of $\mathcal{F}$.

 Let the adapted process $X:=\{X_t^D,t\geq0\}$ represent the cash surplus of the company. We assume that $X$ satisfies
\begin{equation}\label{stateprocess}
\textrm{d}X_t^D = \mu_{\epsilon_{t}}\textrm{d}t + \sigma_{\epsilon_{t}}\textrm{d}W_t -\textrm{d}D_t,
\end{equation}
with initial level of the cash surplus $X_0^D = x \in \mathbb{R}$ and initial regime $\epsilon_0 = i \in \mathcal{S}.$ For every state $i \in \mathcal{S},$ both drift parameter $\mu_i \in \mathbb{R}$ and volatility parameter $\sigma_i>0$ are assumed to be known constants. The adapted process $D:=\{D_t, t\geq 0\}$ represents the cumulative amount of dividends paid from time zero up to $t$. For future frequent use, we define $\mathbb{P}_{x,i}[\cdot]:=\mathbb{P}[\cdot| X^D_0=x,\epsilon_0=i]$ and denote by $\mathbb{E}_{x,i}[\cdot]$ the corresponding expectation operator.

If the regime of the economy at time $t$ is $\epsilon_t=i\in\mathcal{S}$, we assume that the company is considered bankrupt as soon as the cash surplus hits the critical constant level $\theta_i \in \mathbb{R}$, which is in the following referred to as the bankruptcy level. We therefore define the bankruptcy time $\tau$ as
\begin{equation}\label{ruintime}
\tau := \inf\{t \geq0 : X^D_t \leq \theta_{\epsilon_t} \}, \quad  \mathbb{P}_{x,i}\text{-}a.s.,
\end{equation}
(with the usual convention $\inf \emptyset = +\infty$) and we introduce the set of admissible strategies $\mathcal{A}(x,i), (x,i) \in \mathbb{R}\times \mathcal{S}$, as it follows.

\begin{definition}\label{admissiblecontrol}
Let $(x,i) \in \mathbb{R} \times \mathcal{S}$ be given and fixed. An $\mathbb{F}$-adapted, non-decreasing  process $D$ is called an \textbf{admissible strategy} for $(x,i)$, and we write $D \in \mathcal{A}(x,i)$, if it satisfies the following conditions: 
\begin{enumerate}[(i)]
	\item $D: \Omega \times \mathbb{R_+} \to \mathbb{R_+}$, $D_0=0$, with sample paths that are left-continuous with right limits;
	\item $D_{t+}-D_t \leq X^D_t {-\theta_{\epsilon_t}} $  for all  $ t \geq 0, \mathbb{P}_{x,i}$-a.s.;
	\item $\textrm{d} D_t=0$ for $t >\tau, \mathbb{P}_{x,i}$-a.s., where $\tau$ is the bankruptcy time defined in (\ref{ruintime}). 
\end{enumerate}

\end{definition}

{In Definition 2.1, and in the rest of this paper, we shall use that the random measure $dD$ induced by any $D \in \mathcal{A}(x,i)$ admits the decomposition $dD_t =dD^c_t +D_{t+}-D_t, t\geq 0$, where $D^c$ denotes the continuous part of $D$}.

 Given the initial cash surplus $x$ and economic regime $i$, the expected discounted value of all the dividends accrued until ruin following the strategy $D \in \mathcal{A}(x,i)$ is
\begin{equation*}
J(x,i;D) := \mathbb{E}_{x,i} \bigg[ \int_{0}^{\tau}\textrm{ e}^{-\rho t}\textrm{d}D_t       \bigg],
\end{equation*}
where the discount factor $\rho>0$ is constant.

The aim of this paper is to find an admissible strategy $D^* \in \mathcal{A}(x,i)$ such that 
\begin{equation}\label{valuefun}
V(x,i):=	J(x,i;D^*) = \sup_{D \in \mathcal{A}(x,i)}J(x,i;D), \quad (x,i) \in \mathbb{R}\times \mathcal{S},
\end{equation}
where $V(x,i)$ is called the value function and $D^*$ is called an optimal strategy. 

Problem (\ref{valuefun}) falls into the class of singular stochastic control problems with regime switching.  Since, for any admissible $D$, the process ($X^D, \epsilon, D$) in (\ref{stateprocess}) is a strong Markov process,  then we will use the Dynamic Programming Principle and the associated Hamilton-Jacobi-Bellman (HJB for short) equation (see for instance Chapter VIII in \cite{fleming2006controlled}) to solve Problem (\ref{valuefun}).

\section{The verification theorem}

Throughout this paper, without loss of generality, we make the following assumption.

\begin{assume}
One has $\theta_1<\theta_2$. 
\end{assume}

We will see that the parameters $\theta_1$ and $ \theta_2$ play an important role in our subsequent analysis. It is worth noticing that the case $\theta_{1}>\theta_{2}$ is completely symmetric to case $\theta_{1}<\theta_{2}$ and it can be in fact treated with similar mathematical arguments. Also, the case $\theta_1=\theta_2$ has been already solved in \cite{jiang2012optimal} and can be regarded as a special case of our model. For the sake of brevity we therefore omit the discussion of the problem with $\theta_{1} \geq \theta_{2}$. 

Throughout the paper, for any function $g:\mathbb{R} \times \mathcal{S} \to \mathbb{R},$ we use $g', g''$ and $g'''$ to denote, respectively, the first-order, second-order, and third-order derivatives of $g$ with respect to its first argument, unless otherwise stated.

Let $f:\mathbb{R} \times \mathcal{S} \to \mathbb{R}$ be a function such that $f(\cdot,i) \in C^2(\mathbb{R}), i\in\mathcal{S}$, and define the second-order differential operator $\mathcal{L}$ such that
\begin{equation*}
\mathcal{L} f(x,i) := \frac{1}{2}\sigma_i^2f''(x,i) + \mu_i f'(x,i) - \lambda_i f(x,i)+ \sum_{j \neq i}q_{ij}f(x,i), \ (x,i) \in \mathbb{R} \times \mathcal{S}.
\end{equation*}

By the Dynamic Programming Principle, we expect that $V(x,i)$ identifies with a suitable solution $w(x,i)$ to the HJB equation
\begin{equation}\label{HJB}
\max \big\{ (\mathcal{L} -\rho)w(x,i), 1-w'(x,i) \big\}=0, \ x > \theta_i, \ i=1,2,
\end{equation}
with the boundary conditions {$w(x,i) =0$ for $x \leq \theta_i, i=1,2$}. It is worth noting that equation (\ref{HJB}) is actually a system of two variational inequalities with gradient constraints, coupled through the transition rates $\lambda_i, i=1,2.$ In particular, dealing with regime-dependent boundary conditions, we will face a structure of the value function $V$ and of the optimal dividend strategy which is novel with respect to the existing literature. 

From (\ref{HJB}), for any suitable solution $w$ to the HJB equation and regime $i \in \mathcal{S}$, we can define the continuation region
\begin{align*}
 \mathcal{C}(i):=\big \{x \in (\theta_i, \infty): (\mathcal{L} -\rho)w(x,i)=0, \ 1-w'(x,i)< 0 \big \}
\end{align*}
and the intervention region
\begin{align*}
 \mathcal{O}(i):= \big \{x \in (\theta_i, \infty): (\mathcal{L} -\rho)w(x,i) \leq 0, \ 1-w'(x,i)= 0 \big  \}.
\end{align*}
Then, for a given $d_1>\theta_1$, we introduce two dividend strategies $D^{d_1,b,w}$ and $D^{b,w}$ associated with $w$, where, to simplify notation, we write $w:=(w(x,1), w(x,2))$ and $b:=(b_1,b_2)$ for some $b_i \geq \theta_i, i=1,2.$ 
\begin{definition}\label{dividend1}
	For $i=1,2$, let  $N_i$ be a finite subset of $(\theta_i,\infty)$ and $w(\cdot,i) \in C^2((\theta_i,\infty) \backslash N_i) \cap C^1(\theta_{i},\infty) $ be a nondecreasing solution (in the almost-everywhere sense) to the HJB equation (\ref{HJB}). Suppose there exist $b_i>\theta_i,i=1,2,$ and $\theta_1<d_1<b_1$ such that 
	\begin{align*}
		\left\{
		\begin{aligned}
		&	\mathcal{C}(1)=(d_1,b_1),\quad &\mathcal{C}(2)=(\theta_2,b_2),     \\
&\mathcal{O}(1)=(\theta_1,d_1] \cup  [b_1,\infty)   ,\quad &\mathcal{O}(2)=[b_2,\infty).  
		\end{aligned}
		\right.
	\end{align*}
In particular, the following holds
	\begin{align}\label{dividendequation1}
		\left\{
		\begin{aligned}
		&	w'(x,i)=1 , &\forall x\in [b_i,\infty),\\
		& (\mathcal{L}-\rho)w(x,1)=0, &\forall x\in (d_1,b_1), \\
		& w'(x,1)=1 , &\forall x\in (\theta_{1},d_1],\\
		& (\mathcal{L}-\rho)w(x,2)=0, &\forall x\in (\theta_{2},b_2).
		\end{aligned}
		\right.
	\end{align}
	Then,  the admissible \textbf{liquidation-barrier-type} dividend strategy $\{ D^{d_1,b,w}_t, t \geq 0 \}$ is such that  $\mathbb{P}_{x,i}$-a.s.
\begin{enumerate}[(i)]
	\item $X_t^{D^{d_1,b,w}}:=\left\{
\begin{aligned}
& x+ \int_0^t \mu_{\epsilon_s}\textrm{d}s+  \int_0^t \sigma_{\epsilon_s}\textrm{d}W_s-D^{d_1,b,w}_t, \ &\forall   t \in [0, \tau),\\
& \theta_{\epsilon_t}, \ &\forall t \in [\tau, \infty);
\end{aligned}
\right.$
	\item $ X_t^{D^{d_1,b,w}} \in \overline{ \mathcal{C}(\epsilon_t)}, \ \forall   t \in [0, \tau);$
	\item $\int_0^{\infty}\mathbb{I}_{\{    X_t^{D^{d_1,b,w}} \in  \mathcal{C}(\epsilon_t)              \}}\textrm{d}D^{d_1,b,w}_t=\int_0^{\tau}  \mathbb{I}_{\{    X_t^{D^{d_1,b,w}} \in  \mathcal{C}(\epsilon_t)              \}}\textrm{d}D^{d_1,b,w}_t        =      0;$
\item ${D^{d_1,b,w}_{t+}}-{D^{d_1,b,w}_{t}}=X_{t}^{D^{d_1,b,w}}-\theta_1,$ if $\theta_1<X_{t}^{D^{d_1,b,w}}\leq d_1$ on $\{ \epsilon_t =1\}.$
\end{enumerate}
\end{definition}

\begin{definition}\label{dividend2}
For $i=1,2$, let $N_i$ be a finite subset of $(\theta_i,\infty)$ and $w(\cdot,i) \in C^2((\theta_i,\infty) \backslash N_i) \cap C^1(\theta_{i},\infty) $ be a nondecreasing solution (in the almost-everywhere sense) to the HJB equation (\ref{HJB}). Suppose there exist $b_i\geq \theta_i,i=1,2,$ such that 
	\begin{align*}
			\mathcal{C}(i)=(\theta_i,b_i),\   
&\mathcal{O}(i)=[b_i,\infty) ,\ i=1,2,
	\end{align*}
with $\mathcal{C}(i)= \emptyset$ if $b_i=\theta_i$. In particular, the following holds
	\begin{align}\label{dividendequation2}
		\left\{
		\begin{aligned}
		&	w'(x,i)=1 , &\forall x\in [b_i,\infty),\\
		& (\mathcal{L}-\rho)w(x,i)=0, &\forall x\in (\theta_{i},b_i).
		\end{aligned}
		\right.
	\end{align}
	Then, the admissible \textbf{barrier-type} dividend strategy: $D^{b,w}_t:= \sup_{0\leq s\leq t}( x+\mu_{\epsilon_s}s +\sigma_{\epsilon_s}W_s-b_{\epsilon_{s}} )^+, t \in(0,\tau) $  is such that $\mathbb{P}_{x,i}$-a.s.
\begin{enumerate}[(i)]
\item  $X_t^{D^{b,w}}:=\left\{
\begin{aligned}
& x+ \int_0^t \mu_{\epsilon_s}\textrm{d}s+  \int_0^t \sigma_{\epsilon_s}\textrm{d}W_s-D^{b,w}_t, &\forall t \in [0, \tau),\\
& \theta_{\epsilon_t}, & \forall t \in [\tau, \infty);
\end{aligned}
\right.$
\item $ X_t^{D^{b,w}} \in \overline{ \mathcal{C}(\epsilon_t)}, \ \forall   t \in [0, \tau)$;

\item $\int_0^{\infty}\mathbb{I}_{\{    X_t^{D^{b,w}} \in  \mathcal{C}(\epsilon_t)              \}}\textrm{d}D^{b,w}_t=\int_0^{\tau}\mathbb{I}_{\{    X_t^{D^{b,w}} \in  \mathcal{C}(\epsilon_t)              \}}\textrm{d}D^{b,w}_t =  0$.
 \end{enumerate}

\end{definition}

Definition $\ref{dividend1}$ introduces a dividend strategy which is of barrier-type in regime two, and it consists of a combination of a reflecting and liquidation strategy in regime one. {In particular, $D^{d_1,b,w}$ is such that the cash surplus at time $t$ is kept below the reflection barrier $b_{\epsilon_t}$ with minimal effort (i.e., according to a Skorokhod reflection). Moreover, it leads to immediate liquidation whenever the underlying business conditions are in Regime 1 and the cash surplus is sufficiently low, i.e., smaller than $d_1$}. {We will provide in Sections 4.3 and 4.4 sufficient conditions under which the strategy $D^{d_1,b,w}$ is indeed optimal. The numerical analysis of Section 5 suggests that this is the case when $\mu_1 \leq 0$ and $\mu_2$ is sufficiently large} (cf. Sections 4.3, 4.4 and Section 5; compare also to Definition 5.2 in \cite{jiang2012optimal}). Definition \ref{dividend2} instead prescribes a reflecting policy in both regimes and we will see that this is optimal when $\mu_1 \leq 0$ and $|\mu_2|$ is relatively small (cf. Section 4.2).

The strategies $D^{d_1,b,w}$ and $D^{b,w}$ yield lump sum dividend payments at initial time, at the jump times of the Markov chain, and, for the policy $D^{d_1,b,w}$, at the liquidation time $\tau_{d_1}:=\inf\{0\leq t<\tau: X^{D^{d_1,b,w}}_t\leq d_1 \ \text{and} \ \epsilon_{t}=1\} $ (with $\inf \emptyset = +\infty$). In particular, under the above policies we may observe a payout of size $(X_t-b_{\epsilon_{t}})^+$ and, according to policy $D^{d_1,b,w}$, of size $(X_t-\theta_{1})$ if $\epsilon_{t}=1$ and the current surplus level is such that $X_t \leq d_1$. A simulation of the cash reserves process corresponding to a liquidation-barrier-type strategy with $\theta_2<d_1<b_2<b_1$ (compare also to the numerical example in Section 5.3) is displayed in Figure \ref{optimalbarrier}.

\begin{figure}[htbp]
	\centering
	
		\includegraphics[width=3in]{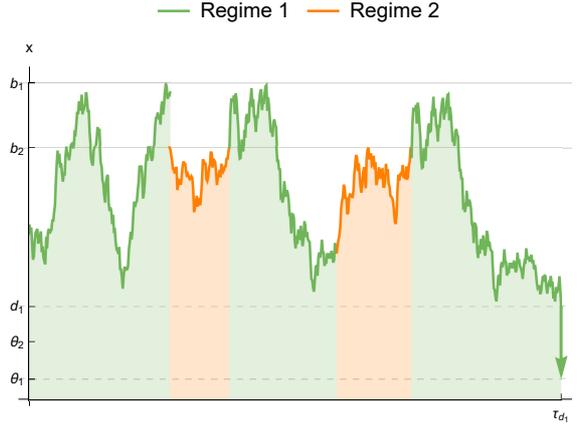}
		\caption{The cash reserves process corresponding to a liquidation-barrier strategy}
		\label{optimalbarrier}
	
\end{figure}

More generally, for a given admissible $D$, we let $\Lambda^D:= \{ t \geq 0: D_{t+} \neq D_{t} \},$ the set of times at which $D$ has a discontinuity. The set $\Lambda^D$ is countable because $D$ is nondecreasing. Recall also that $D^c$ is the continuous part of $D$; that is $D^c_t := D_t-\sum_{0 \leq s \leq t, s \in\Lambda^D }(D_{s+}-D_s). $ In the following we shall denote by $\mathcal{B}(\mathcal{C}(i))$ the boundary points of $\mathcal{C}(i)$ in Definition \ref{dividend1}. That is,   $\mathcal{B}(\mathcal{C}(1))=\{d_1,b_1\}$ and $\mathcal{B}(\mathcal{C}(2))=\{b_2\}.$

\begin{lemma}\label{lemma3.1}
	Let $D \in \{ D^{d_1,b,w}, D^{b,w}\}$, with $ D^{d_1,b,w}$ and $D^{b,w}$ as in Definition \ref{dividend1}, \ref{dividend2}, respectively.  Also, let $\eta \in [0,\tau] $ be any stopping time. Then
	\begin{align*}
		\mathbb{E}_{x,i}\bigg[\int_{0}^{\eta}e^{-\rho s}w'(X^D_s,\epsilon_{s})  dD^c_s\bigg]  - \mathbb{E}_{x,i}\bigg[ \sum_{0 \leq s \leq  \eta, s \in\Lambda^D } e^{-\rho s}(w(X^D_{s+},\epsilon_{s})-w(X^D_{s},\epsilon_{s}) ) \bigg]   = \mathbb{E}_{x,i}\bigg[ \int_{0}^{\eta} e^{-\rho s}dD_s\bigg]. 
	\end{align*}
\end{lemma}

\begin{proof}
	{The proof exploits arguments as those in the proof of Lemma 4.1 in \cite{sotomayor2011classical}. We omit details here for the sake of brevity.}
\end{proof}

\begin{theorem}\label{verificationtheorem}
{	Suppose that $w(\cdot, i) \in C^2((\theta_i,\infty) \backslash N_i) \cap C^1((\theta_i, \infty)), i \in \mathcal{S},$ for a finite subset $N_i \subset (\theta_i,\infty)$, is a nondecreasing solution (in the almost-everywhere sense) to the HJB equation (\ref{HJB}) with the boundary condition {$w(x,i)=0, \forall x\leq \theta_i, i \in \mathcal{S}$}. Then, }
\begin{align*}
w(x,i) \geq V(x,i).
\end{align*}

{Suppose also that $w$ is such that either there exists constants $b_i >\theta_i$, and $d_1$ as in Definition 3.1, or $b_i\geq \theta_i$ as in Definition 3.2.} Then, either the dividend strategies $D^{d_1,b,w}$ or $D^{b,w}$  is optimal. That is, 
\begin{align*}
J(x,i;D^{d_1,b,w}) = V(x,i) \quad \text{or} \quad J(x,i;D^{b,w})  = V(x,i)  .
\end{align*}

\end{theorem}

\begin{proof}
Let $(x,i) \in \mathbb{R} \times \mathcal{S}$ be given and fixed.	We first show that $w(x,i) \geq V(x,i)$. Fix an arbitrary $D \in \mathcal{A}(x,i)$ and let $X^D$ be the corresponding state process, which is the semimartingale
\begin{align*}
X_t^D = x+ \int_0^t \mu_{\epsilon_s}\textrm{d}s+  \int_0^t \sigma_{\epsilon_s}\textrm{d}W_s-D_t^c- \sum_{0 \leq s \leq t, s \in\Lambda^D }(D_{s+}-D_s).
\end{align*}
Then, applying a generalized form of It\^o's formula (see, e.g., \cite{bjork1980finite}) to the process $\{ e^{-\rho (T\wedge \tau \wedge \eta_n)}w(X^D_{T\wedge \tau \wedge \eta_n},$ $ \epsilon_{T\wedge\tau \wedge \eta_n})$, $T \geq 0  \}$, with  $(\eta_n)_n$  being a sequence of stopping times $(\eta_n)_n$ diverging to infinity as $n \uparrow \infty$, we find that
\begin{align*}
e^{-\rho(T \wedge \tau \wedge \eta_n)}&w(X^D_{(T \wedge \tau \wedge \eta_n)+}, \epsilon_{T \wedge \tau \wedge \eta_n})=  w(X^D_0,\epsilon_0) + \int_0^{T \wedge \tau \wedge \eta_n} e^{-\rho s }(\mathcal{L}-\rho)w(X^D_s,\epsilon_s)\textrm{d}s\\
&+ \int_0^{T \wedge \tau \wedge \eta_n} e^{-\rho s }\sigma_{\epsilon_s}w'(X^D_s,\epsilon_s)\textrm{d}W_s -  \int_0^{T \wedge \tau \wedge \eta_n} e^{-\rho s }w'(X^D_s,\epsilon_s)\textrm{d}D_s^c \\  
&+ \sum_{0 \leq s \leq T\wedge \tau \wedge \eta_n, s \in\Lambda^D }(e^{-\rho s}w(X^D_{s+},\epsilon_s)-e^{-\rho s}w(X^D_s,\epsilon_s)) + M_{T \wedge {\tau} \wedge \eta_n}-M_0,
\end{align*}
where we have introduced the local martingale
\begin{align*}
	M_t := \int_{[0,t] \times[0,2]} e^{-\rho s}[ w(X_s,j)-w(X_s,\epsilon_{s})] \widetilde{\pi}(\textrm{d}s,\textrm{d}j), \quad t \geq 0.          
\end{align*}
Here,  $\widetilde{\pi}:=\pi -\nu $ is a compensated random measure (see, e.g., II.1.16 in \cite{jacod2013limit}), where $\pi(\textrm{d}t,\textrm{d}j)=\sum_{s \geq0}\mathbb{I}_{\{\Delta \epsilon_{s(\omega)}\neq0\}}{\delta}_{(s, \Delta \epsilon_s(\omega))}(\textrm{d}t,\textrm{d}j)$, with ${\delta}_{(s,\epsilon)}$ denoting the Dirac measure at the point ($s, \epsilon$), and the compensator $\nu$ is given by $	\nu(\textrm{d}t,\textrm{d}j)= p_{\epsilon_{t-} (j)}(-q_{\epsilon_{t-},\epsilon_{t-}}){\phi}(\textrm{d}j)\textrm{d}t = q_{\epsilon_{t-},j}{\phi}(\textrm{d}j)\textrm{d}t, \ j\in \mathcal{S},$ where $p_{\epsilon_{t-} (j)}= \frac{q_{\epsilon_{t-},j}}{-q_{\epsilon_{t-},\epsilon_{t-}}}=\mathbb{P}(\epsilon_{t}=j|\epsilon_{t-},\Delta \epsilon_t \neq 0)$, and $\phi$ is the counting measure on $\mathcal{S}.$  

Taking conditional expectations, we have
\begin{align}\label{ve1}
\begin{aligned}
&\mathbb{E}_{x,i}\big[e^{-\rho(T \wedge \tau \wedge \eta_n)}w(X^D_{(T \wedge \tau \wedge \eta_n)+}, \epsilon_{T \wedge \tau \wedge \eta_n})\big]= w(X^D_0,\epsilon_0) + \mathbb{E}_{x,i}\bigg[\int_0^{T \wedge \tau \wedge \eta_n} e^{-\rho s }(\mathcal{L}-\rho)w(X^D_s,\epsilon_s)\textrm{d}s\bigg]\\
&+ \mathbb{E}_{x,i}\bigg[\int_0^{T \wedge \tau \wedge \eta_n} e^{-\rho s }\sigma_{\epsilon_s}w'(X^D_s,\epsilon_s)\textrm{d}W_s\bigg] -  \mathbb{E}_{x,i}\bigg[\int_0^{T \wedge \tau \wedge \eta_n} e^{-\rho s }w'(X^D_s,\epsilon_s)\textrm{d}D_s^c \bigg]\\  
&+ \mathbb{E}_{x,i}\bigg[\sum_{0 \leq s \leq T\wedge \tau \wedge \eta_n, s \in\Lambda^D }(e^{-\rho s}w(X^D_{s+},\epsilon_s)-e^{-\rho s}w(X^D_s,\epsilon_s)) \bigg]+ \mathbb{E}_{x,i}\big[M_{T \wedge {\tau} \wedge \eta_n}-M_0\big].
\end{aligned}
\end{align}

The HJB equation (\ref{HJB}) guarantees that $(\mathcal{L}-\rho)w(x,i) \leq 0$ for almost all $x\in (\theta_i,\infty)$. Moreover, $w'(x,i) \geq 1$  for all $x \in (\theta_i, \infty)$ and the mean value theorem implies that $w(y_1,i)-w(y_2,i) \geq y_1-y_2$ for every $y_1,y_2 \in (\theta_i, \infty), y_1 >y_2,$ and for every $i \in \mathcal{S}.$ Hence, $w(X^D_{t+},\epsilon_t) -w(X^D_t, \epsilon_t) \leq X^D_{t+}-X^D_t. $ By observing that $X^D_{t+}-X^D_t = D_{t}-D_{t+}$, we then obtain from  (\ref{ve1}) that
\begin{align}\label{ve2}
\begin{aligned}
&\mathbb{E}_{x,i}\big[e^{-\rho(T \wedge \tau \wedge \eta_n)}w(X^D_{(T \wedge \tau \wedge \eta_n)+}, \epsilon_{T \wedge \tau \wedge \eta_n})\big] \leq  w(x,i) + \mathbb{E}_{x,i}\bigg[\int_0^{T \wedge \tau \wedge \eta_n} e^{-\rho s }\sigma_{\epsilon_s}w'(X^D_s,\epsilon_s)\textrm{d}W_s\bigg]\\ &-  \mathbb{E}_{x,i}\bigg[\int_0^{T \wedge \tau \wedge \eta_n} e^{-\rho s }\textrm{d}D_s^c\bigg] 
- \mathbb{E}_{x,i}\bigg[\sum_{0 \leq s \leq T\wedge \tau \wedge \eta_n, s \in\Lambda^D }e^{-\rho s}(D_{t+}-D_t)\bigg] + \mathbb{E}_{x,i}\big[M_{T \wedge \tau \wedge \eta_n}-M_0\big]\\
&=  w(x,i) + \mathbb{E}_{x,i}\bigg[\int_0^{T \wedge \tau \wedge \eta_n} e^{-\rho s }\sigma_{\epsilon_s}w'(X^D_s,\epsilon_s)\textrm{d}W_s\bigg] -  \mathbb{E}_{x,i}\bigg[\int_0^{T \wedge \tau \wedge \eta_n} e^{-\rho s }\textrm{d}D_s\bigg] + \mathbb{E}_{x,i}\big[M_{T \wedge \tau \wedge \eta_n}-M_0\big].
\end{aligned}
\end{align}

Taking now $(\eta_n)_n$ to be a localizing sequence for the martingale $(M_t+ \int_{0}^t  e^{-\rho s }\sigma_{\epsilon_s}w'(X^D_s,\epsilon_s)\textrm{d}W_s  )_t$, the expectations of the local martingale terms vanish.

Letting now $ n \uparrow \infty$ and $T \to +\infty,$ we get $T \wedge \tau \wedge \eta_n \to \tau$. By monotone convergence theorem and the fact that $w$ is nonnegative, {$X^D_{\tau} \leq \theta_{\epsilon_{\tau}}$} and {$w(x,i)= 0$ for every $x \leq \theta_i, i \in \mathcal{S}$}, we get
\begin{align}\label{ve3}
w(x,i)  \geq \mathbb{E}_{x,i}\big[e^{-\rho \tau }w(X^D_{\tau}, \epsilon_{\tau})\big]+  \mathbb{E}_{x,i}\bigg[\int_0^{\tau} e^{-\rho s }\textrm{d}D_s\bigg] = \mathbb{E}_{x,i}\bigg[\int_0^{\tau} e^{-\rho s }\textrm{d}D_s\bigg].
\end{align}
 
 Thus, by arbitrariness of $D \in \mathcal{A}(x,i)$, we find $w(x,i) \geq V(x,i)$.

 Now consider the dividend strategy $D^{d_1,b,w}$ defined in Definition \ref{dividend1}. From Lemma \ref{lemma3.1} and (ii) of Definition \ref{dividend1}, the inequality in (\ref{ve2}), and therefore also in (\ref{ve3}), becomes an equality. Hence
\begin{align}\label{ve4}
w(x,i)  =\mathbb{E}_{x,i}\big[e^{-\rho \tau }w(X^D_{\tau}, \epsilon_{\tau})\big]+  \mathbb{E}_{x,i}\bigg[\int_0^{\tau} e^{-\rho s }\textrm{d}D_s^{d_1,b,w}\bigg] =   \mathbb{E}_{x,i}\bigg[\int_0^{\tau} e^{-\rho s }\textrm{d}D_s^{d_1,b,w}\bigg] = J(x,i; D^{d_1,b,w}).
\end{align}

Since $w(x,i) = J(x,i; D^{d_1,b,w}) \leq V(x,i),$ then $w(x,i) =V(x,i).$ The argument is similar when we consider the dividend strategy $D^{b,w}$ in Definition \ref{dividend2}.

\end{proof}

\section{Optimal dividend policies}
Our objective in this section is to find the candidate solutions $w(x,i)$ that satisfy the conditions of Theorem \ref{verificationtheorem} and the corresponding free-boundaries for optimal strategies defined in Definitions \ref{dividend1} and \ref{dividend2}. We make the following assumption, that will be in force throughout the paper.

\begin{assume}
$\mu_1\leq 0$ and $\mu_2 \in \mathbb{R}$.
\end{assume}

The case $\mu_1>0$ and $\mu_2 \in \mathbb{R}$ can be treated with similar arguments. In particular, the case $\mu_1>0$ and $\mu_2>0$ with $\theta_1=\theta_2$ has been solved in \cite{sotomayor2011classical}.  Therefore we will restrict our attention to the case $\mu_1 \leq 0, \mu_2 \in \mathbb{R}$. 

Given that $\mu_i, i=1,2,$ can have different signs, we expect that the structure of the optimal dividend strategy varies when the drifts' values change. The mathematical analysis developed in this section provides all analytical solution to the problem. Before giving details, bearing in mind $b_1,b_2$ and $d_1$ as in Definitions \ref{dividend1} and \ref{dividend2}, we provide conjectures about the possible optimal dividend strategies.

\begin{itemize}
\item \textbf{Case (A)}: When $\mu_2<0$ and $|\mu_{2}|$ is sufficiently large, we expect that the optimal strategy is to pay out all the surplus as dividends in both regime 1 and regime 2, since $\mu_1 \leq 0$  as well.

\item \textbf{Case (B)}: When $|\mu_2|$ is sufficiently small, we expect that the optimal strategy is to continue the business in regime 2 and pay dividends according to a barrier strategy with barrier level $b_2$.  On the other hand,  since $\mu_1 \leq 0$, the optimal strategy in regime 1 is to immediately liquidate the company by paying all the surplus as dividends.

\item \textbf{Case (C)}: When $\mu_2>0$ is sufficiently large, we expect that the optimal strategy in regime 2 is similar to that in Case (B). Differently, if the surplus in regime 1 is sufficiently large (larger than $d_1$), then we expect that it may be optimal to continue the business in regime 1 and pay dividends according to a barrier strategy with barrier level $b_1.$ However, if the surplus in regime 1 is sufficiently small (not larger than $d_1$), it may be not worth waiting for a change into the profitable regime 2 and it is therefore optimal to liquidate the company.

\item \textbf{Case (D)}: When continuing to increase $\mu_2>0$, we expect that the optimal dividend strategies are qualitatively similar to Case (C) in both regimes. However, an increase of $\mu_2$ implies that the company might decide to postpone the liquidation option, since it might hope to jump back to the very profitable regime 2. Hence, $d_1$ eventually can decrease and become such that $d_1<\theta_2$. This fact leads to a novel mathematical analysis compared to that of Case (C).
\end{itemize}

According to the previous conjectures, we will now construct a suitable solution to HJB equation (\ref{HJB}) in each of the previous cases and then verify that, under easily verifiable sufficient conditions, our candidate solution is indeed the value function of the problem. The numerical analysis of Section 5 will then confirm the aforementioned structure of the value function and optimal policy with respect to the drift parameter $\mu_2$.

\subsection{Case (A)}

This is the easiest case. We conjecture that it is optimal to liquidate the company at time $0$ in both regimes when $\mu_1$ is not positive and $\mu_2$ is very negative. With regard to this, we guess that 
\begin{equation}\label{1wx1}
	w(x,1)= x-\theta_1, \ x\in[\theta_1,\infty)
\end{equation}
and
\begin{equation}\label{1wx2}
w(x,2)= \left\{
\begin{aligned}
&0,   &x&\in[\theta_1,\theta_2], \\
&x-\theta_2,  &x&\in(\theta_2, +\infty).
\end{aligned}
\right.
\end{equation}

\begin{theorem}\label{theorem4.1}
The function $w(x,i), i=1,2,$ given by (\ref{1wx1}) and (\ref{1wx2}) is the value function if and only if $ \mu_2 \leq (\theta_1-\theta_2)\lambda_2$. Moreover, it is optimal to pay immediately all the surplus as dividends in both regimes.
\end{theorem}

\begin{proof}
	The proof is given in Appendix \ref{ptheorem4.1}.
\end{proof}

 It is worth noting that Theorem \ref{theorem4.1} holds if and only if $\mu_2\leq (\theta_1-\theta_2)\lambda_2<0$. A natural question is now what is the optimal dividend strategy when $(\theta_1-\theta_2)\lambda_2<\mu_2<0$. We will answer such a question in the next subsection.

\subsection{Case (B)}

According to the conjecture presented at the beginning of this section, the optimal dividend strategy in this case is expected to be of the \emph{barrier-type} as in Definition \ref{dividend2} with $b_1=\theta_1$ and $b_2>\theta_2$. Therefore, we consider  $\theta_1< \theta_2< b_2$ and guess that $w(x,1)$ satisfies
\begin{equation}\label{2wx1}
w(x,1)=x-\theta_1, \ x\in[\theta_1,\infty), 
\end{equation}
while $w(x,2)$ satisfies
\begin{align}\label{conb}
	\left\{
	\begin{aligned}
	&	w'(x,2)=1 , &\forall& x\in [b_2,\infty),\\
	& (\mathcal{L}-\rho)w(x,2)=0, &\forall& x\in (\theta_{2},b_2). \\
	\end{aligned}
	\right.
	\end{align}
Solving the equations in (\ref{conb}) thanks to the results in Lemma \ref{ODE}, we have 
	\begin{equation}\label{2wx2}
	w(x,2)= \left\{
	\begin{aligned}
	&0,  &x&\in[\theta_1,\theta_2], \\
	&C_1e^{\alpha_7x}+C_2e^{\alpha_8x}+\frac{\lambda_{2}\mu_2}{(\rho+\lambda_{2})^2}+\frac{\lambda_{2}(x-\theta_1)}{\rho+\lambda_{2}}, &x&\in(\theta_2,b_2),\\
	&x+K_1,  &x&\in[b_2, +\infty),
	\end{aligned}
	\right.
	\end{equation}	
	where $\alpha_7>0>\alpha_8$ are the real roots of the equation  $-\frac{1}{2}\sigma_2^2\alpha^2 -\mu_2\alpha +(\lambda_{2}+\rho)=0$ and the constants $C_1,C_2,K_1$ are to be determined.

To specify the parameters $C_1,C_2,K_1$ and $b_2$, we appeal to the so-called ``smooth fit principle" and to boundary conditions, which dictate that the candidate value function $w(x,2)$ should be $C^0$ in $\theta_{2}$ and $C^2$ in $b_2$. These conditions give rise to the system of equations
\begin{equation}\label{system1}
\left\{
\begin{aligned}
C_1e^{\alpha_7\theta_2}+C_2e^{\alpha_8\theta_2}+\frac{\lambda_{2}\mu_2}{(\rho+\lambda_{2})^2}+\frac{\lambda_{2}(\theta_2-\theta_1)}{\rho+\lambda_{2}}&=0,\\
C_1e^{\alpha_7b_2}+C_2e^{\alpha_8b_2}+\frac{\lambda_{2}\mu_2}{(\rho+\lambda_{2})^2}+\frac{\lambda_{2}(b_2-\theta_1)}{\rho+\lambda_{2}} &=b_2+K_1,\\
C_1\alpha_7e^{\alpha_7b_2}+C_2\alpha_8e^{\alpha_8b_2}+\frac{\lambda_{2}}{\rho+\lambda_{2}}&=1,\\
C_1\alpha_7^2e^{\alpha_7b_2}+C_2\alpha_8^2e^{\alpha_8b_2}&=0.
\end{aligned}
\right.
\end{equation}

The next results are concerned with showing that the system ($\ref{system1}$) leads to a unique solution $b_2>\theta_{2}$ and $w(x,i)$ as in (\ref{2wx1}) and (\ref{2wx2}) is indeed the value function.

\begin{lemma}\label{lemma4.1}
The system of equations (\ref{system1}) admits a unique solution $C_1, C_2,K_1, b_2$ with $b_2> \theta_2$ if and only if $\mu_2> (\theta_1-\theta_2)\lambda_2.$

\end{lemma}

\begin{proof}
	The proof is given in Appendix \ref{plemma4.1}.
\end{proof}

\begin{theorem}\label{theorem4.2}
Suppose that $\mu_{2}>(\theta_{1}-\theta_{2})\lambda_{2}$.  Then the function $w(x,i), i=1,2,$ given by (\ref{2wx1}) and (\ref{2wx2}) is the value function $V(x,i)$ if and only if one of the following conditions holds:
\begin{enumerate}[(i)]
\item $w'(\theta_{2}+,2)\leq \frac{\lambda_{1}+\rho}{\lambda_{1}}$; 

\item $w'(\theta_{2}+,2)> \frac{\lambda_{1}+\rho}{\lambda_{1}} $ and $G(x_0)\leq 0$, where 
$G(x):=\mu_1-(\lambda_{1}+\rho)w(x,1)+\lambda_{1}w(x,2)$ and $x_0$ is the unique solution in $(\theta_2,b_2)$ of the equation $w'(x_0,2) = \frac{\lambda_{1}+\rho}{\lambda_{1}}.$
\end{enumerate}
Moreover, the dividend strategy $D^{b,w}$ with $b_1=\theta_1, b_2>\theta_2$ introduced in Definition \ref{dividend2} is optimal. 
\end{theorem}

\begin{proof}
	The proof is given in Appendix \ref{ptheorem4.2}.
\end{proof}

\begin{remark}\label{remark4.1}
A sufficient condition for the optimality of $w(x,i), i=1,2,$ as in (\ref{2wx1}) and (\ref{2wx2}) is $(\theta_{1}-\theta_{2})\lambda_{2}<\mu_{2}\leq Q:=[(\theta_2-\theta_1)[(\lambda_1+\rho)(\lambda_2+\rho)-\lambda_1\lambda_2]-(\lambda_2+\rho)\mu_1]/\lambda_1>0$. This can be easily shown by studying the function $G$ of $(ii)$ in Theorem \ref{theorem4.2}.
\end{remark}

It is worth noting that, even if $\mu_1$ and $\mu_2$ are both negative, the optimal strategy in regime 2 is to continue the business rather than liquidating the company immediately. The reason lies in the fact that the bankruptcy levels are such that $\theta_{1}< \theta_{2}$, which leads the company to exploit the probability of jumping to regime 1 and then liquidate. Indeed, in regime 2, paying out all dividends at time $0$ would result into a profit $(x-\theta_2)^+$, while employing a barrier strategy the company will accrue the profit due to the singularly continuous payments and eventually, in case of jumping to regime 1, a lump sum payment to the level $\theta_1<\theta_2$. Notice that this structure of the optimal dividend policy cannot be observed in the case $\theta_{1}=\theta_{2}$.

\subsection{Case (C)}

In this case, we shall see that the optimal dividend strategy in regime 1 is completely different from the previous two cases even though we still have $\mu_1 \leq 0$. However, as mentioned before, continuing the business in regime 1 is possible only when the surplus in regime 1 is larger than $d_1$. Therefore, the optimal dividend strategy in the case is expected to be of the \emph{liquidation-barrier-type} presented in Definition \ref{dividend1}.  We discuss the case $ \theta_1<\theta_2<d_1<b_2<b_1$ and guess that $w(x,i)$ satisfies (\ref{dividendequation1}). 

First of all, we consider $x \in [\theta_{1},\theta_{2}]$ and postulate that $ w(x,1) = x-\theta_{1}, w(x,2)=0.$ 

Then, for $x\in(\theta_2,d_1]$, we still have $ w(x,1) = x-\theta_{1}$,  and  
	\begin{equation}\label{4.7}
	\frac{1}{2}\sigma_2^2w''(x,2) +\mu_2w'(x,2) -(\lambda_{2}+\rho)w(x,2) +\lambda_{2}w(x,1) =0.
	\end{equation}
	Solving (\ref{4.7}) thanks to Lemma \ref{ODE}, we find
	\begin{equation*}
	w(x,2)=  C_1e^{\alpha_7x}+C_2e^{\alpha_8x}+\frac{\lambda_{2}\mu_2}{(\rho+\lambda_{2})^2}+\frac{\lambda_{2}(x-\theta_{1})}{\rho+\lambda_{2}},
	\end{equation*}
	where $\alpha_7>0>\alpha_8$ are the real roots of the equation  $-\frac{1}{2}\sigma_2^2\alpha^2 -\mu_2\alpha +(\lambda_{2}+\rho)=0$ and $C_1,C_2$ are constants to be found.

Next we consider $x \in (d_1,b_2)$ and we conjecture that
	\begin{equation}\label{4.8}
	\frac{1}{2}\sigma_1^2w''(x,1) +\mu_1w'(x,1) -(\lambda_{1}+\rho)w(x,1) +\lambda_{1}w(x,2) =0
	\end{equation}
	and
	\begin{equation}\label{4.9}
	\frac{1}{2}\sigma_2^2w''(x,2) +\mu_2w'(x,2) -(\lambda_{2}+\rho)w(x,2) +\lambda_{2}w(x,1) =0.
	\end{equation}
	Solving the previous two equations (cf. Lemma \ref{ODE}), we get
	\begin{equation*}
	w(x,1)=  \widehat{C_3}e^{\alpha_3x}+\widehat{C_4}e^{\alpha_4x}+\widehat{C_5}e^{\alpha_5x}+\widehat{C_6}e^{\alpha_6x}
	\end{equation*}
	and
	\begin{equation*}
	w(x,2)=  C_3e^{\alpha_3x}+C_4e^{\alpha_4x}+C_5e^{\alpha_5x}+C_{6}e^{\alpha_6x},
	\end{equation*}
	where the real numbers $\alpha_3<\alpha_4<0<\alpha_5<\alpha_6$ are the roots of $\phi_1(\alpha)\phi_2(\alpha)-\lambda_{1}\lambda_{2}=0$, and $\phi_i=-\frac{1}{2}\sigma_i^2\alpha^2-\mu_i\alpha+(\lambda_{i}+\rho), i=1,2.$ Also, 
\begin{align}\label{relation1}
C_{j}=\frac{\phi_1(\alpha_j)}{\lambda_{1}}\widehat{C_{j}}=\frac{\lambda_{2}}{\phi_2(\alpha_j)}\widehat{C_j}, \  j=3,4,5,6.
\end{align}

	 Next, we consider $x \in[b_2,b_1).$ Since we expect that $ 1-w'(x,2)=0$, we have $w(x,2)= x+K_1$ for some $K_1 \in \mathbb{R}$. Furthermore, $w(x,1)$ should satisfy
	\begin{equation}\label{4.11}
	\frac{1}{2}\sigma_1^2w''(x,1) +\mu_1w'(x,1) -(\lambda_{1}+\rho)w(x,1) +\lambda_{1}w(x,2) =0.
	\end{equation}
	Solving (\ref{4.11}) with the help of  Lemma \ref{ODE}, we get
	\begin{equation*}
	w(x,1) = C_{7}e^{\alpha_1x}+C_{8}e^{\alpha_2x}+\frac{\lambda_{1}\mu_1}{(\rho+\lambda_{1})^2}+\frac{\lambda_{1}(x+K_1)}{\rho+\lambda_{1}},
	\end{equation*}
	where $\alpha_1>0>\alpha_2$ are the real roots of the equation  $-\frac{1}{2}\sigma_1^2\alpha -\mu_1\alpha +(\lambda_{1}+\rho)=0$.

	Finally, we take $x \in[b_1,+\infty)$. Then from $1-w'(x,1)=0$ and $1-w'(x,2)=0,$ we have $w(x,1)= x+K_2, w(x,2)=x+K_1$, for suitable $K_1, K_2$.
	
	Collecting the above results, in the case $ \theta_{1}<\theta_{2}<d_1<b_2<b_1$, we expect that
	\begin{equation}\label{3wx1}
	w(x,1)= \left\{
	\begin{aligned}
	&x-\theta_{1}, & x&\in[\theta_{1},d_1],\\
	&\widehat{C_3}e^{\alpha_3x}+\widehat{C_4}e^{\alpha_4x}+\widehat{C_5}e^{\alpha_5x}+\widehat{C_6}e^{\alpha_6x}, &x&\in(d_1,b_2), \\
	&C_{7}e^{\alpha_1x}+C_{8}e^{\alpha_2x}+\frac{\lambda_{1}\mu_1}{(\rho+\lambda_{1})^2}+\frac{\lambda_{1}(x+K_1)}{\rho+\lambda_{1}}, &x&\in[b_2,b_1),\\
		&x+K_2,  &x&\in[b_1, +\infty),
	\end{aligned}
	\right.
	\end{equation}
	and
	\begin{equation}\label{3wx2}
	w(x,2)= \left\{
	\begin{aligned}
	&0,  &x&\in[\theta_1,\theta_2], \\
	&C_{1}e^{\alpha_7x}+C_{2}e^{\alpha_8x}+\frac{\lambda_{2}\mu_2}{(\rho+\lambda_{2})^2}+\frac{\lambda_{2}(x-\theta_{1})}{\rho+\lambda_{2}}, &x&\in(\theta_{2},d_1],\\
	& C_3e^{\alpha_3x}+C_4e^{\alpha_4x}+C_5e^{\alpha_5x}+C_{6}e^{\alpha_6x}, &x&\in(d_1,b_2),\\
	&x+K_1,  &x&\in[b_2, +\infty).
	\end{aligned}
	\right.
	\end{equation}
	
Similarly to what we have done in Case (B), we determine the unknown coefficients $C_i, i=1,...,8, $ and $d_1,b_2,b_1,K_1,K_2$  by boundary conditions and smooth fit principle. That is, we assume $w(x,1)$ is $C^1$ in $d_1$ and $b_2$, $w(x,1)$ is $C^2$ in $b_1$; i.e.,
	\begin{align}\label{system31}
\begin{aligned}
	&w(d_1,1) =w(d_{1}+,1), w'(d_1,1) =w'(d_{1}+,1),\\
&w(b_2,1) =w(b_{2}-,1), w'(b_2,1) =w'(b_{2}-,1),\\
	&w(b_1,1)=w(b_{1}-,1), w'(b_1,1)=w'(b_{1}-,1), w''(b_1,1)=w''(b_{1}-,1).
\end{aligned}
	\end{align}
Analogously, we impose that $w(x,2)$ is $C^0$ in $\theta_{2}$, $w(x,2)$ is $C^1$ in $d_1$, $w(x,2)$ is $C^2$ in $b_2$; i.e.,
	\begin{align}\label{system32}
\begin{aligned}
	&w(\theta_{2},2)=w(\theta_{2}+,2), w(d_1,2) =w(d_{1}+,2),w'(d_1,2) =w'(d_{1}+,2),\\
	&w(b_2,2)=w(b_{2}-,2), w'(b_2,2)=w'(b_{2}-,2), w''(b_2,2)=w''(b_{2}-,2).
\end{aligned}	
\end{align}

\begin{remark}\label{non}
Note that system $(\ref{system31})$ and $(\ref{system32})$ is a non-linear system of thirteen equations for thirteen coefficients. By means of simple but lengthy algebra, it is possible to reduce $(\ref{system31})$-$(\ref{system32})$ to a system of three equations for $b_1,b_2$ and $d_1$. Clearly, the ten coefficients $C_1,...,C_8$ and $K_1,K_2$ will be then given in terms of $d_1,b_1,b_2$. However, since the explicit expression of the reduced system is highly unhandy, we refrain from presenting it here. It is worth stressing that, we can neither provide existence nor uniqueness of the solution to system $(\ref{system31})$-$(\ref{system32})$. Nevertheless, we investigate a concrete example in Section 5, in which we solve the system numerically and provide the values of the free-boundaries and their dependencies on relevant parameters.
\end{remark}

The next theorem verifies that $w(x,i)$ as in (\ref{3wx1}) and (\ref{3wx2}) is indeed the value function.  

\begin{theorem}\label{theorem4.3}
Suppose that $\mu_{2}\geq 0$ {and assume that there exist $d_1, b_i, K_i, i=1,2, C_j, j=1,...,8, C_6\neq0$ solving the system of equations (\ref{system31}) and (\ref{system32}) and being} such that $w'(\theta_{2}+,2)\geq1$ and $\theta_1<\theta_2<d_1<b_2<b_1$. Let $\widehat{C_i}, i=3,...,6,$ be defined by (\ref{relation1}). Then the function $w(x,i), i=1,2,$ given by (\ref{3wx1}) and (\ref{3wx2}) is the value function $V(x,i)$  if and only if either of the following conditions holds:
\begin{enumerate}[(i)]
 \item  $w'(\theta_{2}+,2)\leq \frac{\lambda_{1}+\rho}{\lambda_{1}}$; 

\item $w'(\theta_{2}+,2)> \frac{\lambda_{1}+\rho}{\lambda_{1}} $ and $\exists \ x_0 \in (\theta_2,d_1)$ such that $w'(x_0,2) = \frac{\lambda_{1}+\rho}{\lambda_{1}}$ and  $H(x_0)\leq 0$ with $H(x):=\mu_1-(\lambda_{1}+\rho)w(x,1)+\lambda_{1}w(x,2)$;

\item $w'(\theta_{2}+,2)> \frac{\lambda_{1}+\rho}{\lambda_{1}} $ and $H'(d_1)\geq 0$.
\end{enumerate}
Moreover, the associated dividend strategy $D^{d_1,b,w}$ introduced in Definition \ref{dividend1} is optimal.
\end{theorem}

\begin{proof}
The proof is given in Appendix \ref{ptheorem4.3}.
\end{proof}

 It is interesting to point out that dividend payments can also occur just because of a change in regime. In particular, lumpy dividend payouts arise: (a) when the level of the cash surplus falls in the interval $(b_2,b_1)$, and the regime changes from $i=1$ to $i=2$; (b) when the level of the cash surplus falls in the interval $(\theta_2,d_1]$, and the regime changes from $i=2$ to $i=1$.

\begin{remark}\label{remark4.3}
In this subsection, we have considered the case $\theta_1<\theta_2<d_1<b_2<b_1$. As a matter of fact, it is also possible that $\theta_1<\theta_2<d_1<b_1<b_2$. The ordering between $b_1$ and $b_2$ depends indeed on the relations among the different parameters. This case can be treated similarly to the previous one, and for this reason we only give here the form of $V(x,i)$ and the corresponding verification argument without proof. A numerical example will be then provided in Section 5.

\setlength{\parskip}{0.5em}
Consider the case $\theta_1<\theta_2<d_1<b_1<b_2$ and let $\mu_{2}\geq 0$. The value function $V(x,i)$ is given by
\begin{equation}\label{3wx1-1}
	V(x,1)= \left\{
	\begin{aligned}
	&x-\theta_{1}, & x&\in[\theta_{1},d_1],\\
	&C_3e^{\alpha_3x}+C_4e^{\alpha_4x}+C_5e^{\alpha_5x}+C_6e^{\alpha_6x}, &x&\in(d_1,b_1), \\
	&x+K_1,  &x&\in[b_1, +\infty),
	\end{aligned}
	\right.
	\end{equation}
	and
	\begin{equation}\label{3wx2-2}
	V(x,2)= \left\{
	\begin{aligned}
	&0,  &x&\in[\theta_1,\theta_2], \\
	&C_{1}e^{\alpha_7x}+C_{2}e^{\alpha_8x}+\frac{\lambda_{2}\mu_2}{(\rho+\lambda_{2})^2}+\frac{\lambda_{2}(x-\theta_{1})}{\rho+\lambda_{2}}, &x&\in(\theta_{2},d_1],\\
	& \widehat{C_3}e^{\alpha_3x}+\widehat{C_4}e^{\alpha_4x}+\widehat{C_5}e^{\alpha_5x}+\widehat{C_6}e^{\alpha_6x}, &x&\in(d_1,b_1),\\
	&C_{7}e^{\alpha_7x}+C_{8}e^{\alpha_8x}+\frac{\lambda_{2}\mu_2}{(\rho+\lambda_{2})^2}+\frac{\lambda_{2}(x+K_1)}{\rho+\lambda_{2}}, &x&\in[b_1,b_2),\\
	&x+K_2,  &x&\in[b_2, +\infty),
	\end{aligned}
	\right.
	\end{equation}
if and only if either of the following conditions holds:
\begin{enumerate}[(i)]
 \item  $V'(\theta_{2}+,2)\leq \frac{\lambda_{1}+\rho}{\lambda_{1}}$; 
\item $V'(\theta_{2}+,2)> \frac{\lambda_{1}+\rho}{\lambda_{1}} $ and $\exists \ x_0 \in (\theta_2,d_1)$ such that $V'(x_0,2) = \frac{\lambda_{1}+\rho}{\lambda_{1}}$ and  $H(x_0)\leq 0$ with $H(x):=\mu_1-(\lambda_{1}+\rho)V(x,1)+\lambda_{1}V(x,2)$;
\item $V'(\theta_{2}+,2)> \frac{\lambda_{1}+\rho}{\lambda_{1}} $ and $H'(d_1)\geq 0$.
\end{enumerate}
In (\ref{3wx1-1}) and (\ref{3wx2-2}), $d_1,b_1,b_2, \{ C_i\}_{i=1,...,8}, C_6\neq0, K_1,K_2$ satisfy $V'(\theta_{2}+,2)\geq 1$  and
\begin{align}\label{system31-1}
\begin{aligned}
&V(d_1,1) =V(d_{1}+,1), V'(d_1,1) =V'(d_{1}+,1),\\
&V(b_1,1)=V(b_{1}-,1), V'(b_1,1)=V'(b_{1}-,1), V''(b_1,1)=V''(b_{1}-,1),\\
&V(\theta_{2},2)=V(\theta_{2}+,2),V(d_1,2) =V(d_{1}+,2),V'(d_1,2) =V'(d_{1}+,2),\\
&V(b_1,2)= V(b_{1}-,2), V'(b_1,2)=V'(b_{1}-,2),\\
&V(b_2,2)=V(b_{2}-,2), V'(b_2,2)=V'(b_{2}-,2), V''(b_2,2)=V''(b_{2}-,2).
\end{aligned}
\end{align}
Moreover, the dividend strategy $D^{d_1,b,V}$ introduced in Definition \ref{dividend1} is optimal.
\end{remark}

\subsection{Case (D)}

In this case, similarly to Case (C) we expect that the optimal dividend strategy should also be of \emph{liquidation-barrier-type} as in Definition \ref{dividend1}. However, the fact that $d_1<\theta_{2}$ leads to $\theta_{1}<d_1<\theta_{2}<b_2<b_1$ and hence yields a different structure of the (candidate) value function. As a consequence, arguments in the proof of the optimality of the expected solution need to be changed with respect to Case (C). 

We guess that $w(x,i)$ satisfies (\ref{dividendequation1}) and we first consider $x \in [\theta_{1},d_1]$ and obtain that $  w(x,1) = x-\theta_{1}, w(x,2)=0.$ 

Then, for $x\in(d_1,\theta_2],$ since $w(x,2) =0,$ we have
	\begin{equation}\label{4.19}
	\frac{1}{2}\sigma_1^2w''(x,1) +\mu_1w'(x,1) -(\lambda_{1}+\rho)w(x,1) +\lambda_{1}w(x,2) =0.
	\end{equation}
	Solving the above equation thanks to Lemma \ref{ODE}, we have $w(x,1)=  C_1e^{\alpha_1x}+C_2e^{\alpha_2x},$ where $\alpha_1>0>\alpha_2$ are the real roots of the equation  $-\frac{1}{2}\sigma_1^2\alpha^2 -\mu_1\alpha +(\lambda_{1}+\rho)=0$ and $C_1,C_2$ are constants.
	
	Next we consider $x \in (\theta_2,b_2)$ and we have 
	\begin{equation}\label{4.20}
	\frac{1}{2}\sigma_1^2w''(x,1) +\mu_1w'(x,1) -(\lambda_{1}+\rho)w(x,1) +\lambda_{1}w(x,2) =0
	\end{equation}
	and
	\begin{equation}\label{4.21}
	\frac{1}{2}\sigma_2^2w''(x,2) +\mu_2w'(x,2) -(\lambda_{2}+\rho)w(x,2) +\lambda_{2}w(x,1) =0.
	\end{equation}
	Solving those equations (cf. Lemma \ref{ODE}), we obtain
	\begin{equation*}
	w(x,1)=  C_3e^{\alpha_3x}+C_4e^{\alpha_4x}+C_5e^{\alpha_5x}+C_6e^{\alpha_6x}
	\end{equation*}
	and
	\begin{equation*}
	w(x,2)=  \widehat{C_3}e^{\alpha_3x}+\widehat{C_4}e^{\alpha_4x}+\widehat{C_5}e^{\alpha_5x}+\widehat{C_6}e^{\alpha_6x},
	\end{equation*}
	where  $\alpha_3<\alpha_4<0<\alpha_5<\alpha_6$ are the real roots of $\phi_1(\alpha)\phi_2(\alpha)-\lambda_{1}\lambda_{2}=0$, and $\phi_i=-\frac{1}{2}\sigma_i^2\alpha^2-\mu_i\alpha+(\lambda_{i}+\rho), i=1,2.$ Also, 
\begin{align}\label{relation} \widehat{C_j}=\frac{\phi_1(\alpha_j)}{\lambda_{1}}C_{j}=\frac{\lambda_{2}}{\phi_2(\alpha_j)}C_j, \  j=3,4,5,6.
	\end{align}
	
	 Next, we consider $x \in[b_2,b_1)$. Since $ 1-w'(x,2)=0$, we have $w(x,2)= x+K_1,$ for some $K_1\in \mathbb{R}$. Furthermore,
	\begin{equation}\label{4.23}
	\frac{1}{2}\sigma_1^2w''(x,1) +\mu_1w'(x,1) -(\lambda_{1}+\rho)w(x,1) +\lambda_{1}w(x,2) =0,
	\end{equation}
	which, solved thanks to Lemma \ref{ODE}, yields
	\begin{equation*}
	w(x,1) = C_{7}e^{\alpha_1x}+C_{8}e^{\alpha_2x}+\frac{\lambda_{1}\mu_1}{(\rho+\lambda_{1})^2}+\frac{\lambda_{1}(x+K_1)}{\rho+\lambda_{1}},
	\end{equation*}
		where $\alpha_1>0>\alpha_2$ are the real roots of the equation  $-\frac{1}{2}\sigma_1^2\alpha^2 -\mu_1\alpha +(\lambda_{1}+\rho)=0$ and $C_7,C_8$ are constants.
		
	 Finally, we consider $x \in[b_1,+\infty).$ Then, from $1-w'(x,1)=0$ and $1-w'(x,2)=0,$ we have $w(x,1)= x+K_2, w(x,2)=x+K_1,$ for suitable $K_1,K_2$.
	
Summarizing the above arguments, in the case $  \theta_1<d_1<\theta_2<b_2<b_1$, we find that the candidate value function takes the form
	\begin{equation}\label{4wx1}
	w(x,1)= \left\{
	\begin{aligned}
	&x-\theta_{1}, & x&\in [\theta_{1},d_1],\\
	&C_1e^{\alpha_1x}+C_2e^{\alpha_2x},  &x&\in(d_1,\theta_2], \\
	&C_3e^{\alpha_3x}+C_4e^{\alpha_4x}+C_5e^{\alpha_5x}+C_6e^{\alpha_6x}, &x&\in(\theta_2,b_2), \\
	&C_{7}e^{\alpha_1x}+C_{8}e^{\alpha_2x}+\frac{\lambda_{1}\mu_1}{(\rho+\lambda_{1})^2}+\frac{\lambda_{1}(x+K_1)}{\rho+\lambda_{1}}, & x&\in[b_2,b_1),\\
	&x+K_2,  &x&\in[b_1, +\infty),
	\end{aligned}
	\right.
	\end{equation}
	and
	\begin{equation}\label{4wx2}
	w(x,2)= \left\{
	\begin{aligned}
	&0,  &x&\in[\theta_1,\theta_2], \\
	& \widehat{C_3}e^{\alpha_3x}+\widehat{C_4}e^{\alpha_4x}+\widehat{C_5}e^{\alpha_5x}+\widehat{C_6}e^{\alpha_6x}, &x&\in(\theta_2,b_2),\\
	&x+K_1,  &x&\in[b_2, +\infty).
	\end{aligned}
	\right.
	\end{equation}

Similarly to the analysis of the previous sections, we determine the unknown coefficients $C_i, i=1,...,8,$ and $d_1,b_2,b_1,K_1,K_2$  by imposing the boundary conditions and the smooth fit principle. That is, we assume $w(x,1)$ is $C^1$ in $d_1, \theta_{2}$ and $b_{2}$, $w(x,1)$ is $C^2$ in $b_1$. This leads to the following:
	\begin{align}\label{system41}
          \begin{aligned}
	&w(d_1,1) =w(d_{1}+,1), w'(d_1,1)=w'(d_{1}+,1),\\
	&w(\theta_{2},1) =w(\theta_{2}+,1), w'(\theta_{2},1)=w'(\theta_{2}+,1),\\
	&w(b_2,1) =w(b_{2}-,1), w'(b_2,1)=w'(b_{2}-,1),\\
	&w(b_1,1)=w(b_{1}-,1), w'(b_1,1)=w'(b_{1}-,1), w''(b_1,1)=w''(b_{1}-,1).
	\end{aligned}
            \end{align}
Analogously, we assume that	$w(x,2)$ is $C^0$ in $\theta_{2}$ and $w(x,2)$ is $C^2$ in $b_2$, yielding
	\begin{align}
\begin{aligned}\label{system42}
	&w(\theta_{2},2)=w(\theta_{2}+,2),w(b_2,2)=w(b_{2}-,2),\\& w'(b_2,2)=w'(b_{2}-,2), w''(b_2,2)=w''(b_{2}-,2).
\end{aligned}
	\end{align}

As for Case (C), the complexity of the 13 equations arising from (\ref{system41})-(\ref{system42}) is unfortunately such that we can neither provide existence nor uniqueness of the solution. However, a detailed numerical example is provided in Section 5. The next theorem verifies that the $w(x,i)$ in (\ref{4wx1}) and (\ref{4wx2}) is indeed the value function under a set of sufficient conditions.

\begin{theorem}\label{theorem4.5}
Suppose that $\mu_2\geq 0$ {and assume that there exist $d_1, b_i, K_i, i=1,2, C_j, j=1,...,8, C_6\neq0$ solving} the system of equations (\ref{system41}) and (\ref{system42}) {and being} such that $w'(\theta_{2}+,2)\geq 0$ and $\theta_1<d_1<\theta_2<b_2<b_1$. Let $\widehat{C_i}, i=3,...,6,$ be defined by (\ref{relation}).   Then the function $w(x,i), i=1,2,$ given by (\ref{4wx1}) and (\ref{4wx2}) coincides with the value function $V(x,i)$. 

Moreover, the dividend strategy $D^{d_1,b,w}$ introduced in Definition \ref{dividend1} is optimal.
\end{theorem}

\begin{proof}
The proof is given in Appendix \ref{ptheorem4.5}.
\end{proof}

Similarly to Case (C), we notice that dividend payments can also occur just because of a change in regime:  those payments happen when the level of the cash surplus falls in the interval $(b_2,b_1)$, and the regime changes from $i=1$ to $i=2$.

\begin{remark}
Theoretically, it is also possible that $b_1<b_2$ leading to $\theta_1<d_1<\theta_2<b_1<b_2$.  This case can be treated similarly to the previous one. Below we only give the form of $V(x,i)$ and the corresponding verification argument without proof. 

\setlength{\parskip}{0.5em}
Consider the case $  \theta_1<d_1<\theta_2<b_1<b_2$ and let $\mu_{2}\geq0$. The value function $V(x,i)$ is given by
	\begin{equation}\label{4wx1-1}
	V(x,1)= \left\{
	\begin{aligned}
	&x-\theta_{1}, & x&\in[\theta_{1},d_1],\\
	&C_1e^{\alpha_1x}+C_2e^{\alpha_2x},  &x&\in (d_1,\theta_2], \\
	&C_3e^{\alpha_3x}+C_4e^{\alpha_4x}+C_5e^{\alpha_5x}+C_6e^{\alpha_6x}, &x&\in (\theta_2,b_1), \\
	&x+K_1,  &x&\in[b_1, +\infty),
	\end{aligned}
	\right.
	\end{equation}
	and
	\begin{equation}\label{4wx2-2}
	V(x,2)= \left\{
	\begin{aligned}
	&0,  &x&\in[\theta_1,\theta_2], \\
	& \widehat{C_3}e^{\alpha_3x}+\widehat{C_4}e^{\alpha_4x}+\widehat{C_5}e^{\alpha_5x}+\widehat{C_6}e^{\alpha_6x}, &x&\in(\theta_2,b_1),\\
&C_{7}e^{\alpha_7x}+C_{8}e^{\alpha_8x}+\frac{\lambda_{2}\mu_2}{(\rho+\lambda_{2})^2}+\frac{\lambda_{2}(x+K_1)}{\rho+\lambda_{2}}, & x&\in[b_1,b_2),\\
	&x+K_2,  &x&\in[b_2, +\infty),
	\end{aligned}
	\right.
	\end{equation}	
with $d_1, b_i, K_i, i=1,2, C_j, j=1,...,8, C_6\neq0$ satisfying $V'(\theta_{2}+,2)\geq 0$ and
	\begin{align}\label{system41-1}
          \begin{aligned}
	&V(d_1,1) =V(d_{1}+,1), V'(d_1,1)=V'(d_{1}+,1),\\
	&V(\theta_{2},1) =V(\theta_{2}+,1), V'(\theta_{2},1)=V'(\theta_{2}+,1),\\
	&V(b_1,1)=V(b_{1}-,1), V'(b_1,1)=V'(b_{1}-,1), V''(b_1,1)=V''(b_{1}-,1),\\
	&V(\theta_{2},2)=V(\theta_{2}+,2),V(b_1,2) =V(b_{1-},2), V'(b_1,2)=V'(b_{1}-,2),\\
	&V(b_2,2)=V(b_{2}-,2), V'(b_2,2)=V'(b_{2}-,2), V''(b_2,2)=V''(b_{2}-,2).
\end{aligned}
	\end{align}
	 Moreover, the dividend strategy $D^{d_1,b,V}$ introduced in Definition \ref{dividend1} is optimal.

However, it is also worth stressing that our numerical study of Section 5 does not identify a solution to (\ref{system41-1}) such that $\theta_{1}<d_1<\theta_{2}<b_1<b_2$. 
\end{remark}

\begin{remark}
{The results of Theorems 4.3 and 4.4 identify a set of numerically easily verifiable conditions ensuring the optimality of a liquidation-barrier strategy triggered by the boundaries $d_1,b_1$ and $b_2$. However, since we are not able to prove existence and uniqueness of a solution to systems (4.26)-(4.27) and (4.30), we cannot give a definitive answer about the optimality of $D^{d_1,b, w}$ or about the existence of other optimal policies with different structure. On the other hand, the subsequent numerical study performed in Section 5 confirms the optimality of a liquidation-barrier strategy in Cases (C) and (D), as well as the fact that conditions in Theorems 4.1-4.4 are indeed disjoint as one moves from a set of requirements to the next one by increasing $\mu_2$. It is also worth noticing that our numerical analysis seems to be robust w.r.t. different choices of the model's parameters.} 
\end{remark}

\section{A numerical study}

In this section, we provide numerical illustrations of the optimal dividend strategy and of the value function in each of the cases discussed in Section 4. Moreover, we investigate the sensitivities of the optimal threshold levels on relevant parameters and analyze the consequent  economic meaning. We numerically solve the differential equations satisfied by the value functions in each case, and the numerics was done using Mathematica 12.0.

With regards to the discussion on Case (A)-Case (D) at the beginning of Section 4, we consider a company whose cash surplus follows a regime-switching model with fixed parameters as in Table \ref{tab1}, and we vary $\mu_{2}\in \mathbb{R}.$
\begin{table}[htbp]
\caption{The parameter set for varying $\mu_{2}$}
\label{tab1}
\centering
	\begin{tabular}{cccccc}
		\toprule  
		 $i$	& $\mu_i$ &$\sigma_i$ & $\lambda_{i}$   &$\theta_{i}$ &
		  $\rho$ \\
		  \midrule
	1	&-0.8 &0.5 &10& -0.2 &0.5\\
	2 & &0.5   &1 & 0.2&0.5\\
		\bottomrule
	\end{tabular}
\end{table}
It is worth noting that we choose $\lambda_{1}>\lambda_{2}$ in such a way that for $\mu_2$ taking values in $(-\infty, 2]$ we can observe all the cases treated in Section 4.

\subsection{\boldmath $ \mu_{2}<0$ and $|\mu_{2}|$ is sufficiently large}
According to Theorem \ref{theorem4.1} in Case (A), we find that when $\mu_{2}\leq (\theta_1-\theta_2)\lambda_2=-0.4$, the optimal dividend strategy is to liquidate the company at time $0$ in each regime. It is not surprising because the drifts are negative in both regimes. Moreover, from Figure \ref{casea}, we observe that $V(x,1)=x-\theta_1>V(x,2)=x-\theta_2$, as clearly expected since $\theta_1<\theta_2$.  Increasing $\mu_{2}$, the optimal strategy turns to that of Case (B), see Table \ref{tab2}.

\begin{figure}[htbp]
	\centering
		\includegraphics[width=3in]{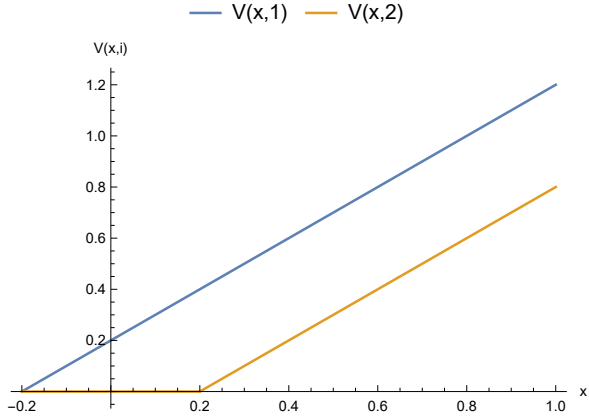}
		\caption{The value function in Case (A)}
		\label{casea}
\end{figure}

\subsection{\boldmath$|\mu_{2}|$ is sufficiently small}

 We consider $\mu_{2}\in [-0.39, 0.68]$, since Figure \ref{gx} shows that $G(x_0)>0$ {---where $x_0$ is the maximiser of $G(x)$---}when $\mu_{2}=0.69$, which then does not satisfy the conditions in Theorem \ref{theorem4.2}. Solving (\ref{system1}) numerically, we obtain {$b_2>\theta_2=0.2$}, whose values are shown in Table \ref{tab2}, where $Q=[(\theta_{2}-\theta_{1})[(\lambda_{2}+\rho)(\lambda_{1}+\rho)-\lambda_{1}\lambda_{2}]-(\lambda_{2}+\rho)\mu_1]/\lambda_{1}=0.35$ as defined in Remark \ref{remark4.1}. Due to $\mu_1<0$ and $|\mu_2|$ is sufficiently small, regime 1 is very unprofitable for the company. Indeed, although the business condition may jump to regime 2 in the future, $|\mu_2|$ is however not  large enough to compensate such a waiting. Also, the bankruptcy level $\theta_2$ is larger than $\theta_1$, then leading to a possible lower value of liquidation. Thus the company chooses to pay all the surplus as dividends and then immediately go bankrupt in regime 1. However, the business conditions are not too adverse for the company in regime 2. Hence the possibility to jump to regime 1 in the future, and then to liquidate with reference to a lower bankruptcy level $\theta_1<\theta_2$, leads to a compensation effect that induces the company to continue the business and pay dividends according to a barrier strategy in regime 2.  Figure \ref{valuefunctionb} shows $V(x,1)$ and $V(x,2)$ for different choices of $\mu_2$.

\begin{figure}[htbp]
	\centering
	\begin{minipage}[b]{0.5\linewidth}
		\includegraphics[width=3in]{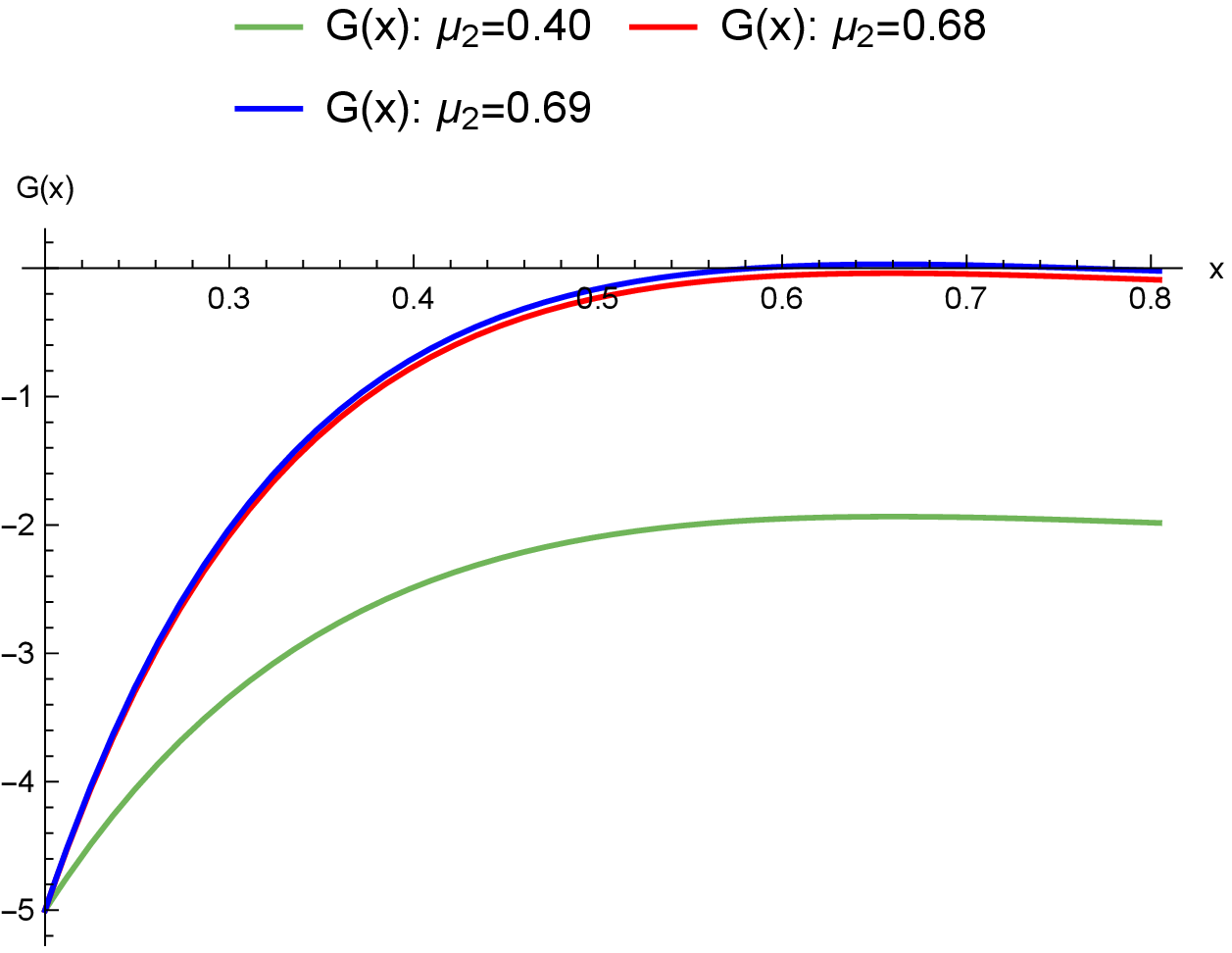}
		\caption{The function $G(x)$ of Theorem \ref{theorem4.2}}
		\label{gx}
	\end{minipage}%
	\begin{minipage}[b]{0.5\linewidth}
		\includegraphics[width=3in]{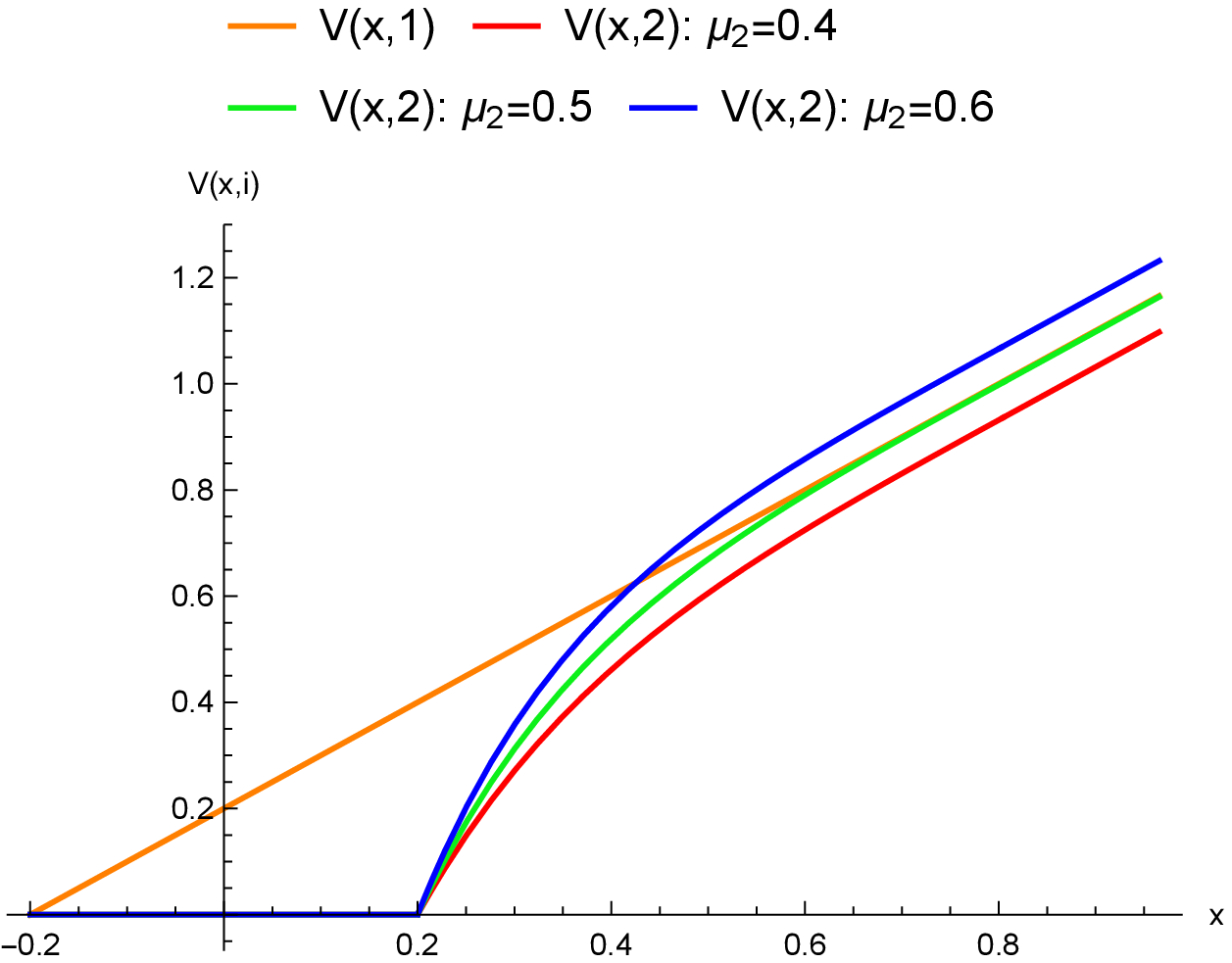}
		\caption{The value function for different values of $\mu_{2}$}
		\label{valuefunctionb}
	\end{minipage}
\end{figure}

Furthermore, from Table \ref{tab2}, we note that $b_2$ is increasing when $\mu_{2}$ increases in the interval $[-0.39,0.5]$ and then decreases when $\mu_{2}$ continues to increase. This is consistent with what it is observed in the no-regime-switching case (see, e.g., Section 2.4 in \cite{moreno2014market}). When $\mu_{2}$ is relatively large fewer precautionary reserves are needed. However, when the business conditions are relatively poor, the barrier $b_2$ must be decreased so to increase the probability of hitting it, even if this leads to a large bankruptcy risk.

\begin{table}[htbp]
	\caption{The optimal barrier levels for different $\mu_{2}$ in Case (B)}
	\label{tab2}
	\centering
	\begin{tabular}{cccccccccccc}
		\toprule  
		$\mu_2$	& -0.39  & -0.3 & -0.2 & 0  &0.2 & Q=0.35&0.4& 0.5 &0.6&0.64&0.68  \\
		\bottomrule
		$b_2$ & 0.220 &0.388  &0.532 &  0.703&0.780& 0.802       &0.805& 0.806 &0.802&0.800&0.797\\
		\bottomrule
	\end{tabular}
\end{table}

Next, we choose the parameters as in Table \ref{tab1} and vary $\sigma_2, \lambda_{2}, \theta_{2}, \rho$ individually, while keeping the other parameters fixed. The results are shown in Figures \ref{sigma2}-\ref{rho}. From Figure \ref{sigma2}, we note that $b_2$ is increasing when $\sigma_2$ increases regardless of the value of $\mu_{2}$. The result implies that when the volatility of a regime is overwhelmingly large, it is better to maintain a high level of reserve in the more volatile regime to reduce the risk of early ruin and to adopt a much lower reserve level to achieve early dividend yields in the less volatile regime.
\begin{figure}[htbp]
		\setlength{\abovecaptionskip}{0pt}
	\setlength{\belowcaptionskip}{5pt}
	\begin{minipage}[b]{0.5\linewidth}
		\centering
		\includegraphics[width=2.7in]{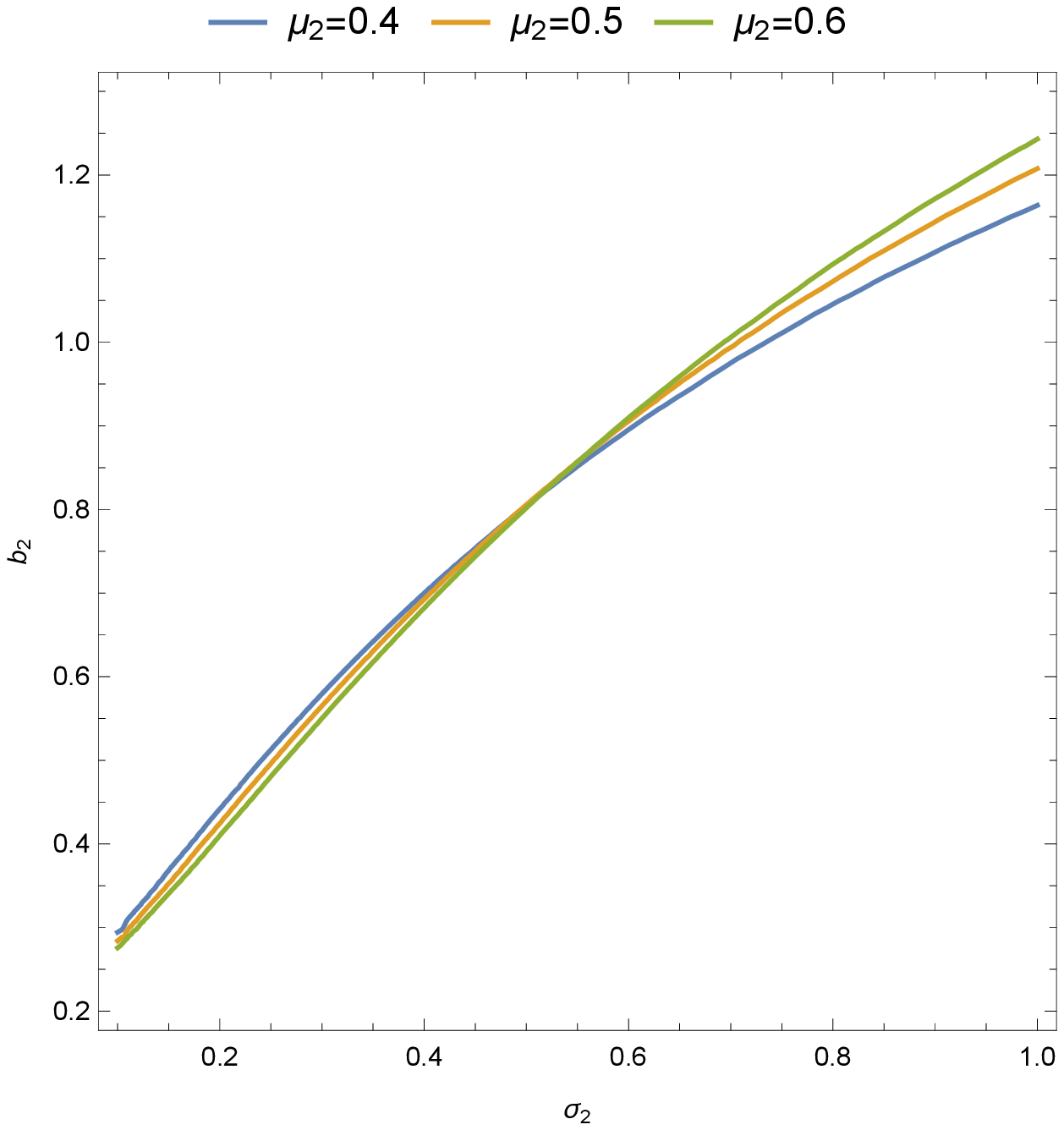}
		\caption{Values of $b_2$ when varying $\sigma_2$}
		\label{sigma2}
	\end{minipage}%
	\begin{minipage}[b]{0.5\linewidth}
		\centering
		\includegraphics[width=2.7in]{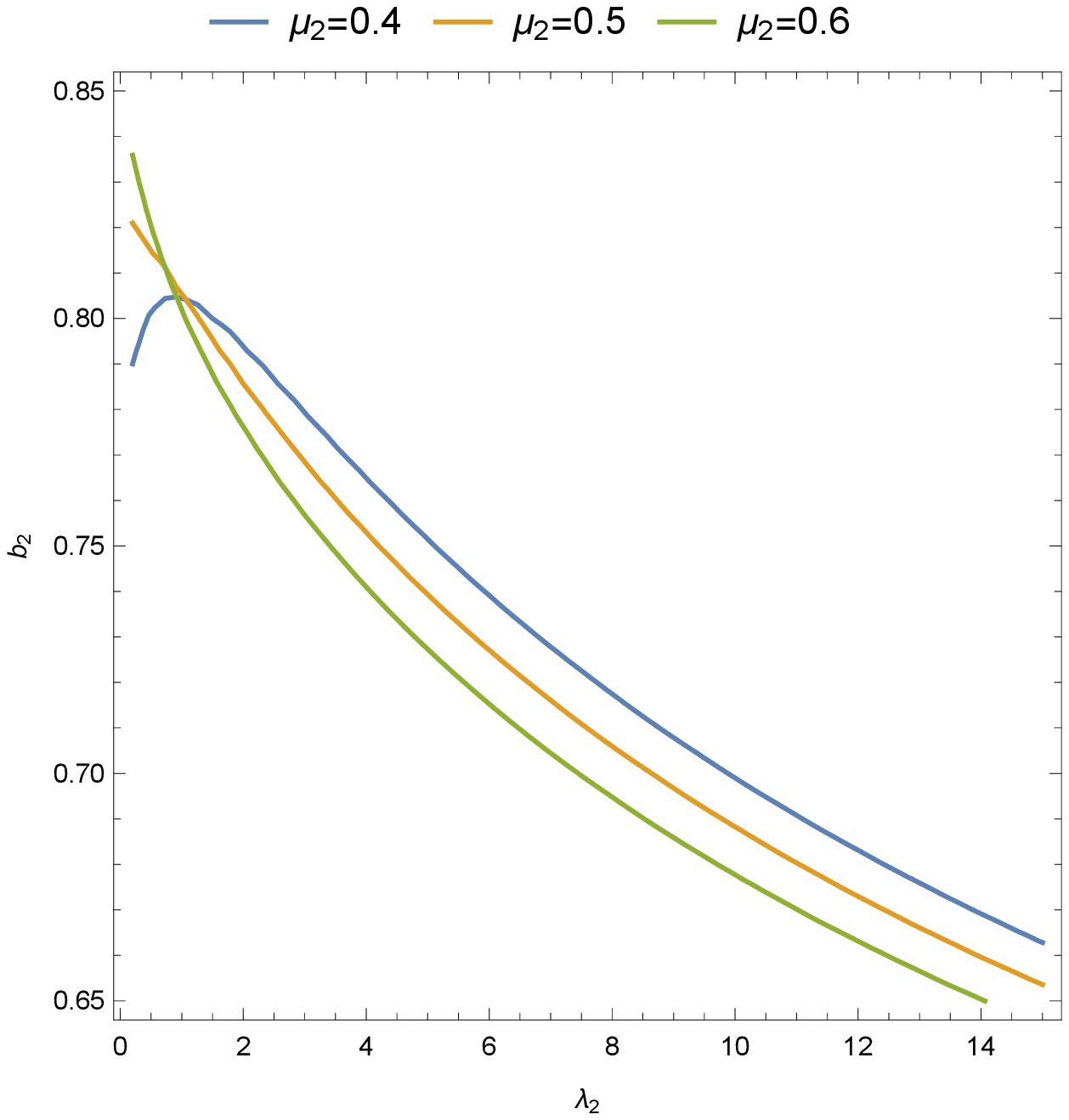}
		\caption{Values of $b_2$ when varying $\lambda_2$}
		\label{lambda2}
	\end{minipage}
\end{figure}

From Figure \ref{lambda2}, we observe that different values of $\mu_2$ induce different behaviors of $b_2$ as the Markov transition rate $\lambda_{2}$ increases. Specifically, $b_2$ is decreasing when $\lambda_{2}$ increases, if $\mu_{2}$ takes a relatively large value. This can be explained by observing that the stationary probability of the Markov chain being in state 2 is $\lambda_{1}/(\lambda_{1}+\lambda_{2})$. Therefore, when $\lambda_{2}$ is increased, the Markov chain is expected to stay in regime 2 for a shorter time and the company's risk of entering regime 1 (an undesirable regime with respect to regime 2 with relatively large $\mu_{2}$) is getting higher. Hence, the company should maximize the probability of hitting the dividend barrier by lowering the latter.

\begin{figure}[htbp]
	\setlength{\abovecaptionskip}{0pt}
	\setlength{\belowcaptionskip}{5pt}
\begin{minipage}[b]{0.5\linewidth}
	\centering
	\includegraphics[width=2.7in]{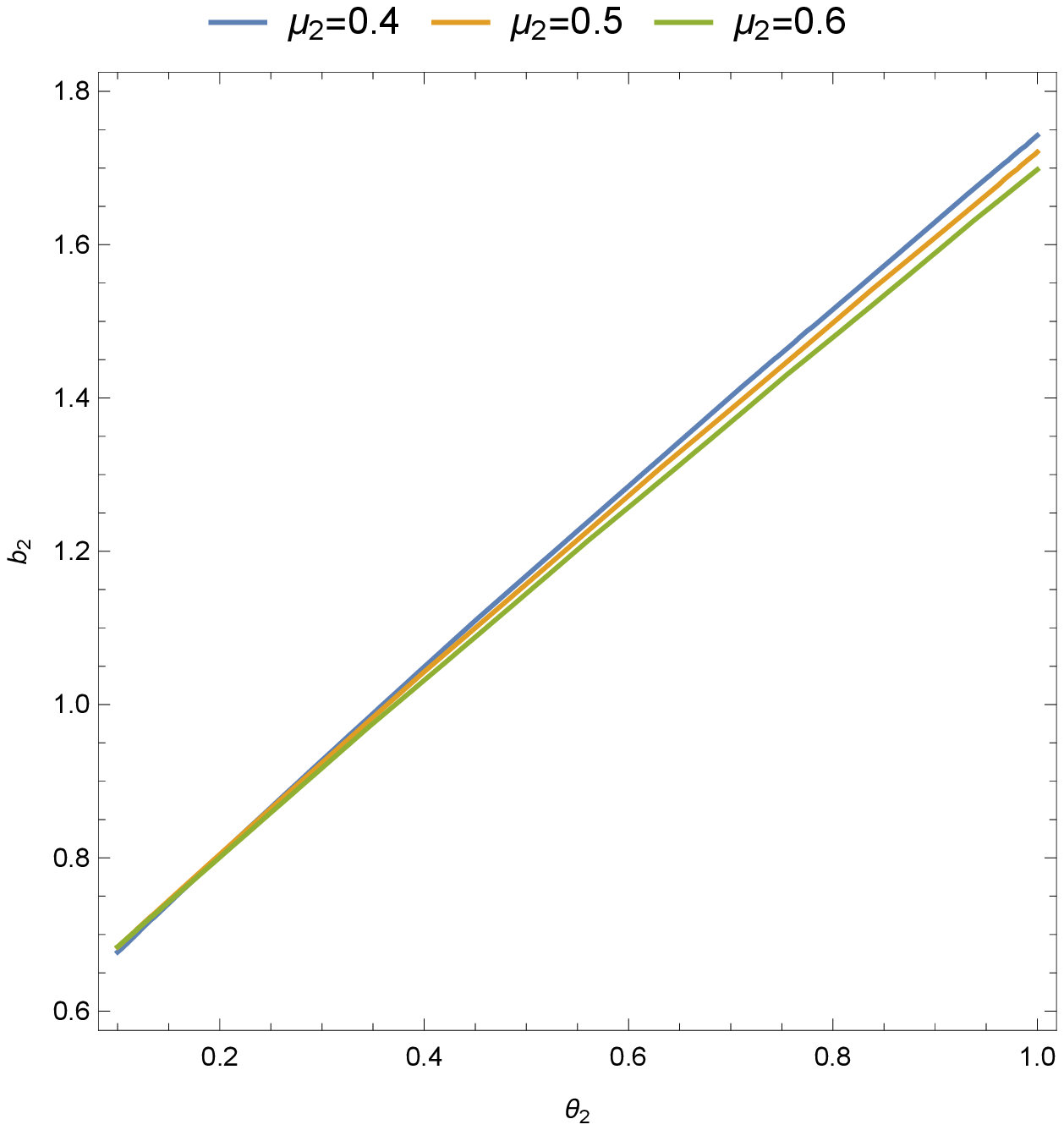}
	\caption{Values of $b_2$ when varying $\theta_2$}
	\label{theta2}
\end{minipage}%
\begin{minipage}[b]{0.5\linewidth}
	\centering
	\includegraphics[width=2.7in]{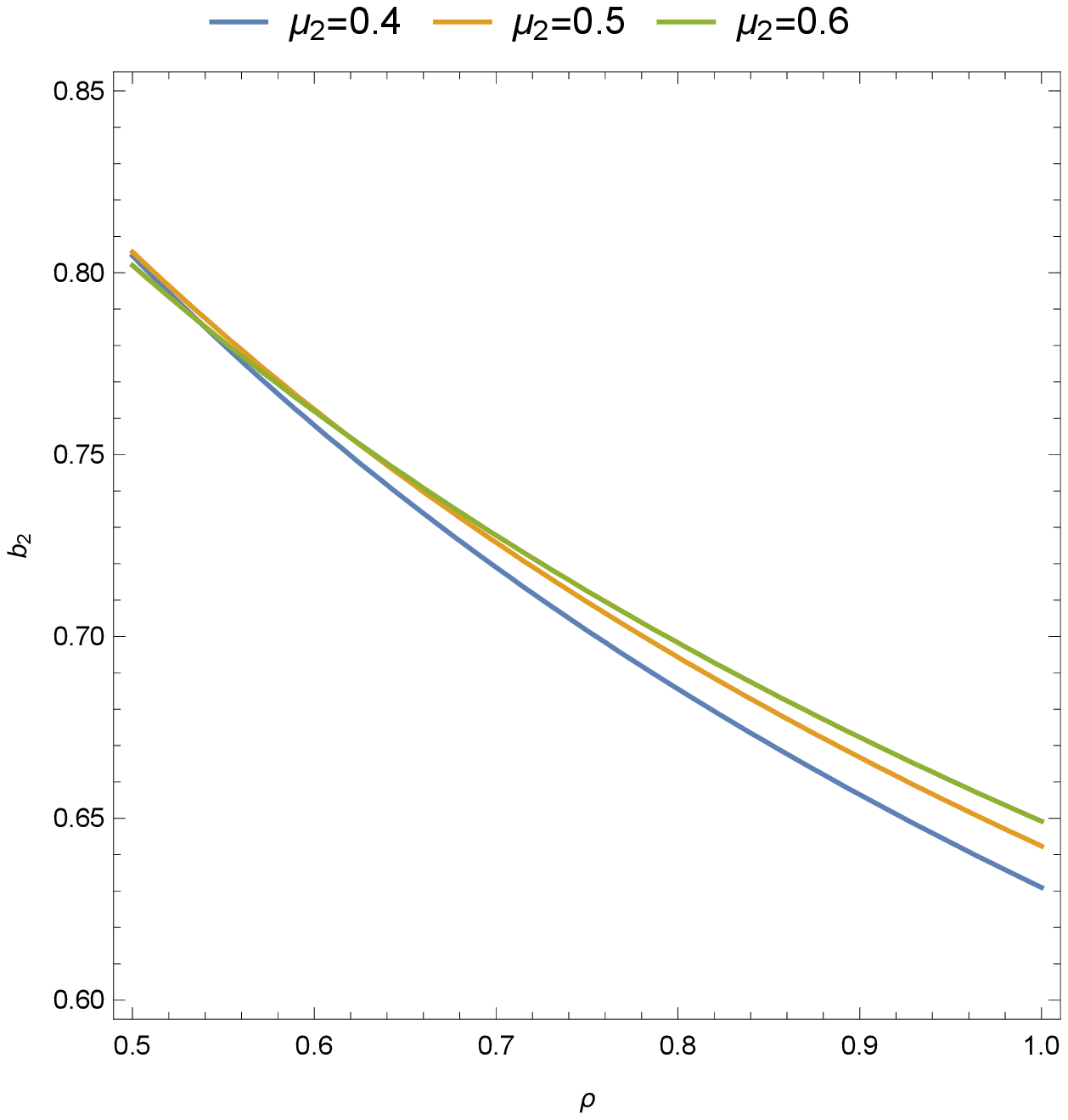}
	\caption{Values of $b_2$ when varying $\rho$}
	\label{rho}
\end{minipage}
\end{figure}

However, when $\mu_{2}$ takes a relatively small value, $b_2$ is a single-picked function of $\lambda_{2}$. Decreasing $\lambda_{2}$ the (stationary) probability of the business condition to be in regime 2 increases. Then, in order to maximize the profits resulting from a barrier strategy, the optimal dividend threshold $b_2$ should be decreased (although this might increase the bankruptcy risk) in order to compensate the relatively small trend of profitability in state 2. On the other hand, increasing $\lambda_{2}$ yields a larger (stationary) probability of being in state 1, which is characterized by a lower bankruptcy level. In this case the company decreases the level of the dividend barrier $b_2$ in order to maximize the dividends' payments accrued according to a barrier strategy, before jumping to the more probable state 1 where a liquidation is optimal.

From Figure \ref{theta2}, it is clear that $b_2$ is increasing when $\theta_{2}$ increases regardless of the value of $\mu_{2}$. This suggests that, when $\theta_{2}$ increases, the company should set a higher dividend barrier $b_2$ so to increase the cash reserves and minimize the bankruptcy risk.

From Figure \ref{rho}, we see that $b_2$ is decreasing when the rate of discounting $\rho$ increases, regardless of the value of $\mu_{2}$. If the rate of discounting is higher, which means a higher degree of impatience of the company, it is then optimal to anticipate dividend payments by lowering the dividend barrier $b_2$.

\subsection{\boldmath$\mu_{2}> 0$ is sufficiently large} 
We now consider $\mu_{2} \in [0.69,1.09],$ which leads to $\theta_{2}<d_1$ as in Case (C). Numerically solving (\ref{system31-1}) and (\ref{system31}), (\ref{system32}), respectively, the results for the optimal liquidation-barrier levels are shown in Table \ref{tab3}, where the left chart shows the case $\theta_{1}<\theta_{2}<d_1<b_1<b_2$ discussed in Remark 4.3, while the right one corresponds to the case $\theta_{1}<\theta_{2}<d_1<b_2<b_1$ of Theorem 4.3.  Also here we find that the conditions {(i)-(iii)} in Theorem 4.3 and Remark 4.3 are satisfied, see Figures \ref{hx1}-\ref{hx2}. 

\begin{table}[htpb]
	\caption{The optimal liquidation-barrier levels for different $\mu_{2}$ in Case (C)}
	\label{tab3}
	\centering
	\begin{minipage}{0.28\textwidth}
		\centering
		\begin{tabular}{ccc}
			\toprule  
			$\mu_2$	& 0.69  & 0.70  \\
			\bottomrule
			
			$d_1$	&0.537 &0.469\\
			$b_1$ &0.741 &0.799  \\
			$b_2$&0.797 &0.800  \\
			\bottomrule
		\end{tabular}
	\end{minipage}
	\begin{minipage}{0.68\textwidth}  
		\centering
		\begin{tabular}{cccccccc}
			\toprule  
			$\mu_2$	& 0.71  & 0.74 & 0.80 & 0.90  & 1.0&1.06&1.09\\
			\bottomrule
			
			$d_1$ & 0.431 &0.366 &0.299 & 0.245&0.216& 0.205       &0.20002\\
			$b_2$ & 0.804 &0.814  &0.830 & 0.845&0.853& 0.854       &0.855\\
			$b_1$&0.834 &0.897  &0.963 &1.022&1.059&1.075&1.082\\
			\bottomrule
		\end{tabular}
	\end{minipage}
	\centering
\end{table}

 From Table \ref{tab3}, we note that $d_1$ is decreasing when $\mu_{2}$ increases. This suggests that when $\mu_{2}$ is increasing, it is optimal to set a lower liquidation level $d_1$ and thus postpone the liquidation in regime 1, since the company might hope to jump to the very profitable regime 2.  On the other hand, we also find that $b_1$ is increasing when $\mu_{2}$ increases. This implies that it is optimal to maintain a high level of the reserve by setting a larger dividend barrier $b_1$  in order to reduce the chance of liquidating in regime 1. Meanwhile, $b_2$ is increasing as $\mu_{2}$ increases. Although $\mu_{2}>0$ is sufficiently large, there is still a risk of jumping to regime 1. In particular, when the cash surplus $X_t$ in regime 2 belongs to $(\theta_{2},d_1]$, jumping to regime 1 will lead to an immediate liquidation. In order to postpone the liquidation, the company retains enough precautionary reserves by increasing dividend barrier $b_2$. Figures \ref{valuefunctionc1}-\ref{valuefunctionc2} illustrate the value function for both regimes when $\mu_{2}>0$ is sufficiently large. 
\begin{figure}[htbp]
	\centering
	\begin{minipage}[b]{0.5\linewidth}
		\includegraphics[width=3in]{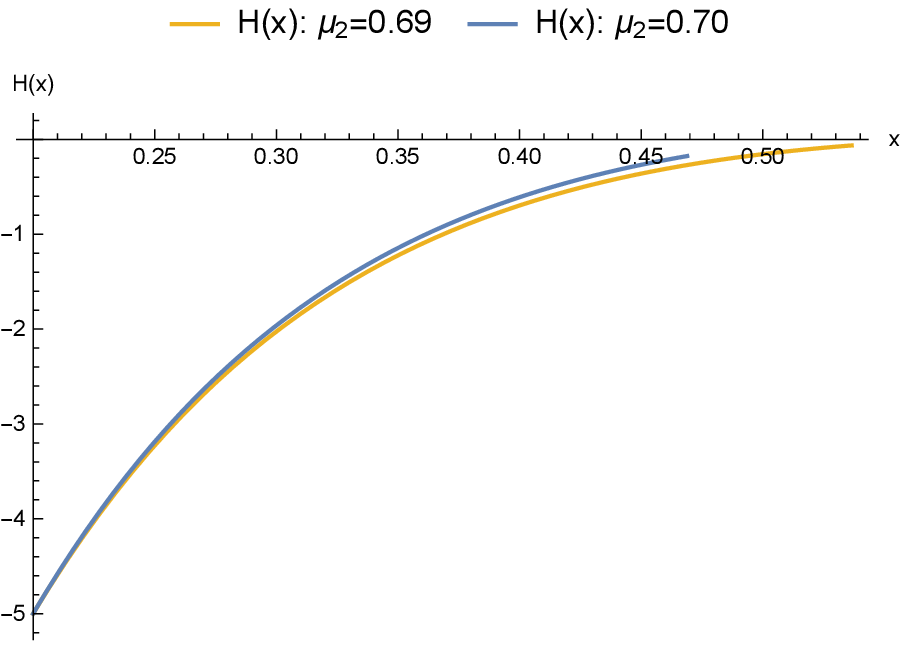}
		\caption{The function $H(x)$ of Remark 4.3}
		\label{hx1}
	\end{minipage}%
	\begin{minipage}[b]{0.5\linewidth}
		\includegraphics[width=3in]{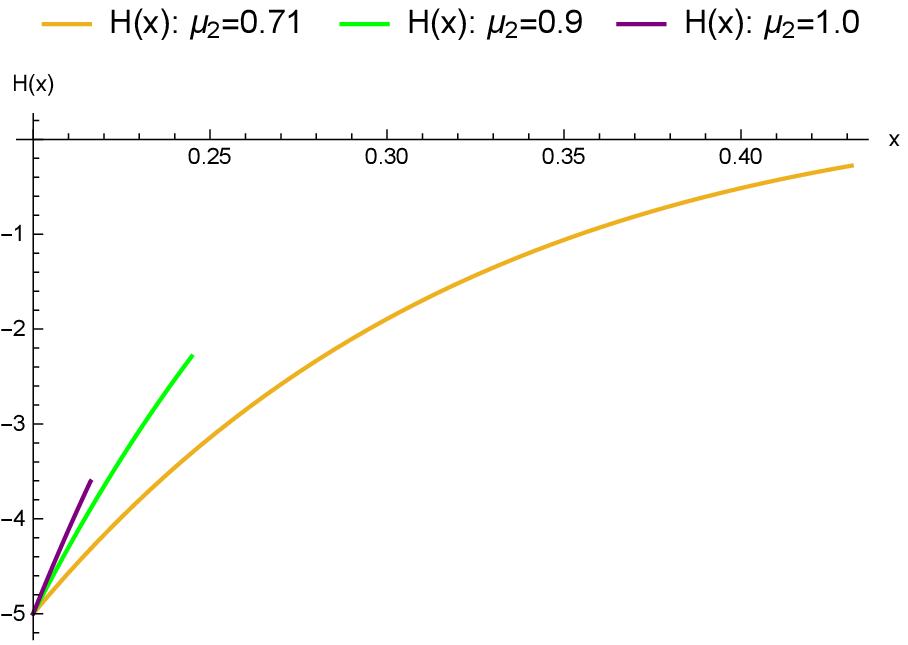}
		\caption{The function $H(x)$ of Theorem \ref{theorem4.3}}
		\label{hx2}
	\end{minipage}
\end{figure}

\begin{figure}[htbp]
	\begin{minipage}[b]{0.5\linewidth}
		\centering
		\includegraphics[width=3in]{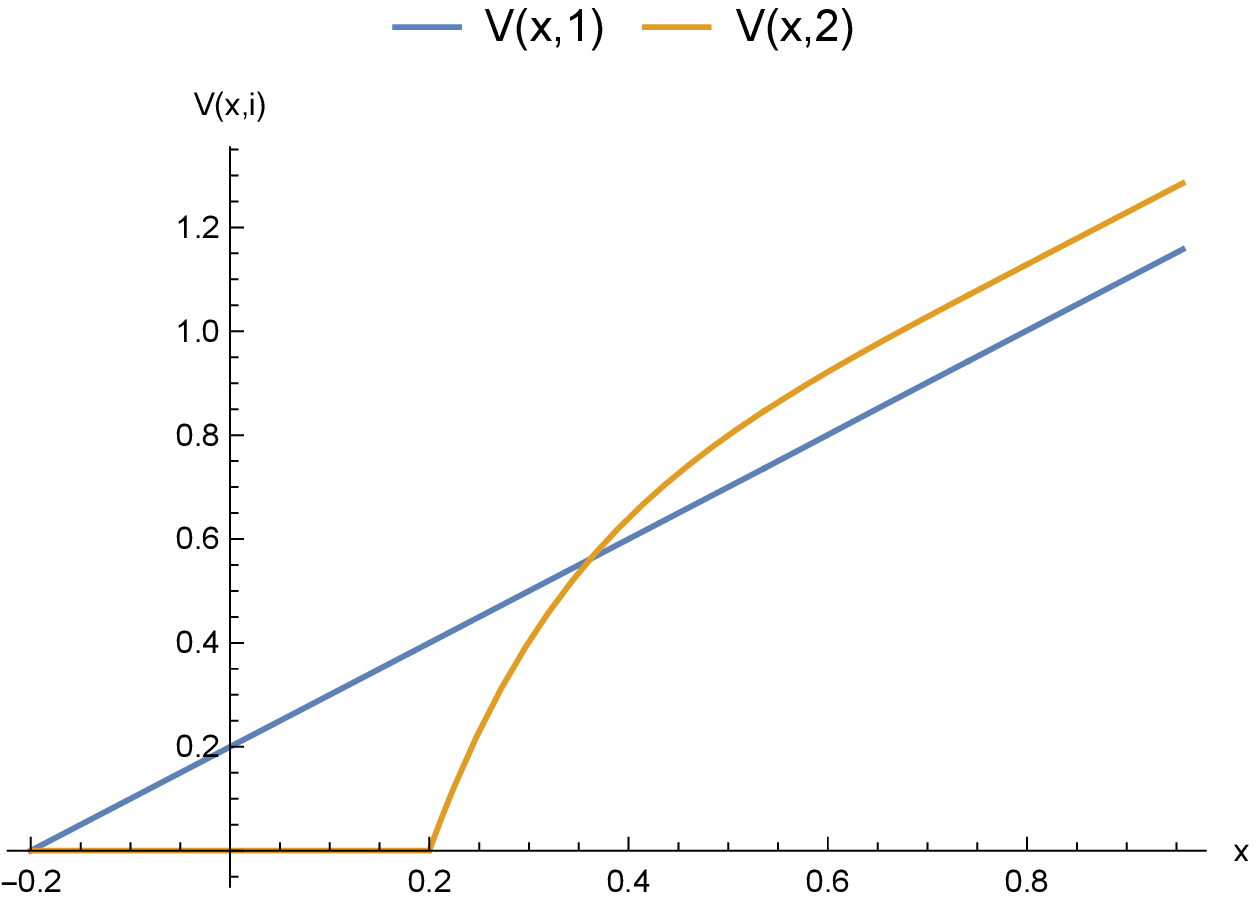}
		\caption{The value function for $\mu_{2}=0.69$}
		\label{valuefunctionc1}
	\end{minipage}%
	\begin{minipage}[b]{0.5\linewidth}
		\centering
		\includegraphics[width=3in]{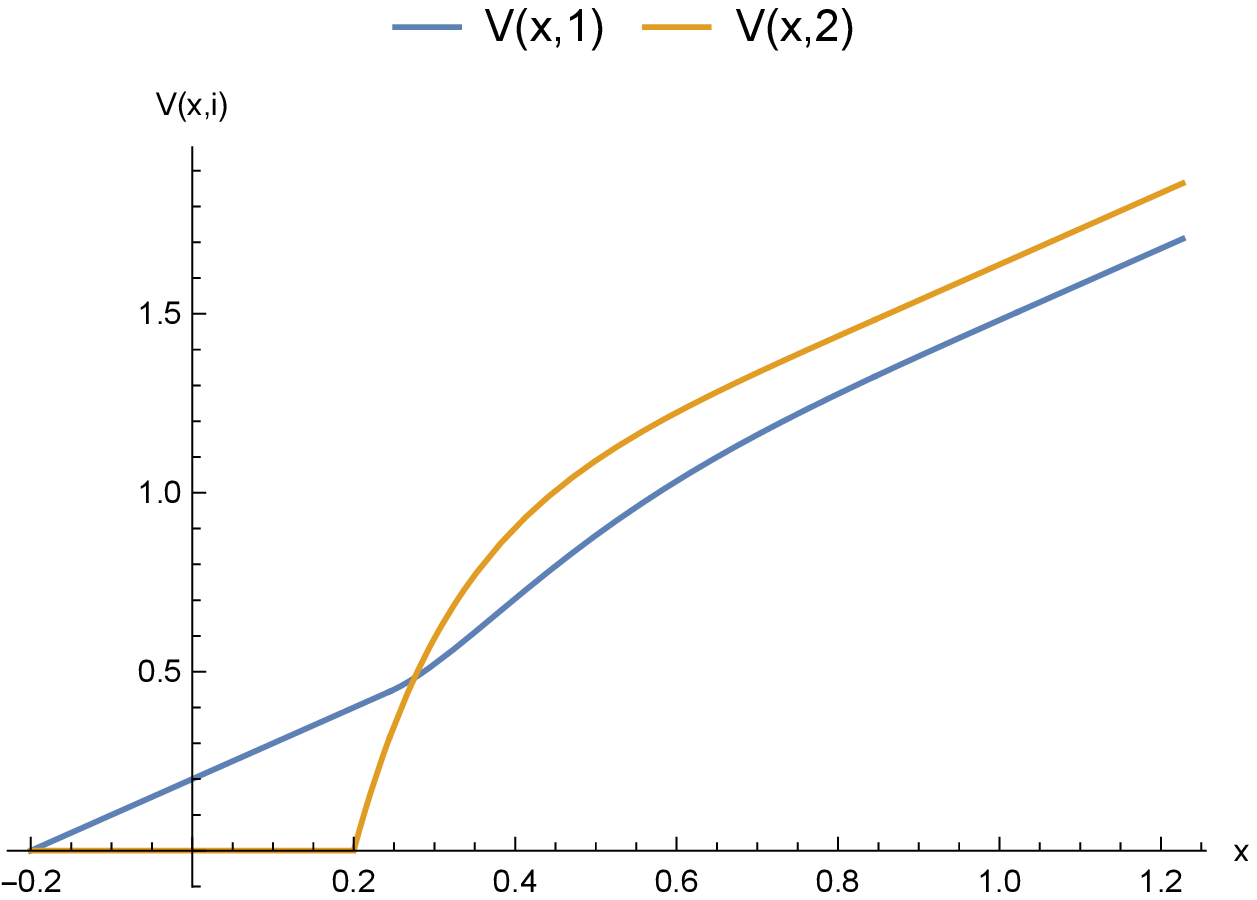}
		\caption{The value function for $\mu_{2}=0.9$}
		\label{valuefunctionc2}
	\end{minipage}
\end{figure}

Next, we choose $\mu_{2}=0.9$ and vary $\mu_{1}, \sigma_1, \lambda_{1}, \theta_{1}, \rho$ individually while keeping the other parameters fixed as in in Table \ref{tab1}. The results are shown in Tables \ref{tab4}-\ref{tab6}.

\begin{table}[htbp]
	\caption{The optimal liquidation-barrier levels for different $\mu_{1},\sigma_{1}$ in Case (C)}
	\label{tab4}
	\begin{minipage}{0.48\textwidth}
		
		\begin{tabular}{cccccc}
		\toprule  
		$\mu_1$	& -0.2 & -0.4 & -0.6 & -0.8  & -1.0\\
		\bottomrule
		
		$d_1$ &0.201 &0.214 &0.229 & 0.245&0.262      \\
		$b_2$ & 0.852&0.851 &0.849 & 0.845&0.840  \\
		$b_1$&0.960&0.983&1.004&1.022&1.036\\
		\bottomrule
	\end{tabular}
	\end{minipage}
	\begin{minipage}{0.4\textwidth}  
		\begin{tabular}{cccccc}
		\toprule  
		$\sigma_1$	& 0.2 & 0.4 & 0.6 & 0.8  & 1.0\\
		\bottomrule
		$d_1$ &0.282 &0.256 &0.235 & 0.217&0.203      \\
		$b_2$ & 0.842&0.844 &0.846 & 0.845&0.843  \\
		$b_1$&0.874&0.970&1.075&1.177&1.274\\
		\bottomrule
	\end{tabular}
	\end{minipage}

\end{table}

From Table \ref{tab4}, we note that $d_1$ and $b_1$ are increasing  when $\mu_{1}$ decreases.  These results suggest that when $\mu_{1}$ is decreasing, it is optimal to keep enough surplus to minimize the probability of liquidation. To proceed with it, the company should raise the dividend barrier $b_1$. On the other hand, the company's profitability is getting worse, which then incentives the company to liquidate earlier.

  In addition, $b_2$ decreases as $\mu_{1}$ decreases. It implies that in regime 2 the company should pay dividends earlier as $\mu_{1}$ decreases, in order to compensate the losses due to the risk of jumping to a worse regime. Moreover, we also notice that $b_1$ is increasing while $d_1$ is decreasing as $\sigma_1$ increases. These results imply that when the volatility of a regime is getting larger, it is better to maintain a high level of the reserve by setting a higher dividend barrier $b_1$ in the more volatile regime to reduce the risk of early ruin. On the other hand, postponing the liquidation option (lower liquidation barrier $d_1$) can compensate the negative impacts of an increase in volatility.

\begin{table}[htbp]
	\caption{The optimal liquidation-barrier levels for different $\lambda_{1},\theta_{1}$ in Case (C)}
	\label{tab5}
	\begin{minipage}{0.48\textwidth}
		
			\begin{tabular}{cccccc}
			\toprule  
			$\lambda_1$	&5 & 7 & 9 & 11  & 13\\
			\bottomrule
			
			$d_1$ &0.371 &0.279 &0.252 & 0.240&0.233      \\
			$b_2$ & 0.784&0.820 &0.839 & 0.850&0.855  \\
			$b_1$&0.819&0.988&1.020&1.021&1.014\\
			\bottomrule
			
		\end{tabular}
	\end{minipage}
	\begin{minipage}{0.4\textwidth}  
		
	\begin{tabular}{cccccc}
		\toprule  
		$\theta_1$	&-0.1 & -0.2 & -0.3 & -0.4  & -0.5\\
		\bottomrule
		
		$d_1$ &0.216 &0.245 &0.281 & 0.332&0.420      \\
		$b_2$ & 0.852&0.845 &0.837 & 0.828&0.815  \\
		$b_1$&1.041&1.022&0.997&0.959&0.882\\
		\bottomrule
		
	\end{tabular}
	\end{minipage}
	
\end{table}

From Table \ref{tab5}, we note that $d_1$ is decreasing as $\lambda_1$ increases. The result suggests that when $\lambda_1$ is increasing, the Markov chain is expected to stay in regime 1 for a shorter time.  Therefore it is optimal to postpone the liquidation option since the chance of jumping to regime 2 is greater. In addition, $b_1$ is a single-picked function of $\lambda_{1}$, which can indeed be explained by the same reasoning applied to discuss the relationship between $b_2$ and $\lambda_{2}$ in Section 5.2. We also notice that $b_1$ in decreasing while $d_1$ is increasing as $\theta_1$ decreases: The company should maximize the probability of hitting the dividend barrier by lowering the latter while exploiting that the probability of ruin is lower. On the other hand,  setting a higher liquidation barrier $d_1$ the company can obtain a larger liquidation value $d_1-\theta_{1}$.

\begin{table}[htbp]

	\caption{The optimal liquidation-barrier levels for different $\rho$ in Case (C)}
	\label{tab6}
	\centering
	\begin{tabular}{cccccc}
		\toprule  
		$\rho$	&0.38 & 0.40 & 0.50 & 0.60  & 0.70\\
		\bottomrule
		
		$d_1$ &0.206 &0.212 &0.245 & 0.294&0.403      \\
		$b_2$ & 0.959&0.937 &0.845 & 0.775&0.720  \\
		$b_1$&1.176&1.149&1.022&0.899&0.725\\
		\bottomrule
		
	\end{tabular}

\end{table}

From Table \ref{tab6}, we note that $d_1$ is increasing while $b_1$ is decreasing when $\rho$ increases. As already discussed, increasing $\rho$ the company becomes more impatient and therefore pays dividends and liquidates earlier.

\subsection{Continuing to increase \boldmath$\mu_{2}$}

We continue to increase $\mu_2$ and consider $\mu_{2} \in [1.10,2.0].$ Numerically solving (\ref{system41}) and (\ref{system42}) in Case (D), we obtain the results for the optimal liquidation-barrier levels {($\theta_1<d_1<\theta_2<b_2<b_1$)} that are shown in Table \ref{tab7}. 
\begin{table}[htbp]
	\caption{ The optimal liquidation-barrier levels with different $\mu_{2}$ in Case (D)}
	\label{tab7}
\centering
	\begin{tabular}{ccccccc}
		\toprule  
		 $\mu_2$	& 1.10  & 1.20 & 1.40 & 1.60  & 1.80&2.0\\
		  \bottomrule
		
	$d_1$ & 0.199 &0.185 &0.161 & 0.140&0.122& 0.106      \\
	$b_2$ & 0.855&0.854 &0.848 & 0.840&0.834& 0.830   \\
           $b_1$&1.084&1.104&1.132&1.152&1.168&1.181\\
		\bottomrule
		
	\end{tabular}
\end{table}
Therein, we note that $d_1$ continues to decrease and $b_1$ continues to increase when $\mu_2$ increases. However, $b_2$ is observed to be decreasing as $\mu_2$ increases: The company exploits the positive trend in regime 2 and increases the dividends payout by shifting the dividend barrier $b_2$ downwards. Notice that here there is no risk of liquidation through a regime change since $d_1<\theta_{2}$. Figures \ref{valuefunctiond1.2}-\ref{valuefunctiond1.8} show the plots of value function in Case (D).

\begin{figure}[htbp]
	\begin{minipage}[b]{0.5\linewidth}
		\centering
		\includegraphics[width=3in]{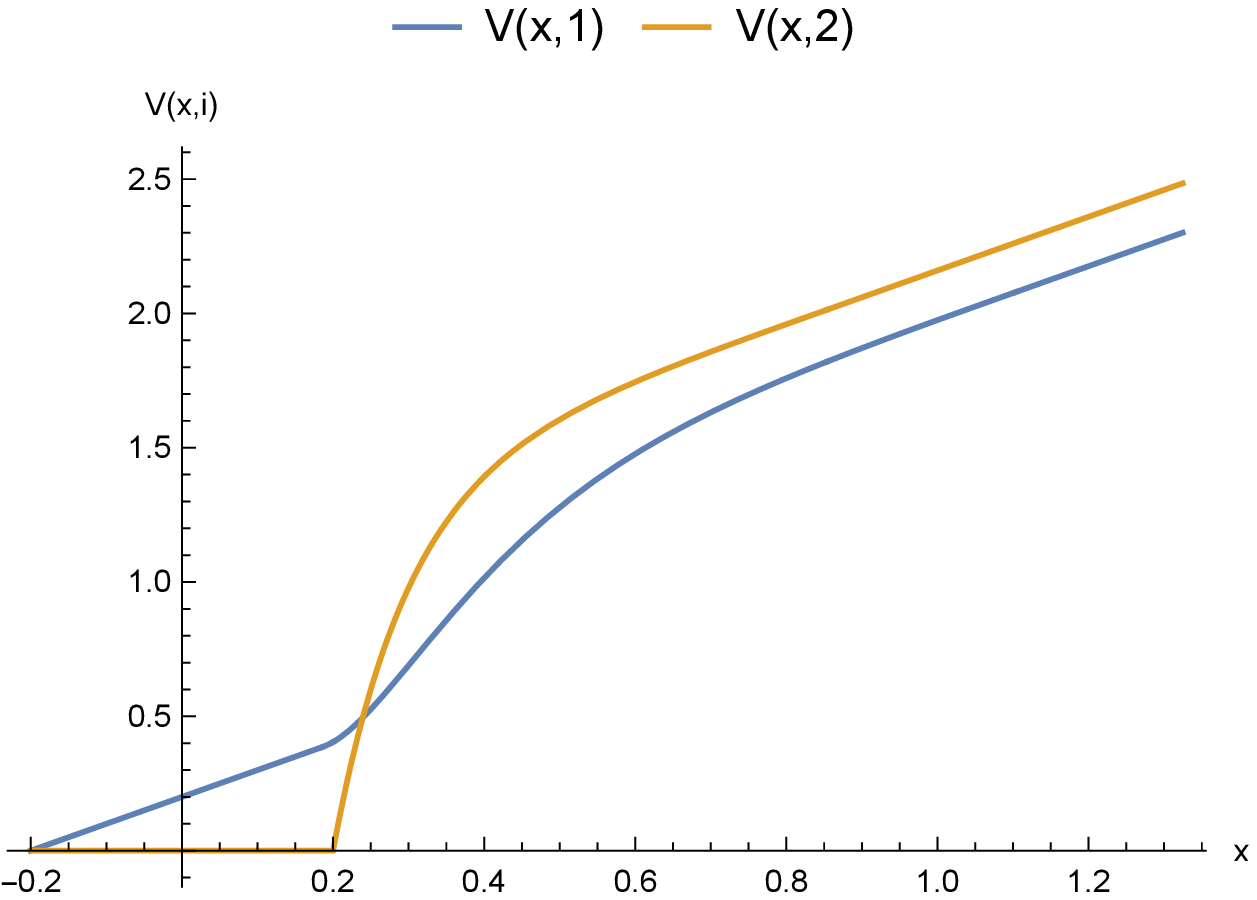}
		\caption{The value function for $\mu_{2}=1.20$}
		\label{valuefunctiond1.2}
	\end{minipage}%
\begin{minipage}[b]{0.5\linewidth}
	\centering
	\includegraphics[width=3in]{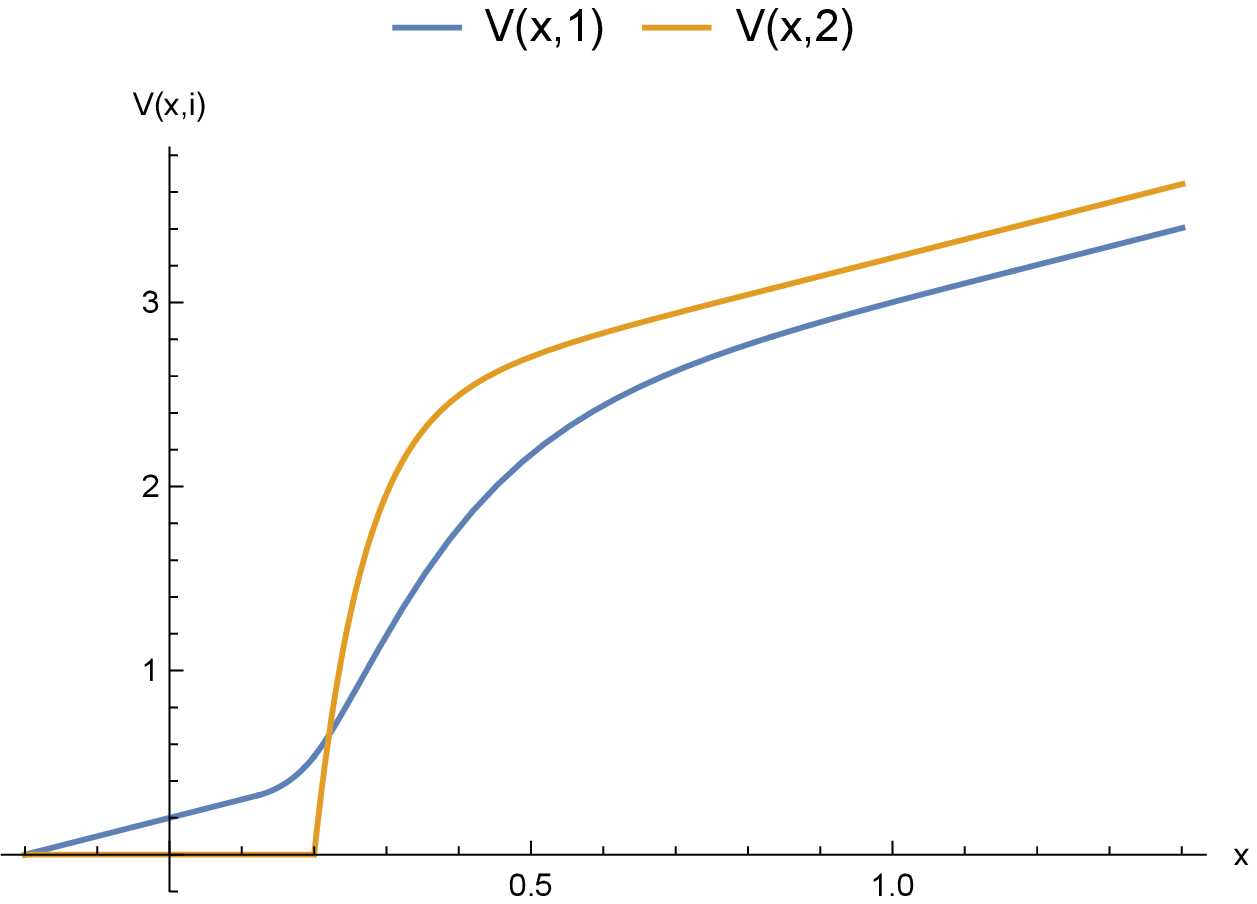}
	\caption{The value function for $\mu_{2}=1.80$}
	\label{valuefunctiond1.8}
\end{minipage}
\end{figure}

\subsection{Comparison to the case \boldmath$\theta_{1}=\theta_{2}$ } We compare our results to the ones obtained in \cite{jiang2012optimal} by assuming that $\theta_1=\theta_2$. Under this assumption, Case (D) disappears and we numerically solve the differential equations of the value functions in Case (B) and Case (C), respectively. We choose $\theta_1=\theta_2=-0.2$ and fix all the other parameters as in Table \ref{tab1}. Then we vary $\mu_2$, and the corresponding results are shown in Table \ref{tab8}.

\begin{table}[htbp]
	\caption{ The optimal dividend policies}
	\label{tab8}
	\centering

	\begin{tabular}{cccccccccccccc}
		\toprule  
		$\mu_2$	&0.1   & 0.2 & 0.35  & 0.36  && 0.37&0.38&0.39 && 0.40 &0.45&0.50       \\
		\bottomrule
		
		$d_1$ &           &         &             &  &&         0.047      &-0.050&-0.086&&   -0.109      &-0.170&-0.198  \\
		
		$b_1$&         &           &          &         &&       0.133              &0.211&0.249&&   0.279 &0.378&0.445\\
		$b_2$ &  -0.014       &0.120         & 0.236         & 0.242   &&        0.247     &0.254&0.263&&        0.272 &0.313& 0.349   \\
		\bottomrule
		
	\end{tabular}
	
\end{table}

From Table \ref{tab8}, we note that $b_2$ is decreasing when $\mu_2$ decreases. This suggests that, when $\mu_{2}>0$ decreases, it is optimal to reduce the dividend barrier $b_2$ so to maximize the probability of hitting it, even if this implies a higher risk of bankruptcy. Table \ref{tab8} also shows that $d_1$ is decreasing while $b_1$ is increasing as $\mu_2$ increases, which are behaviors similar to those in Tables \ref{tab3}-\ref{tab7}.

From Figures \ref{valuefunctiontheta1=2b}-\ref{valuefunctiontheta1=2c}, it is clear that $V(x,2)\geq V(x,1)$, which is consistent with the result in Section 5.2 of \cite{jiang2012optimal} and actually expected being $ \mu_1=-0.8$ and $\mu_2>0$.

\begin{figure}[htbp]
	\begin{minipage}[b]{0.5\linewidth}
		\centering
		\includegraphics[width=3in]{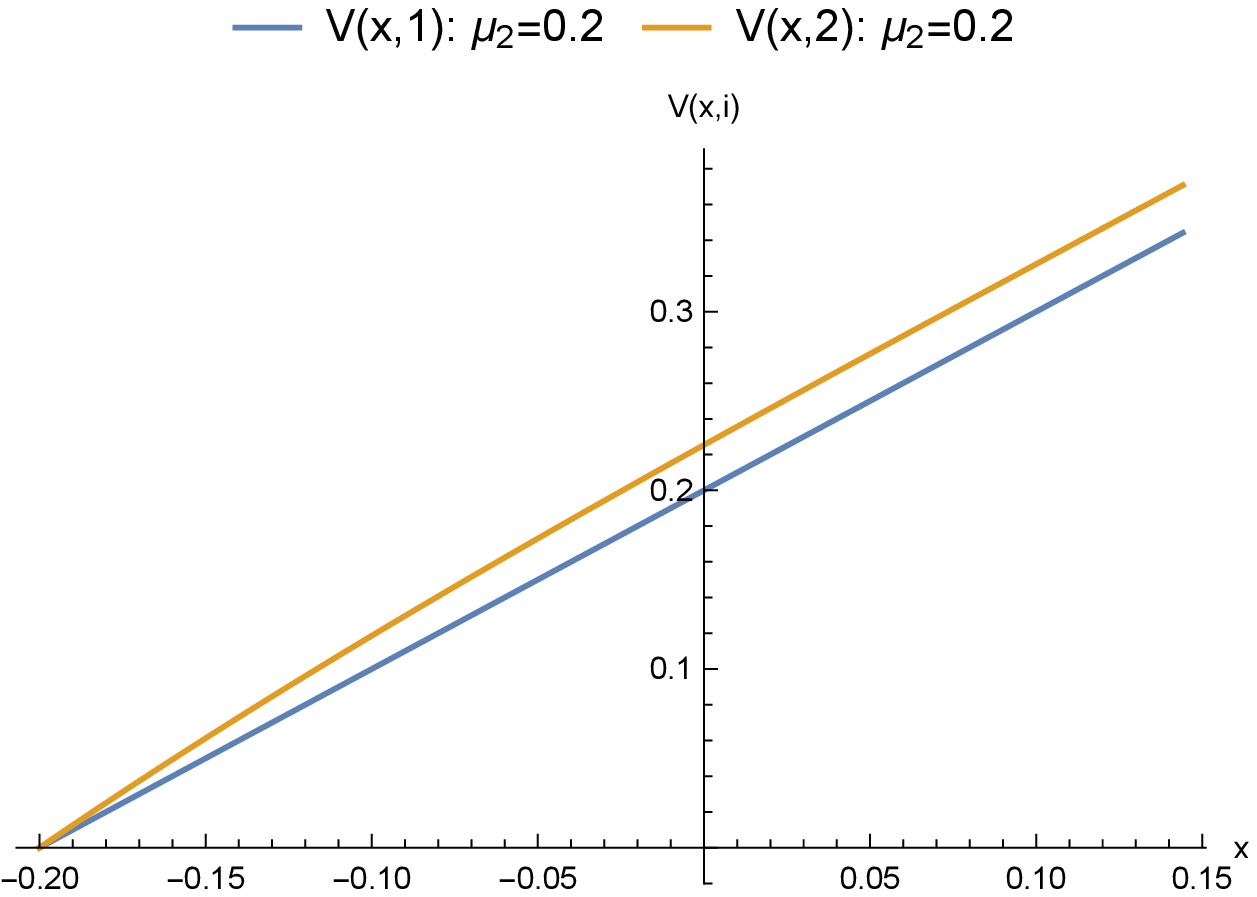}
		\caption{The value function for $\theta_1=\theta_2=-0.2$}
		\label{valuefunctiontheta1=2b}
	\end{minipage}%
\begin{minipage}[b]{0.5\linewidth}
	\centering
	\includegraphics[width=3in]{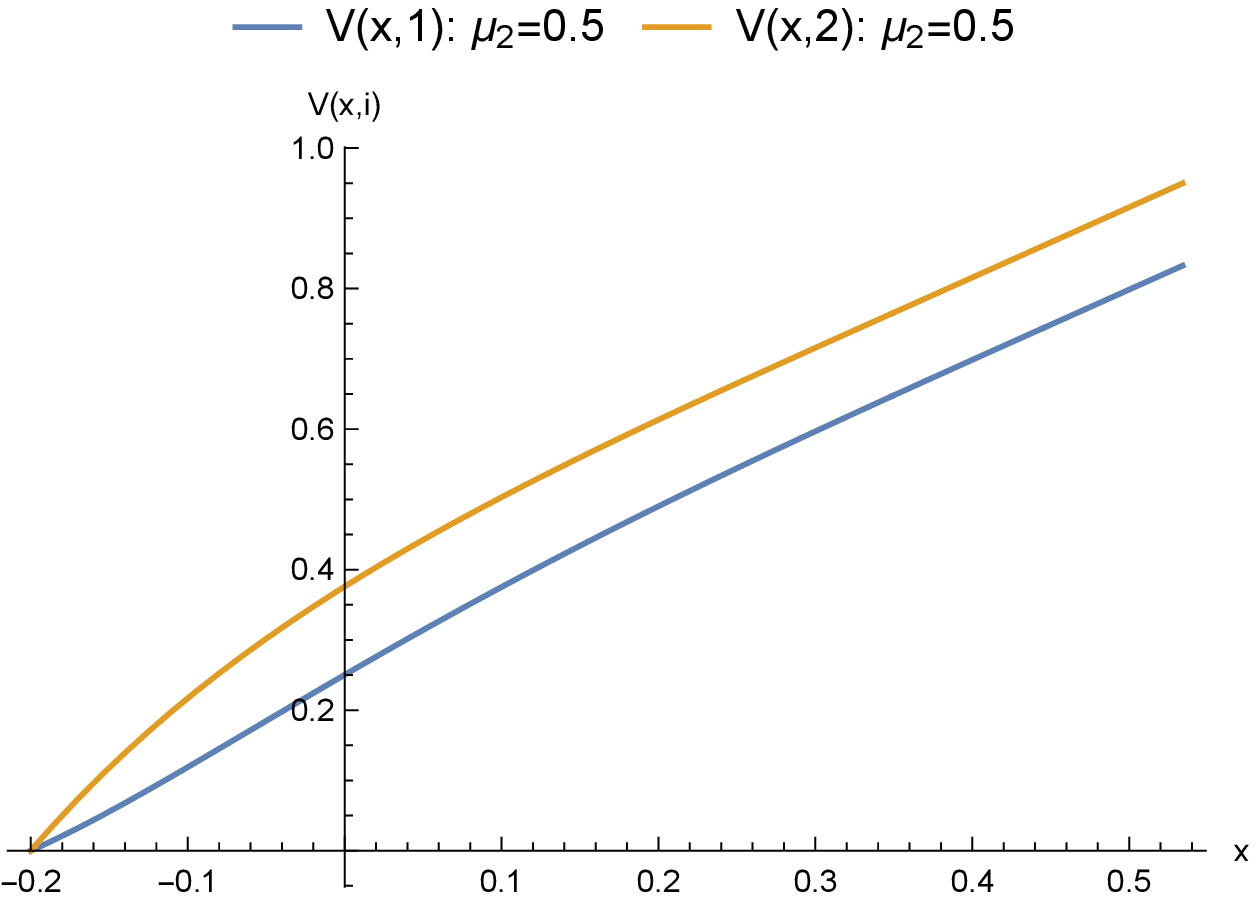}
	\caption{The value function for $\theta_1=\theta_2=-0.2$}
	\label{valuefunctiontheta1=2c}
\end{minipage}
\end{figure}

\section{Conclusions}

We have provided a analytical and numerical study of the optimal dividend problem of a company whose bankruptcy level and cash surplus' dynamics are modulated by a two-state continuous-time Markov chain. This models a macroeconomic shock that not only affects the company's profitability and the market's fluctuations, but also triggers bankruptcy differently across the two regimes. For example, during a crisis like the recent Covid-epidemics, governmental measures like the postponement of taxes' payments effectively delay the bankruptcy of insolvent companies.

Regime-dependent bankruptcy level $\theta_{i}$ and cash surplus' drifts $\mu_i$ --- which are allowed to assume different signs across the two regimes --- leads to a rich structure of the optimal policy, which can be of \emph{barrier-type} or of \emph{liquidation-barrier-type}. For example, we prove that, even if $\mu_i<0 \ \text{for all}\ i=1,2,$ the fact that $\theta_{1}<\theta_{2}$ induces the company to continue business in regime two, rather than bankrupting immediately, in order to strategically exploit a regime change to regime one. A detailed numerical study complements the theoretical analysis and allows to draw interesting economic implications about the sensitivity of the optimal strategy with respect to the model's parameters.

There are many directions towards which this work can be extended. For example, it would be interesting to allow for partial observation of the underlying macroeconomic shock, to study the role of capital injections or to consider dynamic optimal reinsurance strategies of the company. These and other aspects are left for future research.

\appendix
\setcounter{subsection}{0}
\renewcommand\thesubsection{A.\arabic{subsection}}
\section{Proofs}
\setcounter{equation}{0}
\renewcommand\theequation{A.\arabic{equation}}
\setcounter{lemma}{0}
    \renewcommand{\thelemma}{\Alph{section}.\arabic{lemma}}

\subsection{Proof of Theorem \ref{theorem4.1}}
\begin{proof}\label{ptheorem4.1}
	
	{By construction, the boundary conditions are satisfied and} we have that $w'(x,i)=1, x\in [\theta_{i},\infty).$ Therefore, we need to show that  
	\begin{align}\label{3.2-1}
\begin{aligned}
	(i)\	(\mathcal{L} -\rho)w(x,1)\leq 0, \ x> \theta_{1},\\
	(ii)\ (\mathcal{L} -\rho)w(x,2)\leq 0, \ x> \theta_{2},
\end{aligned}
	\end{align}
	hold if and only if $\mu_2 \leq (\theta_1-\theta_2)\lambda_2$.	

\setlength{\parskip}{0.5em}
The proof of (i) in (\ref{3.2-1}) is as follows. When $x \in [\theta_1,\theta_2]$,  since $w(x,2)=0,$ the left-hand side of (i) in (\ref{3.2-1}) becomes	$	\frac{1}{2}\sigma_1^2w''(x,1) +\mu_1w'(x,1) -(\lambda_{1}+\rho)w(x,1) +\lambda_{1}w(x,2) =\mu_1-(\lambda_{1}+\rho) (x-\theta_1)$, and the right-hand side of the above equation is negative since $\mu_1 \leq 0$.
		
		 When $x \in (\theta_2,\infty),$ because $w(x,1)=x-\theta_{1}, w(x,2)=x-\theta_{2}$, the left-hand side of (i) in (\ref{3.2-1}) satisfies $\mu_1-(\lambda_{1}+\rho) (x-\theta_1)+\lambda_{1}(x-\theta_2)=\mu_1+\lambda_{1}(\theta_1-\theta_{2})-\rho(x-\theta_1).$ But now $ \mu_1+\lambda_{1}(\theta_1-\theta_{2})-\rho(x-\theta_1) \leq 0$ for all $x \in(\theta_2, \infty)$ if and only if $\mu_1 \leq (\rho+\lambda_1)(\theta_2-\theta_1).$ However, the latter inequality is true because $\mu_1 \leq0.$
	
	The proof of (ii) in (\ref{3.2-1}) proceeds as follows. When $x \in (\theta_2,\infty), w(x,1)={x-\theta_1}$. Then (ii) in (\ref{3.2-1}) satisfies $	\frac{1}{2}\sigma_2^2w''(x,2) +\mu_2w'(x,2) -(\lambda_{2}+\rho)w(x,2) +\lambda_{2}w(x,1)=\mu_2-\lambda_{2}(\theta_{1}-\theta_{2})-\rho(x-\theta_{2}). $ But $\mu_2-\lambda_{2}(\theta_{1}-\theta_{2})-\rho(x-\theta_{2}) \leq 0$ if and only if $\mu_2 \leq (\theta_1-\theta_2)\lambda_2.$ 
		
	 We have thus verified that $w(x,i), i=1,2,$ is the solution to HJB equation (\ref{HJB}) and the dividend strategy leading to payoff $w(x,i)$ consists of a single jump at initial time of size $x-\theta_i$.  Such a strategy is of the form $D^{b,w}$ in Definition \ref{dividend2}, with $b_i=\theta_i$. Hence, $w(x,i),i=1,2,$ in (\ref{1wx1}) and (\ref{1wx2}) is indeed the value function if and only if $\mu_2 \leq (\theta_1-\theta_2)\lambda_2$ by Theorem \ref{verificationtheorem}.

\end{proof}

\subsection{Proof of Lemma \ref{lemma4.1} and of Theorem \ref{theorem4.2}  }
\subsubsection{Proof of Lemma \ref{lemma4.1}}
\begin{proof}\label{plemma4.1}
	From the fourth equation in (\ref{system1}), we have 
\begin{equation}\label{plemma3.2-1}
C_1 = \frac{-C_2\alpha_8^2}{\alpha_7^2}e^{(\alpha_8-\alpha_7)b_2}.
\end{equation}
Combing (\ref{plemma3.2-1}) with the first equation in (\ref{system1}), we find
	\begin{equation}\label{plemma3.2-2}
	C_2\bigg[\frac{\alpha_8^2}{\alpha_7^2}e^{(\alpha_8-\alpha_7)b_2}-e^{(\alpha_8-\alpha_7)\theta_{2}}\bigg]= Pe^{-\alpha_7\theta_{2}},
	\end{equation}
	where $P:=\frac{\mu_2 \lambda_2}{(\lambda_2+\rho)^2}+\frac{\lambda_2(\theta_2-\theta_1)}{\lambda_2+\rho}.$ On the other hand, combine (\ref{plemma3.2-1}) with the third equation in (\ref{system1}), one obtains 
	\begin{equation}\label{plemma3.1-3}
	C_2e^{\alpha_8b_2}\alpha_8(\alpha_7-\alpha_8) = \frac{\alpha_7\rho}{\rho+\lambda_{2}}>0
	\end{equation}
	Therefore, if there exists a solution $b_2> \theta_{2}$, then it must be
\begin{align}\label{plemma3.1-4}
C_2<0, \quad C_1>0.
	\end{align}

	 The equation for free-boundary $b_2$ is then obtained by combining  (\ref{plemma3.2-2}) with (\ref{plemma3.1-3}), which indeed gives
	\begin{equation*}
	\frac{\alpha_8^2}{\alpha_7^2}e^{-\alpha_7b_2} -e^{(\alpha_8-\alpha_7)\theta_2-\alpha_8b_2} = \frac{P}{\rho\alpha_7}\bigg[(\rho+\lambda_{2})(\alpha_7\alpha_8-\alpha_8^2)e^{-\alpha_7\theta_2}\bigg].
	\end{equation*}
	We now show that the above equation admits a unique solution $b_2>\theta_2$ if and only of $\mu_2>\lambda_2(\theta_1-\theta_2).$ Define $F(x):=\frac{\alpha_8^2}{\alpha_7^2}e^{-\alpha_7x} -e^{(\alpha_8-\alpha_7)\theta_2-\alpha_8x} - \frac{P}{\rho\alpha_7}[(\rho+\lambda_{2})(\alpha_7\alpha_8-\alpha_8^2)e^{-\alpha_7\theta_2}]$, and note that $F(+\infty)<0$ and $F'(x)=\frac{\alpha_8^2}{\alpha_7^2}(-\alpha_7)e^{-\alpha_7x}+\alpha_8e^{(\alpha_8-\alpha_7)\theta_2-\alpha_8x}<0.$ On the other hand,
	\begin{equation*}
	\begin{aligned}
	F(\theta_2)= e^{-\alpha_7\theta_2}\bigg[\frac{\rho\alpha_8^2-\rho\alpha_7^2-\alpha_7(\alpha_7\alpha_8-\alpha_8^2)P(\lambda_{2}+\rho)}{\rho\alpha_7^2}\bigg].
	\end{aligned}
	\end{equation*}
	
	 Since $F'(x)<0, F(+\infty)<0$, it follows that there exists a unique solution $b_2> \theta_{2}$ such that $F(b_2)=0$ if and only if  $F(\theta_{2})> 0$; i.e. if and only if $\rho\alpha_8^2-\rho\alpha_7^2-\alpha_7(\alpha_7\alpha_8-\alpha_8^2)P(\lambda_{2}+\rho)> 0$. Using that $\alpha_7+\alpha_8 = -2\mu_2/ \sigma_2^2$ and $\alpha_7\alpha_8= -2(\lambda_{2}+\rho)/\sigma_2^2,$ and performing standard manipulations, one can show that
\begin{align*}
\rho\alpha_8^2-\rho\alpha_7^2-\alpha_7(\alpha_7\alpha_8-\alpha_8^2)P(\lambda_{2}+\rho)> 0 \Longleftrightarrow  \mu_2>\lambda_{2}(\theta_{1}-\theta_{2}),
\end{align*}
which completes the proof.
	
\end{proof}

\subsubsection{Proof of Theorem \ref{theorem4.2}}

\begin{proof}\label{ptheorem4.2}
	From Lemma \ref{lemma4.1}, there exists a unique free boundary $b_2>\theta_{2}$. Then we show that $w(x,i)$ in (\ref{2wx1}) and (\ref{2wx2}) satisfy the HJB equation (\ref{HJB}) under the conditions in Theorem \ref{theorem4.2}.  {By construction, the boundary conditions are satisfied and} in (\ref{2wx1}) and (\ref{conb}), we have that $w'(x,1)=1$ for all  $ x\in [\theta_{1},\infty), w'(x,2)=1$ for all $ x\in[b_2,\infty)$ and $(\mathcal{L} -\rho)w(x,2)= 0,  x\in (\theta_{2},b_2).$ Therefore, we only need to show that 	
\begin{align}\label{3.3-1}
\begin{aligned}
	&(i)\ w'(x,2)\geq 1, \ &\forall x\in (\theta_{2},b_2),\\
	&(ii)\ (\mathcal{L} -\rho)w(x,2) \leq 0,\ &\forall  x\in [b_2,\infty),\\
	&(iii)\ (\mathcal{L} -\rho)w(x,1) \leq 0,\  &\forall x\in [\theta_{1},\infty).
\end{aligned}
	\end{align}

\textbf{	\textsl{Step 1.}} The proof of (i) in (\ref{3.3-1}) is as follows. When $x \in (\theta_2,b_2), w(x,2)=C_1e^{\alpha_7x}+C_2e^{\alpha_8x}+\frac{\lambda_{2}\mu_2}{(\rho+\lambda_{2})^2}+\frac{\lambda_{2}(x-\theta_1)}{\rho+\lambda_{2}}. $  We compute $w'''(x,2)= C_1\alpha_7^3e^{\alpha_1x}+C_2\alpha_8^3e^{\alpha_2x}>0$, where we have used the fact that $C_1>0$ and $C_2<0$  (cf. (\ref{plemma3.1-4})). Then $w''(x,2)$ is strictly increasing. Since $w''(b_2,2)=0,$ we have $ w''(x,2)\leq 0, x\in(\theta_{2},b_2)$, hence $w'(x,2)$ is decreasing when $x\in(\theta_{2},b_2).$ Since $w'(b_2,2)=1$, it follows that $w'(x,2)\geq1$ for all $x\in(\theta_{2},b_2).$
	
\setlength{\parskip}{0.5em}	
\textbf{	\textsl{Step 2.}} The proof of (ii) in (\ref{3.3-1}) is as follows. We have to show that	$\frac{1}{2}\sigma_2^2w''(x,2) +\mu_2w'(x,2) -(\lambda_{2}+\rho)w(x,2) +\lambda_{2}w(x,1) \leq 0,\  \forall  x\in [b_2,\infty).$ Define $F(x):=\frac{1}{2}\sigma_2^2w''(x,2) +\mu_2w'(x,2) -(\lambda_{2}+\rho)w(x,2) +\lambda_{2}w(x,1)  $. Since $w'(x,2)=1=w'(x,1)$ on $[b_2,\infty)$, then $F'(x)=-\rho<0$. Moreover, by continuity, $F(b_2)= F(b_{2}-) = 0.$ Therefore, $F(x) \leq 0$ for all $x\in [b_2,\infty)$ and (ii) in (\ref{3.3-1}) is satisfied.

\textbf{	\textsl{Step 3.}}  The proof of (iii) in (\ref{3.3-1}) is as follows. We have to show that
		\begin{equation}\label{3.3-2}
		\frac{1}{2}\sigma_1^2w''(x,1) +\mu_1w'(x,1) -(\lambda_{1}+\rho)w(x,1) +\lambda_{1}w(x,2) \leq0, \ \forall x\in [\theta_{1},\infty).
		\end{equation}

		For $x\in[\theta_1,\theta_2],$ since $w(x,2)=0$, left-hand side of  (\ref{3.3-2}) rewrites as $\frac{1}{2}\sigma_1^2w''(x,1) +\mu_1w'(x,1) -(\lambda_{1}+\rho)w(x,1) +\lambda_{1}w(x,2) =\mu_1-(\lambda_{1}+\rho) (x-\theta_1).$ However, since $\mu_1 \leq 0$, this shows that (iii) in (\ref{3.3-2}) holds for $x \in[\theta_1,\theta_2].$

		Consider now $x \in (\theta_2,b_2). $ Since $w'(x,1)=1,$ the left-hand side of (\ref{3.3-2}) rewrites as 
		$\frac{1}{2}\sigma_1^2w''(x,1) +\mu_1w'(x,1) -(\lambda_{1}+\rho)w(x,1) +\lambda_{1}w(x,2) =\mu_1-(\lambda_{1}+\rho)w(x,1)+\lambda_{1}w(x,2). $ Define $G(x):=\mu_1-(\lambda_{1}+\rho)w(x,1)+\lambda_{1}w(x,2).$ Then $G''(x)=\lambda_{1}w''(x,2) \leq 0,$ since $w''(x,2)$ has been shown to be negative for $x\in(\theta_2,b_2)$ in Step 1 above.  Hence,  $G'(x) $ is decreasing for all $x \in(\theta_{2},b_2).$ On the other hand, $G(\theta_{2}+)=G(\theta_{2})=\mu_1-(\lambda_1+\rho)(\theta_2-\theta_1)<0$ and $G'(b_{2}-)=G'(b_2)=-\rho<0.$ It thus follows that in order to show that $G(x)\leq 0$ for all $x\in (\theta_2,b_2)$, the following equivalent conditions should be satisfied:
		\begin{align*}
		& G(x)\leq0, \ \forall x\in(\theta_{2},b_2) \Longleftrightarrow \max_{x\in (\theta_{2},b_2)}G(x)\leq0 \\
&\Longleftrightarrow \emph{either} \ G'(\theta_{2}+)\leq 0 \
		\emph{or} \ G(x_0)\leq 0\ \text{and}\ G'(\theta_{2}+)>0, \text{where} \  x_0 \ \text{satisifies} \ G'(x_0)=0. \\
		& \Longleftrightarrow  \emph{either} \ w'(\theta_{2}+,2)\leq \frac{\lambda_{1}+\rho}{\lambda_{1}} \ \emph{or} \  w'(\theta_{2}+,2)> \frac{\lambda_{1}+\rho}{\lambda_{1}}\ \text{ and}\ G(x_0)\leq 0,
		\text{where} \ w'(x_0,2) = \frac{\lambda_{1}+\rho}{\lambda_{1}}.
		\end{align*}
		Notice that it indeed exists a unique $x_0\in(\theta_2,b_2)$ satisfying $ w'(x_0,2) = (\lambda_{1}+\rho)/\lambda_{1}.$ In fact, since $G'(\theta_2+)>0, G'(b_2)<0$ and $G''(x)\leq 0, x\in (\theta_2,b_2)$, then there exists a unique $x_0\in(\theta_2,b_2)$ satisfying $G'(x_0)=0$; i.e.,  $ w'(x_0,2) = (\lambda_{1}+\rho)/\lambda_{1}.$

Finally, consider $x \in [b_2,\infty)$ and define $F(x):=\frac{1}{2}\sigma_1^2w''(x,1) +\mu_1w'(x,1) -(\lambda_{1}+\rho)w(x,1) +\lambda_{1}w(x,2)  $.   Because $F'(x)=-\rho<0,$ and $F(b_2)= F(b_{2}-) \leq 0$ under the above equivalent conditions, we have $F(x)\leq0$ for all $x \in [b_2,\infty)$. Hence, $(\mathcal{L} -\rho)w(x,1) \leq 0,\  \forall x\in [\theta_{1},\infty),$ and this shows (iii) in (\ref{3.3-1}).

		 We have verified that $w(x,i), i=1,2,$ is the solution of HJB equation under a set of easily verifiable equivalent conditions. Finally,  we notice that the barrier strategy $D^{b,w}$ as in Definition \ref{dividend2} associated to $w(x,i)$ in  (\ref{2wx1}) and (\ref{2wx2}) is such that (\ref{dividendequation2}) and (i)-(iii) of Definition \ref{dividend2} hold true.  Therefore, $w(x,i),i=1,2,$ is indeed the value function by Theorem \ref{verificationtheorem} and $D^{b,w}$ is an optimal policy.
\end{proof}

\subsection{Proof of Theorem \ref{theorem4.3}}\label{ptheorem4.3}
 Throughout this proof we denote by $w^{(n)}$ the $n$-th derivative of the function $w, n\geq 4$, with respect to $x$. Before we proceed with the proof, we need the following lemmata.

\begin{lemma}\label{regularity1}
 Set $I_i:=(\theta_i,\infty), i=1,2.$ Then the function $w(x,i)$ in (\ref{3wx1}) and (\ref{3wx2}) is such that $w(\cdot,1) \in C^1(I_1) \ \cap  \ C^2(I_1 \backslash \{d_1\})\ \cap\ C^{\infty}(I_1 \backslash \{d_1,b_2,b_1\})$ and $w(\cdot,2) \in  C^2(I_2)  \cap C^4(I_2 \backslash \{d_1,b_2\}) \cap C^{\infty}(I_2 \backslash \{d_1,b_1,b_2\}).$ 
\end{lemma}
\begin{proof}
	First, note that by construction $w(\cdot,1) \in C^1(I_1) \ \cap  \ C^2(I_1 \backslash \{d_1,b_2\})\ \cap\ C^{\infty}(I_1 \backslash \{d_1,b_2,b_1\})$ and $ w(\cdot,2) \in C^1(I_2) \ \cap \ C^2(I_2 \backslash \{d_1\}) \ \cap\ C^{\infty}(I_2 \backslash \{d_1,b_2\}).$
	
Also, for $x \in (d_1,b_1)$, $w(x,1)$ satisfies $	\frac{1}{2}\sigma_1^2w''(x,1) +\mu_1w'(x,1) -(\lambda_{1}+\rho)w(x,1) +\lambda_{1}w(x,2) =0$ by (\ref{4.8}) and (\ref{4.11}). Since $w'(b_2,1)=w'(b_{2}-,1), w(b_2,1)=w(b_{2}-,1), w(b_2,2)=w(b_{2}-,2)$, it follows that $w''(b_2,1)=w''(b_{2}-,1).$ Similarly, when $x\in(\theta_2,b_2)$, by (\ref{4.7}) and (\ref{4.9}), we obtain that $ w''(d_1,2)=w''(d_{1}-,2)$. Therefore, $  w(\cdot,1) \in C^1(I_1) \ \cap  \ C^2(I_1 \backslash \{d_1\})\ \cap\ C^{\infty}(I_1 \backslash \{d_1,b_2,b_1\}) $  and $                w(\cdot,2) \in  \ C^2(I_2) \ \cap\ C^{\infty}(I_2 \backslash \{d_1,b_2\}). $
	
	When $\theta_{2}<x <b_2,$ we have by (\ref{4.7}) and (\ref{4.9}) that $	\frac{1}{2}\sigma_2^2w''(x,2) +\mu_2w'(x,2) -(\lambda_{2}+\rho)w(x,2) +\lambda_{2}w(x,1) =0. $ Since $w(\cdot,1) \in C^2(I_2\backslash \{d_1\}), $ we see that $w(\cdot,2) \in C^4((\theta_{2},b_2)\backslash \{d_1\}) $. Hence, we obtain that $w(\cdot,2) \in  C^2(I_2)  \cap C^4(I_2 \backslash \{d_1,b_2\}) \cap C^{\infty}(I_2 \backslash \{d_1,b_1,b_2\}).$

\end{proof}

\begin{lemma}\label{zero2}
Suppose that $w''(b_2,1)=w''(b_{2}-,1)<0, w'(b_2,1)\geq 1.$	Define the equation $f(x):= w''(x,1) =\widehat{ C_3}\alpha_3^2e^{\alpha_3x}+\widehat{C_4}\alpha_4^2e^{\alpha_4x}+\widehat{C_5}\alpha_5^2e^{\alpha_5x}+\widehat{C_6}\alpha_6^2e^{\alpha_6x}, x \in (d_1,b_2)$. Then $f$ admits a real zero and, in particular, at most 4 zeros when $x\in(d_1,b_2).$
\end{lemma}
\begin{proof}
	Firstly,  we show that there exists zeros. We know that $w'(d_{1}+,1)=1$ and by assumption that $ w''(b_2,1)=w''(b_{2}-,1)<0, w'(b_2,1)\geq 1$. Since $w''(x,1)$ is continuous when $x \in (d_1,b_2),$ if there does not exist zeros when $x\in(d_1,b_2),$ then $w''(x,1)<0$ when $x\in (d_1,b_2)$. Hence $w'(x,1)$ is strictly decreasing, which is a contradiction with the fact that  $w'(d_1,1)=1$ and $ w'(b_2,1)\geq1.$ Therefore, there exists at least a zero of $f$ when $x\in (d_1,b_2).$
	
	Then we prove that there exists at most 4 zeros.
	Letting $\beta_3:=\widehat{C_3}\alpha_3^2, \beta_2:=\widehat{C_4}\alpha_4^2, \beta_1:=\widehat{C_5}\alpha_5^2, \beta_0:=\widehat{C_6}\alpha_6^2$ and $f_3:=e^{\alpha_3}, f_2:=e^{\alpha_4}, f_1:=e^{\alpha_5}, f_0:=e^{\alpha_6}$, we can write $f(x)=\beta_0f_0^x+\beta_1f_1^x+\beta_2f_2^x+\beta_3f_3^x=\sum_{i=0}^{3}\beta_if_i^x$. Since $f_0>f_1>f_2>f_3>0, \beta_0\neq0,$ then $f(x)$ has at most 4 zeros when $x\in(d_1,b_2) $ by Theorem 1 in \cite{tossavainen2006zeros}.

\end{proof}

\begin{proof}[\textbf{Proof of Theorem \ref{theorem4.3}}]

{By construction, the boundary conditions are satisfied.} The rest of the  proof is organized in four steps.

\setlength{\parskip}{0.5em}
\textbf{	\textsl{Step 1.}}  Firstly, we show that if $\mu_2\geq0$ and $w'(\theta_{2}+,2)\geq1$, then $w(x,2)$ is concave. We know that $w''(x,2)=0$ for all $x\in[b_2,\infty).$ By Lemma \ref{regularity1}, we also know that  $w(\cdot,2) \in C^4(I_2 \backslash \{d_1, b_2\})$. Applying It\^o's formula to $ (e^{-\rho t}w''(X^D_{t},2))_{t\in [0,\xi]}$, where $\xi:=\xi_1 \wedge \xi_2$ with $\xi_1:=\inf\{t\geq 0:X_t^D \geq b_{2}\} $ and $ \xi_2:= \inf \{t\geq 0:X^D_t\le \theta_2\}$, we have 
	\begin{align*}
	\mathbb{E}_{x,2}\big[e^{-\rho \xi}w''(X_\xi^D,2)\big]&=w''(x,2)+\mathbb{E}_{x,2}\bigg[\int_{0}^{\xi}e^{-\rho t}(\frac{1}{2}\sigma_2^2w^{(4)}(X^D_t,2)+ \mu_{2}w'''(X^D_t,2)\\&-(\rho+\lambda_{2}) w''(X^D_t,2)+ \lambda_{2}w''(X^D_t,1))\bigg]\textrm{d}t.
	\end{align*}	
Since $(\mathcal{L}-\rho)w''(X^D_t,2)=0$ for all $t \leq \xi$, it follows that
	\begin{align*}
	\mathbb{E}\big[e^{-\rho \xi_1}w''(X^D_{\xi_1},2)\mathbb{I}_{\{\xi_1<\xi_2\}}\big]+\mathbb{E}\big[e^{-\rho \xi_2}w''(X^D_{\xi_2},2)\mathbb{I}_{\{\xi_1 \geq \xi_2\}}\big]= w''(x,2).
	\end{align*}
	Since $w''(x,2)=0$ for all $x \in [b_2, \infty)$, we have $\mathbb{E}\big[e^{-\rho \xi_1}w''(X^D_{\xi_1},2)\mathbb{I}_{\{\xi_1<\xi_2\}}\big]=0.$ Thus, we only need to compute $w''(\theta_{2}+,2).$
Notice that 
\begin{align}\label{a13}
w(\theta_{2},2)=C_{1}e^{\alpha_7\theta_{2}}+C_{2}e^{\alpha_8\theta_{2}}+\frac{\lambda_{2}\mu_2}{(\rho+\lambda_{2})^2}+\frac{\lambda_{2}(\theta_{2}-\theta_{1})}{\rho+\lambda_{2}}=0,
\end{align}
since $\mu_2\geq 0$, we have $\frac{\lambda_{2}\mu_2}{(\rho+\lambda_{2})^2}+\frac{\lambda_{2}(\theta_{2}-\theta_{1})}{\rho+\lambda_{2}}>0$, and (\ref{a13}) implies  
\begin{align}\label{a14}
C_{1}e^{\alpha_7\theta_2}+C_{2}e^{\alpha_8\theta_2}<0.
\end{align}
On the other hand, since $w'(\theta_{2}+,2)\geq1$ by assumption, we see that $w'(\theta_{2}+,2)= C_{1}\alpha_7e^{\alpha_7\theta_2}+C_{2}\alpha_8e^{\alpha_8\theta_2}+\frac{\lambda_{2}}{\rho+\lambda_{2}}\geq1 $, so that
\begin{align}\label{a15}
C_{1}\alpha_7e^{\alpha_7\theta_2}+C_{2}\alpha_8e^{\alpha_8\theta_2}>0.
\end{align}

Multiplying (\ref{a14}) by $\alpha_7>0$, substituting the result into (\ref{a15}) and recalling that $\alpha_8<0$, we have $C_2<0$ and, from (\ref{a14}), $C_1<-C_2e^{(\alpha_8-\alpha_7)\theta_2}$. Then 
\begin{align*}
w''(\theta_{2}+,2)&= C_{1}\alpha_7^2e^{\alpha_7\theta_2}+C_{2}\alpha_8^2e^{\alpha_8\theta_2}< -C_2e^{(\alpha_8-\alpha_7)\theta_2}\alpha_7^2e^{\alpha_7\theta_2}+C_{2}\alpha_8^2e^{\alpha_8\theta_2} \\
&=C_2(\alpha_8^2-\alpha_7^2)e^{\alpha_8 \theta_2}\leq0.
\end{align*}
where we have used the fact $\alpha_8^2-\alpha_7^2\geq0 $ if and only if $ \mu_2\geq0.$ Therefore, if $\mu_2\geq 0$ and $w'(\theta_{2}+,2)\geq 1$, then $w''(x,2) \leq 0, x\in(\theta_2,b_2)$. Hence, $w(x,2)$ is concave.

\setlength{\parskip}{0.5em}
\textbf{	\textsl{Step 2.}}  Now we show that the candidate value function $w(x,1)$ as in (\ref{3wx1}) satisfies the HJB equation under either of (i), (ii), (iii) of Theorem \ref{theorem4.3}.

$\underline{\textsl{Step 2-(a).}}$ Firstly,  we consider $x \in [\theta_1,\theta_2].$ Since $ 1-w'(x,1) =0$, we should prove that $\frac{1}{2}\sigma_1^2w''(x,1) +\mu_1w'(x,1) -(\lambda_{1}+\rho)w(x,1) +\lambda_{1}w(x,2) \leq0.$ Since $w(x,2)=0, x\in[\theta_1,\theta_2]$ and $\mu_1 \leq 0$, then the inequality holds because $\frac{1}{2}\sigma_1^2w''(x,1) +\mu_1w'(x,1) -(\lambda_{1}+\rho)w(x,1) +\lambda_{1}w(x,2) =\mu_1-(\lambda_{1}+\rho) (x-\theta_1) \leq0.$

$\underline{\textsl{Step 2-(b).}}$ Now we consider $x \in [b_2,b_1).$ Since	$\frac{1}{2}\sigma_1^2w''(x,1) +\mu_1w'(x,1) -(\lambda_{1}+\rho)w(x,1) +\lambda_{1}w(x,2) =0,$	we should prove that $1-w'(x,1)\leq 0.$ From (\ref{3wx1}) we have that $w(x,1)=C_{7}e^{\alpha_1x}+C_{8}e^{\alpha_2x}+\frac{\lambda_{1}\mu_1}{(\rho+\lambda_{1})^2}+\frac{\lambda_{1}(x+K_1)}{\rho+\lambda_{1}}$ and we claim that $C_7>0, C_8<0$. Indeed, from $w'(b_1,1)=w'(b_{1}-,1)$, $w''(b_1,1)=w''(b_{1}-,1)$ in (\ref{system31}), we have $C_{7}\alpha_1e^{\alpha_1b_1}+C_{8}\alpha_2e^{\alpha_2b_1}+\frac{\lambda_{1}}{\rho+\lambda_{1}}=1$ and $C_{7}\alpha_1^2e^{\alpha_1b_1}+C_{8}\alpha_2^2e^{\alpha_2b_1}=0$, from which we can obtain $C_{7}>0$ and $C_{8}<0$ by simple calculations using that $\alpha_1>0$ and $\alpha_2<0$.

Because now $w'''(x,1)=C_{7}\alpha_1^3e^{\alpha_1x}+C_{8}\alpha_2^3e^{\alpha_2x}>0$, $w''(x,1)$ is strictly increasing. Since $w''(b_1,1)=0,$ we have $w''(b_2,1)<0$ and $w''(x,1)\leq 0$ when $x\in [b_2,b_1),$ and hence $w'(x,1)$ is decreasing on that interval. But $w'(b_1,1)=1$, and therefore $w'(x,1)\geq1$ when $x\in [b_2,b_1).$

$\underline{\textsl{Step 2-(c).}}$  We take $x \in [b_1, \infty).$ Since $w'(x,1)=1,$ we should prove that $	\frac{1}{2}\sigma_1^2w''(x,1) +\mu_1w'(x,1) -(\lambda_{1}+\rho)w(x,1) +\lambda_{1}w(x,2) \leq 0.$ Define $F(x):=\frac{1}{2}\sigma_1^2w''(x,1) +\mu_1w'(x,1) -(\lambda_{1}+\rho)w(x,1) +\lambda_{1}w(x,2) ={\mu_1-(\lambda_1+\rho)w(x,1)+\lambda_1w(x,2)}$.  Because $F'(x)=-\rho<0$, {as $w'(x,i)=1 \ \forall x \geq b_1>b_2, i \in \mathcal{S},$} and $F(b_1)=F(b_{1}-)=0,$ it follows that $F(x)\leq 0$ when $x\in [b_1,\infty).$
	
$\underline{\textsl{Step 2-(d).}}$ Now we consider $x \in (d_1,b_2).$  Since $\frac{1}{2}\sigma_1^2w''(x,1) +\mu_1w'(x,1) -(\lambda_{1}+\rho)w(x,1) +\lambda_{1}w(x,2) =0$ by (\ref{4.8}), we should prove that $1-w'(x,1)\leq 0$. Differentiating the previous second-order differential equation two times respect to $x$, we obtain 
\begin{align*}
\frac{1}{2}\sigma_1^2w^{(4)}(x,1) +\mu_1w'''(x,1) -(\lambda_{1}+\rho)w''(x,1)+\lambda_{1}w''(x,2)  =0, \ x \in (d_1,b_2).
\end{align*}
Since $w''(x,2)\leq 0$ when $x \in (d_1,b_2)$ by Step 1,  we have $\frac{1}{2}\sigma_1^2w^{(4)}(x,1) +\mu_1w'''(x,1) -(\lambda_{1}+\rho)w''(x,1) \geq 0.$

From Step 2-(b), we know that $ w''(b_2,1)=w''(b_{2}-,1)<0, w'(b_2,1)\geq 1$. Therefore, by Lemma \ref{zero2}, there exists $N$ points $d_1<x_1<,...,<x_N<b_2$ such that $w''(x_i,1)=0, i \in \{1,...,N\}, N \leq 4.$ Firstly,  we consider  $x \in [x_N,b_2)$.  Since $w''(b_{2}-,1)=w''(b_2,1)<0$ by Step 2-(b),  the weak maximum principle of Lemma \ref{weak} implies that 
\begin{align*}
w''(x,1) \leq \sup \big\{w''(x_N,1)^+,w''(b_{2}-,1)^+\big\} =0,
\end{align*}
where $u(x)^+ := \max\{u(x),0\}.$ Next, iterating the same argument as above in the interval $[x_i,x_{i+1}), i=1,...,N-1,$ we obtain that $w''(x,1)\leq 0 $ when $x \in [x_1,b_2).$ Since $w'(b_{2}-,1)\geq 1$ by Step 2-(b), we have $w'(x,1)\geq1$ when $x \in [x_1,b_2)$. Finally, we consider $x \in (d_1,x_1)$ and claim that $w''(d_{1}+,1)\geq 0$. Indeed,  if  $w''(d_{1}+,1)<0$, since $ w''(x_1,1)=0$ and $x_1$ is the smallest zero of $w''(\cdot,1)$ in the interval $(d_1,b_2)$, it follows that $w''(x,1)<0$ when $x \in (d_1,x_1).$ But now because $w'(d_{1}+,1)=w'(d_{1},1)=1,$  we have $w'(x_1,1)< 1$  which is a contradiction with the fact that $w'(x,1) \geq 1$ for all $x \in [x_1,b_2).$ Therefore, $w''(d_{1}+,1)\geq0.$

Since $w''(d_{1}+,1)\geq 0$ and  $w''(x_1,1)=0$, we find $w''(x,1)\geq 0$ when $x \in (d_1,x_1).$ Because $w'(d_1,1)=w'(d_{1}+,1)=1,$ it follows that $w'(x,1)\geq 1$ for all $x \in (d_1,x_1).$
	
$\underline{\textsl{Step 2-(e).}}$ Finally, we consider $x \in (\theta_2,d_1].$ Since $ 1-w'(x,1) =0$, we should prove that $\frac{1}{2}\sigma_1^2w''(x,1) +\mu_1w'(x,1) -(\lambda_{1}+\rho)w(x,1) +\lambda_{1}w(x,2) \leq0.$ Because the left-hand side of the previous inequality equals $H(x):=\mu_1-(\lambda_{1}+\rho) w(x,1)+\lambda_1 w(x,2)$, our aim is therefore to prove that $H(x) \leq 0$ when $x\in (\theta_2,d_1].$  
	
We notice that $H(\theta_{2})=\mu_1-(\lambda_{1}+\rho)(\theta_{2}-\theta_{1})<0$ and $H''(x)=w''(x,2)\leq 0$ by Step 1. Also, $H(d_{1}) =\mu_1-(\lambda_{1}+\rho) w(d_{1},1)+\lambda_1 w({d_{1}},2)= \mu_1-(\lambda_{1}+\rho) w(d_{1}+,1)+\lambda_1 w({d_{1}+},2),$ where the continuity of $w(x,i)$ at $x=d_1$ has been used. When $x \in (d_1,b_2),$ taking limits as $x \downarrow d_1$, from (\ref{4.8}) we have $	\frac{1}{2}\sigma_1^2w''(d_{1}+,1) +\mu_1w'(d_{1}+,1) -(\lambda_{1}+\rho)w(d_{1}+,1) +\lambda_{1}w(d_{1}+,2) =0.$ Therefore, $H(d_{1})=\mu_1-\mu_1w'(d_{1}+,1)-\frac{1}{2}\sigma_1^2w''(d_{1}+,1).$ Since $w'(d_1,1)=w'(d_{1}+,1)=1$, it follows that $H(d_{1})=-\frac{1}{2}\sigma_1^2w''(d_{1}+,1)\leq 0,$ where we have used the fact that $w''(d_{1}+,1)\geq 0$ as proven in Step 2-(d).

Hence, the fact that $H(d_1)\leq 0, H(\theta_{2})<0$, and $H$ is concave imply that $H(x) \leq 0, \forall x\in [\theta_{2},d_1]$ if and only if the following equivalent conditions are satisfied:
		\begin{align*}
		& H(x)\leq0, \ \forall x\in(\theta_{2},d_1] \Longleftrightarrow \max_{x\in (\theta_{2},d_1]}H(x)\leq0 \\
&\Longleftrightarrow \emph{either}\ H'(\theta_{2}+)\leq 0 \
	\emph{or} \  H'(\theta_{2}+)>0, \exists \  x_0 \in (\theta_2,d_1) \ \text{satisifies} \ H'(x_0)=0, H(x_0)\leq 0 \\
&\emph{or} \  H'(\theta_{2}+)>0 , H'(d_1)\geq 0.\\
		& \Longleftrightarrow \emph{either}\ w'(\theta_{2}+,2)\leq \frac{\lambda_{1}+\rho}{\lambda_{1}} \ \emph{or} \  w'(\theta_{2}+,2)> \frac{\lambda_{1}+\rho}{\lambda_{1}}, 
	\exists \  x_0 \in (\theta_2,d_1) \ \text{satisifies}\ w'(x_0,2) = \frac{\lambda_{1}+\rho}{\lambda_{1}},\\ &H(x_0)\leq 0
\ \emph{or} \ w'(\theta_{2}+,2)> \frac{\lambda_{1}+\rho}{\lambda_{1}} , H'(d_1)\geq 0.
		\end{align*}
Concluding, the results of Steps 2-(a)---2-(e) guarantee that $w(x,1)$ as in (\ref{3wx1}) satisfies the HJB equation under either of (i), (ii), (iii) of Theorem \ref{theorem4.3}.

\textbf{	\textsl{Step 3.}}  Now we show that the candidate value function $w(x,2)$ satisfies the HJB equation as well. Here no additional requirements are needed.
		
$\underline{\textsl{Step 3-(a).}}$	First we consider $x \in (\theta_2,b_2)$. Since	$	\frac{1}{2}\sigma_2^2w''(x,2) +\mu_2w'(x,2) -(\lambda_{2}+\rho)w(x,2) +\lambda_{2}w(x,1) =0$ by (\ref{4.7}) and (\ref{4.9}), we should prove that $1-w'(x,2)\leq 0$. Since $w''(x,2)\leq 0$ by Step 1, we have that $w'(x,2)$ is decreasing. Because $w'(b_{2}-,2)=w'(b_2,2)=1,$ it follows that $w'(x,2)\geq 1$ when $x \in (\theta_{2},b_2).$ 
		
$\underline{\textsl{Step 3-(b).}}$ Now we consider  $x\in [b_2,b_1)$. Since $ 1-w'(x,2) =0$,  we should prove that $	\frac{1}{2}\sigma_2^2w''(x,2) +\mu_2w'(x,2) -(\lambda_{2}+\rho)w(x,2) +\lambda_{2}w(x,1) \leq0.$ Define $F(x):= \frac{1}{2}\sigma_2^2w''(x,2) +\mu_2w'(x,2) -(\lambda_{2}+\rho)w(x,2) +\lambda_{2}w(x,1)=\mu_{2}-(\lambda_{2}+\rho)w(x,2) +\lambda_{2}w(x,1).$    Then we compute $F''(x)=\lambda_{2}w''(x,1)\leq0$, where we have used the fact that $w''(x,1)\leq0$ when $x\in[b_2,b_1)$ by Step 2-(b). We want to prove that $F(x)\leq0, x\in[b_2, b_1).$ Since $F(b_2)=F(b_{2}-)=0,$ it is enough to show that $F(x)$ is decreasing; i.e., $F'(x) \leq 0.$ Since $F''(x)\leq 0,$ we have  $F'(x)\leq F'(b_2), \forall x\in[b_2,b_1)$. Therefore, it is enough to prove that $	F'(b_2)= -(\lambda_{2}+\rho)+\lambda_{2}w'(b_2,1)        \leq0$.
	 
	In order to show the previous inequality, we next investigate $\lambda_{2}w'(b_2,1)$. When $x\in (d_1,b_2)$, from (\ref{4.9}) we have $	\frac{1}{2}\sigma_2^2w''(x,2) +\mu_2w'(x,2) -(\lambda_{2}+\rho)w(x,2) +\lambda_{2}w(x,1) =0.$ Differentiating the equation one time respect to $x$ and taking limits as $x \uparrow b_2$,  we obtain
\begin{align*}
\lambda_{2}w'(b_2,1) = \lambda_{2}w'(b_{2}-,1)= -\frac{1}{2}\sigma_2^2w'''(b_{2}-,2)-\mu_2w''(b_{2}-,2)+(\lambda_{2}+\rho)w'(b_{2}-,2) \leq \lambda_{2}+\rho,
\end{align*}
	where the last step follows from the fact that $w''(b_{2}-,2)=w''(b_2,2)=0, w'(b_{2}-,2)=w'(b_2,2)=1$ and $w'''(b_{2}-,2)\geq0.$ The fact that $w'''(b_2-,2)\geq 0$ can be indeed shown as follows. Assume, in the opposite, $w'''(b_{2}-,2)<0.$ Then since $w''(b_{2},2)=0,$  we have $w''(b_{2}-\epsilon,2)>0$ for $\epsilon$ small enough, which contradicts  $w''(x,2)\leq 0$ when $x\in [d_1,b_2)$ as proved in Step 1. Therefore, $F(x)\leq0$ when $x\in[b_2,b_1).$
	
$\underline{\textsl{Step 3-(c).}}$	   Finally, we consider $x\in [b_1,\infty)$. Similar to Step 3-(b), we should prove that $F(x) \leq 0$ for all $x\in[b_1, \infty).$ Since $w(x,1)=x+K_2,$ one has $F'(x)= -\rho<0$.  However $F(b_1)= F(b_{1}-)\leq 0,$ and therefore $F(x)\leq0 $ when $x\in [b_1,\infty).$

\textbf{	\textsl{Step 4.}} 
We notice that the liquidation-barrier dividend strategy $D^{d_1,b,w}$ associated with $w(x,i)$ in (\ref{3wx1}) and (\ref{3wx2}) fulfills the conditions of Definition \ref{dividend1}. Therefore, $w(x,i),i=1,2,$ in (\ref{3wx1}) and (\ref{3wx2}) is indeed the value function by Theorem \ref{verificationtheorem}.	
\end{proof}

\subsection{Proof of Theorem \ref{theorem4.5}}\label{ptheorem4.5}

 Throughout this proof we denote by $w^{(n)}$ the $n$-th derivative of the function $w, n\geq 4$, with respect to $x$. Before the proof, we need the following lemmata.

\begin{lemma}\label{regularity2}
	Set $I_i:=(\theta_i,\infty), i=1,2.$ Then the function $w(x,i)$ defined in (\ref{4wx1}) and (\ref{4wx2}) is such that $w(\cdot,1) \in C^1(I_1) \ \cap  \ C^2(I_1 \backslash \{d_1\})\ \cap\ C^{\infty}(I_1 \backslash \{d_1,\theta_{2},b_2,b_1\})$ and $w(\cdot,2) \in  C^2(I_2)  \cap C^4(I_2 \backslash \{b_2\}) \cap C^{\infty}(I_2 \backslash \{b_2\}).$ Moreover, for fixed $i \in \{1,2\}$ and $\theta_{2} <x <b_2$ , we have
	\begin{align*}
	\frac{1}{2}\sigma_i^2w^{(4)}(x,i) +\mu_iw'''(x,i) -(\lambda_{i}+\rho)w''(x,i) +\lambda_{i}w(x,3-i)=0.
	\end{align*}
\end{lemma}
\begin{proof}
	The result can be shown by arguing as in Lemma \ref{regularity1}.
\end{proof}

\begin{lemma}\label{zero1}
	Suppose that $w''(\theta_2,1)>0, w''(b_2,1)<0$.  Define the equation $f(x):= w''(x,1) = C_3\alpha_3^2e^{\alpha_3x}+C_4\alpha_4^2e^{\alpha_4x}+C_5\alpha_5^2e^{\alpha_5x}+C_6\alpha_6^2e^{\alpha_6x}, x \in (\theta_2,b_2)$. Then $f$ admits a real zero and, in particular, at most 4 zeros when $x\in(\theta_2,b_2).$
\end{lemma}
\begin{proof} First, we show that there exists zeros. By assumption,  $w''(\theta_2,1)>0, w''(b_2,1)<0$. Since $w''(x,1)$ is continuous when $x \in (\theta_2,b_2),$ it thus follows that there exists zeros when $x\in(\theta_2,b_2).$ 
	
	 The rest of the proof now follows as in Lemma \ref{zero2}.
\end{proof}

\begin{proof}[\textbf{Proof of Theorem \ref{theorem4.5}}]
{By construction, the boundary conditions are satisfied.} The rest of the proof is organized in five steps.

\setlength{\parskip}{0.5em}
\textbf{	\textsl{Step 1.}}  We show that $w(x,1)$ is increasing when $x\in (\theta_{1},\theta_{2})$. By construction, we know that $w'(x,1)=1>0, x\in (\theta_{1},d_1].$  Therefore, we only need to prove that $w'(x,1)\geq 0$ for all $x\in (d_{1},\theta_{2}).$ 

Suppose by contradiction that there exists an interval $\mathcal{I}:=(h_2,h_1) \in (d_{1},\theta_{2})$ such that $w'(x,1)<0,x\in \mathcal{I},$ where $h_2:=\inf\{ x\in (d_1,\theta_2): w'(x,1)< 0 \}$. Now if $h_2=d_1,$ since $w'(d_1,1)=1$ and $w'(x,1)<0$ for $x\in (d_1,h_1)$, we would get a contraction with the fact that $w'(x,1)$ is continuous in $d_1$. If $h_2 >d_1,$ we can pick $x\in (d_1,h_2).$ Since $w(x,2)=0$ for $x\in (d_1,\theta_{2}],$ from (\ref{4.19}) we have
\begin{equation}\label{a16}
	\frac{1}{2}\sigma_1^2w''(x,1) +\mu_1w'(x,1) -(\lambda_{1}+\rho)w(x,1)  =0.
	\end{equation}
Due to $w'(x,1)\geq 0$ for $x\in (d_1,h_2)$ (by definition of $h_2$) and $ w(d_1,1)>0$ (because $w(\theta_{1},1)=0$ and $w'(x,1)=1$ for $x\in(\theta_{1},d_1]$), we have $ w(x,1)>0 $ on $(d_1,h_2)$. But now $\mu_{1}\leq 0$ together with (\ref{a16}) imply $w''(x,1)>0$ when $x \in (d_1,h_{2})$. Therefore, $w'(x,1)$ is strictly increasing when $x\in(d_1,h_{2}).$ Because $w'(d_1,1)=w'(d_{1}+,1)=1,$ we have $w'(x,1)\geq1$ when $x\in (d_1,h_{2}).$ On the other hand, $w'(x,1)\leq 0, x\in [h_2,h_1)$, thus leading to a contraction with $w'(x,1) $ being continuous in $h_2 \in (d_{1},\theta_{2})$. Therefore, there is no interval $\mathcal{I} \in (d_1,\theta_{2})$ such that $w'(x,1)< 0,$  and  $w(x,1)$ is increasing for all $x\in (\theta_{1},\theta_{2}).$

\textbf{	\textsl{Step 2.}}  Now we show that if $\mu_2\geq 0$ and $w'(\theta_{2}+,2)\geq 0$, then $w(x,2)$ is concave. In particular, $ w''(x,2)\leq 0$  when $x \in (\theta_2,b_2).$  

 By Lemma \ref{regularity2}, we have  $w(\cdot,2) \in C^4(I_2 \backslash \{ b_2\})$.  Applying It\^o's formula to $ (e^{-\rho t}w''(X^D_{t},2))_{t\in [0,\xi]}$, where $\xi:=\xi_1 \wedge \xi_2$ with $\xi_1:=\inf\{t\geq 0:X_t^D \geq b_{2}\} $ and $ \xi_2:= \inf \{t\geq 0:X^D_t\le \theta_2\}$, we have 
\begin{align*}
	\mathbb{E}\big[e^{-\rho \xi}w''(X_\xi^D,2)\big]&=w''(x,2)+\mathbb{E}\bigg[\int_{0}^{\xi}e^{-\rho t}(\frac{1}{2}\sigma_2^2w^{(4)}(X^D_t,2)+ \mu_{2}w'''(X^D_t,2)\\&-(\rho+\lambda_{2}) w''(X^D_t,2)+ \lambda_{2}w''(X^D_t,1))\bigg]\textrm{d}t.
	\end{align*}	
Since $(\mathcal{L}-\rho)w''(X^D_t,2)=0$ for all $t \leq \xi$, it follows that
	\begin{align}\label{a16-2}
	\mathbb{E}\big[e^{-\rho \xi_1}w''(X^D_{\xi_1},2)\mathbb{I}_{\{\xi_1<\xi_2\}}\big]+\mathbb{E}\big[e^{-\rho \xi_2}w''(X^D_{\xi_2},2)\mathbb{I}_{\{\xi_1 \geq \xi_2\}}\big]= w''(x,2).
	\end{align}
	Because $w''(x,2)=0$ for all $x \in [b_2, \infty)$, we have $\mathbb{E}\big[e^{-\rho \xi_1}w''(X^D_{\xi_1},2)\mathbb{I}_{\{\xi_1<\xi_2\}}\big]=0.$ Thus, we only need to compute $w''(\theta_{2}+,2).$
	When $x\in(\theta_2,b_2)$, from (\ref{4.21}) we have
\begin{equation}\label{3.5-1}
	\frac{1}{2}\sigma_2^2w''(x,2) +\mu_2w'(x,2) -(\lambda_{2}+\rho)w(x,2) +\lambda_{2}w(x,1) =0.
	\end{equation}
Therefore, taking limits, $\lambda_{2}w(\theta_{2}+,1)=-\frac{1}{2}\sigma_2^2w''(\theta_{2}+,2)-\mu_2w'(\theta_{2}+,2)+(\lambda_{2}+\rho)w(\theta_{2}+,2).$ Since $w(\theta_{1},1)=0$ and $w(x,1)$ is increasing by Step 1, we have $w(\theta_{2},1) = w(\theta_{2}+,1) > 0.$ Also, since $w(\theta_{2},2)=0,$ from (\ref{3.5-1}) we have $	-\frac{1}{2}\sigma_2^2w''(\theta_{2}+,2)>\mu_2w'(\theta_{2}+,2).$ Therefore, if $\mu_2 \geq 0$ and $w'(\theta_{2}+,2) \geq 0,$ then $w''(\theta_{2}+,2) \leq 0.$ As a consequence, from (\ref{a16-2}) $w''(x,2) \leq 0$ for $x\in (\theta_{2},b_2)$ and $w(x,2)$ is concave.

\textbf{	\textsl{Step 3.}}  We here show that the candidate value function $w(x,1)$ as in (\ref{4wx1}) satisfies the HJB equation. 

$\underline{\textsl{Step 3-(a).}}$	 Firstly, we consider $x \in (\theta_1,d_1].$ Since $ 1-w'(x,1) =0$, we should prove that $	\frac{1}{2}\sigma_1^2w''(x,1) +\mu_1w'(x,1) -(\lambda_{1}+\rho)w(x,1) +\lambda_{1}w(x,2) \leq0.$ Since $w(x,2)=0, x\in(\theta_1,d_1]$ and $\mu_1\leq 0$, we can obtain that $	\frac{1}{2}\sigma_1^2w''(x,1) +\mu_1w'(x,1) -(\lambda_{1}+\rho)w(x,1) +\lambda_{1}w(x,2) =\mu_1-(\lambda_{1}+\rho) (x-\theta_1) \leq0.$

$\underline{\textsl{Step 3-(b).}}$ Now we consider $x \in (d_1,\theta_2]$. Since $w(x,2)=0,  x\in (d_1,\theta_{2}],$ and because from (\ref{4.19}) we have
	\begin{equation}\label{3.5-2}
	\frac{1}{2}\sigma_1^2w''(x,1) +\mu_1w'(x,1) -(\lambda_{1}+\rho)w(x,1)  =0,
	\end{equation}
	we should prove that $1-w'(x,1)\leq 0$. From Step 1, we know that  $w(x,1)>0, w'(x,1) \geq 0 $ for all $x\in (d_1,\theta_2].$ Since $\mu_1\leq 0,$  from (\ref{3.5-2}) we have  $w''(x,1)>0$ when $x \in (d_1,\theta_{2}]$, thus $w'(x,1)$ is strictly increasing when $x\in(d_1,\theta_{2}].$  But it holds that $w'(d_1,1)=w'(d_{1}+,1)=1,$  which in turn yields that $w'(x,1)\geq1$ when $x\in (d_1,\theta_{2}].$

$\underline{\textsl{Step 3-(c).}}$ Next we consider $x \in [b_2,b_1).$ Since $	\frac{1}{2}\sigma_1^2w''(x,1) +\mu_1w'(x,1) -(\lambda_{1}+\rho)w(x,1) +\lambda_{1}w(x,2) =0$ by (\ref{4.23}), we should prove that $1-w'(x,1)\leq 0.$ From (\ref{4wx1}) we have $w(x,1)=C_{7}e^{\alpha_1x}+C_{8}e^{\alpha_2x}+\frac{\lambda_{1}\mu_1}{(\rho+\lambda_{1})^2}+\frac{\lambda_{1}(x+K_1)}{\rho+\lambda_{1}}$ and we claim that $C_7>0, C_8<0$. Indeed, from $w'(b_1,1)=w'(b_{1}-,1), w''(b_1,1)=w''(b_{1}-,1)$ in (\ref{system41}), we have $C_{7}\alpha_1e^{\alpha_1b_1}+C_{8}\alpha_2e^{\alpha_2b_1}+\frac{\lambda_{1}}{\rho+\lambda_{1}}=1$ and
$C_{7}\alpha_1^2e^{\alpha_1b_1}+C_{8}\alpha_2^2e^{\alpha_2b_1}=0$, from which we can obtain $C_{7}>0$ and $C_{8}<0$ by simple calculations using that $\alpha_1>0, \alpha_2<0$.

Because $w'''(x,1)=C_{7}\alpha_1^3e^{\alpha_1x}+C_{8}\alpha_2^3e^{\alpha_2x}>0$,  $w''(x,1)$ is strictly increasing. Since $w''(b_1,1)=0,$ we have $w''(b_2,1)<0$ and $w''(x,1)\leq 0$ when $x\in [b_2,b_1),$ and hence $w'(x,1)$ is decreasing for all $x\in [b_2,b_1) $. But $w'(b_1,1)=1$, and therefore $w'(x,1)\geq1$ when $x\in [b_2,b_1).$

$\underline{\textsl{Step 3-(d).}}$ Pick now $x \in [b_1, \infty).$ Since $w'(x,1)=1,$ we should prove that $	\frac{1}{2}\sigma_1^2w''(x,1) +\mu_1w'(x,1) -(\lambda_{1}+\rho)w(x,1) +\lambda_{1}w(x,2) \leq 0.$ Define $F(x):=\frac{1}{2}\sigma_1^2w''(x,1) +\mu_1w'(x,1) -(\lambda_{1}+\rho)w(x,1) +\lambda_{1}w(x,2) \leq 0$.  Because $F'(x)=-\rho<0$ and $F(b_1)=F(b_{1}-)=0,$ it follows that $F(x)\leq 0$ when $x\in [b_1,\infty).$

$\underline{\textsl{Step 3-(e).}}$ Finally, we consider $x \in (\theta_2,b_2)$. From (\ref{4.20}), we have $	\frac{1}{2}\sigma_1^2w''(x,1) +\mu_1w'(x,1) -(\lambda_{1}+\rho)w(x,1) +\lambda_{1}w(x,2) =0$, and we therefore should prove that $1-w'(x,1)\leq 0$.  Thanks to Lemma \ref{regularity2} we can differentiate (\ref{4.20}) two times respect to $x$ and obtain $\frac{1}{2}\sigma_1^2w^{(4)}(x,1) +\mu_1w'''(x,1) -(\lambda_{1}+\rho)w''(x,1)+\lambda_{1}w''(x,2)  =0$ for all $x \in (\theta_{2},b_2)$. Since $w''(x,2)\leq0$ when $x \in (\theta_{2},b_2)$ by Step 2, we find $\frac{1}{2}\sigma_1^2w^{(4)}(x,1) +\mu_1w'''(x,1) -(\lambda_{1}+\rho)w''(x,1) \geq0.$

From Step 3-(b), we have $w''(\theta_{2},1)>0$, and $w''(b_2,1)>0$ by Step 3-(c). Thus, by Lemma \ref{zero1}, there exists $N$ points $\theta_{2}<x_1<,...,<x_N<b_2$ such that $w''(x_i,1)=0, i \in \{1,...,N\}, N \leq 4$. Firstly, consider  $x \in [x_N,b_2).$  Since $w''(b_{2}-,1)=w''(b_2,1)<0$ by Step 3-(c), the weak maximum principle of Lemma \ref{weak} implies that 
	\begin{align*}
	w''(x,1) \leq \sup\{w''(x_N,1)^+,w''(b_2,1)^+\} =0,
	\end{align*}
	where $u(x)^+ := \max\{u(x),0\}.$ Next, iterating the same argument as above in the interval $[x_i,x_{i+1}), i=1,...,N-1,$ we obtain that $w''(x,1)\leq 0 $ when $x \in [x_1,b_2).$ Since $w'(b_{2}-,1)\geq1$ by Step 3-(c), we have $w'(x,1)\geq1$ when $x \in [x_1,b_2)$. Finally, we consider $x \in (\theta_{2},x_1).$  Since $w''(\theta_{2},1)=w''(\theta_{2}+,1)>0$ by Step 3-(b) and $ w''(x_1,1)=0$, we have $w''(x,1)>0$ when $x \in (\theta_{2},x_1).$ Because $w'(\theta_{2},1)=w'(\theta_{2}+,1)\geq 1$ by Step 3-(b), it follows that $w'(x,1)\geq1$ when $x \in (\theta_{2},x_1).$ Therefore, we have proved that $w'(x,1)\geq1$ when $x \in (\theta_{2},b_2).$

\textbf{	\textsl{Step 4.}}  In order to complete the proof it remains to check that $w(x,2)$ satisfies the HJB equation as well. This can be shown by arguing as in Step 3 of the proof of Theorem \ref{theorem4.3} and we therefore omit details.

\textbf{	\textsl{Step 5.}} 
We notice that the liquidation-barrier dividend strategy associated with $w(x,i)$ in (\ref{4wx1}) and (\ref{4wx2}) fulfills the conditions of $D^{d_1,b,w}$ in Definition \ref{dividend1}. Therefore, $w(x,i),i=1,2,$ in (\ref{4wx1}) and (\ref{4wx2}) is indeed the value function by Theorem \ref{verificationtheorem}.

\end{proof}

\section{Some auxiliary results}
\renewcommand\theequation{B.\arabic{equation}}

\begin{lemma}\label{ODE}
\begin{enumerate}[(i)]
	\item The general solution to the equation   \begin{equation*}
	 \frac{1}{2}\sigma^2g''(x) +\mu g'(x) -(\lambda+\rho)g(x)=0
	\end{equation*}
	 is given by $g(x) = C_1 e^{\alpha_1x}+ C_2 e^{\alpha_2x},$ where $C_1,C_2 \in \mathbb{R}$ are constants and the real numbers $\alpha_1,\alpha_2$ are given by
	\begin{equation}\label{alpha1}
	\alpha_1 = \frac{1}{\sigma^2}(-\mu +\sqrt{\mu^2+2\sigma^2(\lambda+\rho)})>0,
	\end{equation}
	and
	\begin{equation}\label{alpha2}
	\alpha_2 = \frac{1}{\sigma^2}(-\mu -\sqrt{\mu^2+2\sigma^2(\lambda+\rho)})<0.
	\end{equation}
	
	 \item Given a constant $K$. The general solution to the equation \begin{equation*}
	\frac{1}{2}\sigma^2h''(x) +\mu h'(x) -(\lambda+\rho)h(x)+\lambda(x+K)=0
	\end{equation*}
	 is given by $h(x) = C_3 e^{\alpha_1x}+ C_4 e^{\alpha_2x}+ \frac{\mu \lambda}{(\lambda+\rho)^2}+\frac{\lambda(x+K)}{\lambda+\rho}$, where $C_3,C_4 \in \mathbb{R}$ are constants and the real numbers $\alpha_1,\alpha_2$ are given in (\ref{alpha1}) and (\ref{alpha2}).

	\item The general solutions to the system of differential equations
\begin{equation*}
 \left\{
\begin{aligned}
\frac{1}{2}\sigma_1^2h''(x) +\mu_1 h'(x) -(\lambda_1+\rho)h(x)+\lambda_1 g(x)=0 \\
\frac{1}{2}\sigma_2^2g''(x) +\mu_2 g'(x) -(\lambda_2+\rho)g(x)+\lambda_2 h(x)=0
\end{aligned}
\right.
\end{equation*}
 are given by 
	\begin{equation*}
	h(x) = A_1 e^{\alpha_1x}+ A_2 e^{\alpha_2x}+ A_3 e^{\alpha_3x}+ A_4 e^{\alpha_4x}
	\end{equation*}
	\begin{equation*}
	g(x) = B_1 e^{\alpha_1x}+ B_2 e^{\alpha_2x}+ B_{3} e^{\alpha_3x}+ B_{4} e^{\alpha_4x}
	\end{equation*}
	where, for each $j =1,2,3,4,$ $B_j = \frac{\phi_1(\alpha_j)}{\lambda_{1}}A_j = \frac{\lambda_{2}}{\phi_2(a_j)}A_j.$ The real numbers $\alpha_1< \alpha_2<0<\alpha_3<\alpha_4$ are the real roots of the equation $\phi_1(\alpha)\phi_2(\alpha) = \lambda_{1}\lambda_{2}$, where $\phi_i = -\frac{1}{2}\sigma_i^2\alpha^2 -\mu_i \alpha +(\lambda_i+\rho), i =1,2.$ 
	\end{enumerate}
\end{lemma}

\begin{proof}
We skip the classical proofs of (i) and (ii). For the proof of (iii) we refer the reader to Lemma A.1 in \cite{ferrari2018optimal}. 
\end{proof}

\begin{lemma}(Weak maximum principle for $c\leq 0$)\label{weak}
	Let $\Omega$ be an open connected set in $\mathbb{R}^n$ with boundary $\partial \Omega = \bar{\Omega} \cap (\mathbb{R}^n \setminus \Omega) $. Let $L$ be the second order differential operator
	\begin{equation*}
	L = \sum_{i,j=1}^{n}a_{ij}(x)D_{ij}+\sum_{i=1}^{n}b_i(x)D_i+c(x)
	\end{equation*}
	with $a_{ij}\in L^\infty_{\text{loc}}(\Omega)$ and $b_i, c \in L^{\infty}(\Omega).$
	
	Assume $u \in C^2(U)\cap C(\bar{U})$ and $c \leq 0 $ in $U.$
	If $Lu \geq 0$ in $U$, then $\sup_{\bar{U}} u \leq \sup_{\partial\Omega}u^+.$ If $Lu \leq 0$ in $U$, then $\inf_{\bar{U}} u \geq
	\inf_{\partial\Omega}u^-.$
\end{lemma}
\begin{proof}
We refer the reader to Chapter 3 in \cite{gilbarg2001elliptic}.
\end{proof}

\section*{Acknowledgments}
The authors gratefully acknowledge funding by the Deutsche Forschungsgemeinschaft (DFG, German Research Foundation) – SFB 1283/2 2021 – 317210226. The work of Shihao Zhu was also supported by the program of China Scholarships Council.

\bibliographystyle{siam}

\bibliography{regime}

\end{document}